\documentclass[11pt]{amsart}
\usepackage{amssymb}
\usepackage{verbatim}
\usepackage{hyperref}
\usepackage{color}
\usepackage{enumerate}

\begin{document}

\newtheorem{theorem}{Theorem}
\newtheorem{proposition}[theorem]{Proposition}
\newtheorem{conjecture}[theorem]{Conjecture}
\def\theconjecture{\unskip}
\newtheorem{corollary}[theorem]{Corollary}
\newtheorem{lemma}[theorem]{Lemma}
\newtheorem{sublemma}[theorem]{Sublemma}
\newtheorem{fact}[theorem]{Fact}
\newtheorem{observation}[theorem]{Observation}
\theoremstyle{definition}
\newtheorem{definition}{Definition}
\newtheorem{notation}[definition]{Notation}
\newtheorem{remark}[definition]{Remark}
\newtheorem{question}[definition]{Question}
\newtheorem{questions}[definition]{Questions}
\newtheorem{claim}[definition]{Claim}
\newtheorem{example}[definition]{Example}
\newtheorem{problem}[definition]{Problem}
\newtheorem{exercise}[definition]{Exercise}

 \numberwithin{theorem}{section}
 \numberwithin{definition}{section}
 \numberwithin{equation}{section}

\def\reals{{\mathbb R}}
\def\torus{{\mathbb T}}
\def\heis{{\mathbb H}}
\def\integers{{\mathbb Z}}
\def\rationals{{\mathbb Q}}
\def\naturals{{\mathbb N}}
\def\complex{{\mathbb C}\/}
\def\distance{\operatorname{distance}\,}
\def\support{\operatorname{support}\,}
\def\dist{\operatorname{dist}}
\def\Dist{\operatorname{Distance}}
\def\Span{\operatorname{span}\,}
\def\degree{\operatorname{degree}\,}
\def\kernel{\operatorname{kernel}\,}
\def\dim{\operatorname{dim}\,}
\def\codim{\operatorname{codim}}
\def\trace{\operatorname{trace\,}}
\def\Span{\operatorname{span}\,}
\def\dimension{\operatorname{dimension}\,}
\def\codimension{\operatorname{codimension}\,}
\def\nullspace{\scriptk}
\def\kernel{\operatorname{Kernel}}
\def\ZZ{ {\mathbb Z} }
\def\p{\partial}
\def\rp{{ ^{-1} }}
\def\Re{\operatorname{Re\,} }
\def\Im{\operatorname{Im\,} }
\def\ov{\overline}
\def\eps{\varepsilon}
\def\lt{L^2}
\def\diver{\operatorname{div}}
\def\curl{\operatorname{curl}}
\def\etta{\eta}
\newcommand{\norm}[1]{ \|  #1 \|}
\def\expect{\mathbb E}
\def\bull{$\bullet$\ }
\def\det{\operatorname{det}}
\def\Det{\operatorname{Det}}
\def\multiR{\mathbf R}
\def\bestA{\mathbf A}
\def\Apq{\mathbf A_{p,q}}
\def\Apqr{\mathbf A_{p,q,r}}
\def\bestC{\mathbf C}
\def\bestAqd{\mathbf A_{q,d}}
\def\bestBqd{\mathbf B_{q,d}}
\def\bestB{\mathbf B}
\def\Apq{\mathbf A_{p,q}}
\def\Apqr{\mathbf A_{p,q,r}}
\def\rank{\mathbf r}
\def\diameter{\operatorname{diameter}}
\def\bp{\mathbf p}
\def\bff{\mathbf f}
\def\bx{\mathbf x}
\def\by{\mathbf y}
\def\bz{\mathbf z}
\def\bu{\mathbf u}
\def\bg{\mathbf g}
\def\essd{\operatorname{essential\ diameter}}

\def\symdif{\,\Delta\,}
\def\frakE{{\mathfrak E}}
\def\frakI{{\mathfrak I}}
\def\defe{\dist(E,\frakE)}
\def\defi{\dist(E,\frakI)}

\def\mab{\max(|A|,|B|)}
\def\t2{\tfrac12}
\def\tatb{tA+(1-t)B}
\def\divergence{\operatorname{div}}
\def\steiner{\starstar}
\def\tarstar{{\star\star}}

\newcommand{\abr}[1]{ \langle  #1 \rangle}

\newcommand{\Norm}[1]{ \Big\|  #1 \Big\| }
\newcommand{\set}[1]{ \left\{ #1 \right\} }
\def\one{{\mathbf 1}}
\newcommand{\modulo}[2]{[#1]_{#2}}

\def\scriptf{{\mathcal F}}
\def\scriptq{{\mathcal Q}}
\def\scriptg{{\mathcal G}}
\def\scriptm{{\mathcal M}}
\def\scriptb{{\mathcal B}}
\def\sB{{\mathcal B}}
\def\scriptc{{\mathcal C}}
\def\scriptt{{\mathcal T}}
\def\scripti{{\mathcal I}}
\def\scripte{{\mathcal E}}
\def\scriptv{{\mathcal V}}
\def\scriptw{{\mathcal W}}
\def\scriptu{{\mathcal U}}
\def\scriptS{{\mathcal S}}
\def\scripta{{\mathcal A}}
\def\scriptr{{\mathcal R}}
\def\scripto{{\mathcal O}}
\def\scripth{{\mathcal H}}
\def\scriptd{{\mathcal D}}
\def\ovscriptd{\overline{\mathcal D}}
\def\scriptl{{\mathcal L}}
\def\scriptn{{\mathcal N}}
\def\scriptp{{\mathcal P}}
\def\scriptk{{\mathcal K}}
\def\scriptP{{\mathcal P}}
\def\scriptj{{\mathcal J}}
\def\scriptz{{\mathcal Z}}
\def\scripts{{\mathcal S}}
\def\frakv{{\mathfrak V}}
\def\frakG{{\mathfrak G}}
\def\aff{\operatorname{Aff}}
\def\frakB{{\mathfrak B}}
\def\frakC{{\mathfrak C}}
\def\aA{{\mathfrak A}}

\def\aff{\operatorname{Aff}}
\def\bestC{{\mathbf C}}

\def\bE{{\mathbf E}}
\def\boldC{{\mathbf C}}
\def\Ceta{{\mathbf C}_\eta}
\def\bB{{\mathbf B}}
\def\bF{{\mathbf F}}
\def\bG{{\mathbf G}}
\def\Psharp{P^\sharp}
\def\bA{{\mathbf A}}
\def\bI{{\mathbf I}}
\def\bEstar{{\mathbf E}^\star}
\def\bAstar{{\mathbf A}^\star}
\def\be{{\mathbf e}}
\def\bv{{\mathbf v}}
\def\bw{{\mathbf w}}
\def\br{{\mathbf r}}
\def\Star{\star}
\def\doublestar{\dagger\ddagger}
\def\unitQ{{\mathbf Q}}
\def\Gl{\operatorname{Gl}}

\def\Gjo{G_{\text{$j$,o}}}
\def\Gje{G_{\text{$j$,e}}}

\def\ball{\mathbb B}

\def\barrier{\medskip\hrule\hrule\medskip}

\def\bb{\mathbb B}
\def\br{\mathbf{r}}
\def\bt{\mathbf{t}}
\def\sstar{{\dagger\star}}
\def\Star{{\bullet}}
\def\aff{\operatorname{Aff}}
\def\gl{\operatorname{Gl}}
\def\baff{\operatorname{\bf Aff}}
\newcommand{\mbf}[1]{\mathbf #1}
\def\T{{\mathcal T}}

\def\defb{|E\symdif \bb|}

\def\orbit{\scripto}
\def\repair{\medskip \hrule \hrule \medskip}

\def\astar{{A^\star}}
\def\bstar{{B^\star}}
\def\cstar{{C^\star}}
\def\bC{{\mathbf C}}
\def\freiman{Fre{\u\i}man}

\def\defb{|E\symdif \bb|}
\def\sym{\operatorname{Sym}}
\def\freiman{Fre{\u\i}man}

\author{Michael Christ}

\address{
        Michael Christ\\
        Department of Mathematics\\
        University of California \\
        Berkeley, CA 94720-3840, USA}
\email{mchrist@berkeley.edu}

\author{Marina Iliopoulou}
\address{
        Marina Iliopoulou\\
        School of Mathematics, Statistics and Actuarial Science\\
        University of Kent\\
        Canterbury, CT2 7PE, UK}
\email{m.iliopoulou@kent.ac.uk}
\thanks{The first author was supported in part by NSF grants DMS-1363324 
and DMS-1901413.
Both authors were supported in part by the Mathematical Sciences Research Institute,
which in turn was also supported in part by the National Science Foundation under Grant No.~1440140. Much of this work was carried out while the second author held a postdoctoral position at UC Berkeley.}

\date{August 12, 2019}

\title[Inequalities of Riesz-Sobolev type]
{Inequalities of Riesz-Sobolev type \\ for compact connected Abelian groups}

\begin{abstract}
A version of the Riesz-Sobolev convolution inequality
is formulated and proved for arbitrary compact connected Abelian groups.
Maximizers are characterized and a quantitative stability theorem
is proved, under natural hypotheses. A corresponding stability theorem
for sets whose sumset has nearly minimal measure is also proved,
sharpening recent results of other authors. For the special case
of the group $\reals/\integers$,
a continuous deformation of sets is developed, under which
an appropriately scaled Riesz-Sobolev functional is shown to
be nondecreasing.
\end{abstract}

\maketitle

\tableofcontents

\section{Introduction} \label{section:intro}

Let $G$ be a compact connected Abelian topological group,
equipped with Haar measure $\mu$. 
Throughout this paper, 
the measure $\mu$ is assumed to be complete. 
We say that $\mu$ is normalized to mean that $\mu(G)=1$.
By a measurable subset of $G$ we will always mean a $\mu$--measurable subset.
$\mu_*$ denotes the associated inner measure.
Let $\torus=\reals/\integers$, equipped with Lebesgue
measure $m$, with $m(\torus)=1$. 

Our first result is 
a Riesz-Sobolev--type inequality for $G$,  of which the 
following is one of several formulations.
Assuming $\mu$ to be normalized,
to any measurable set $A\subset G$ is associated the set $\astar\subset\torus$,
which is defined to be the closed interval centered at $0$ satisfying $m(\astar) = \mu(A)$. 
Convolution on $G$ is defined by $f*g(x) = \int_G f(x-y)g(y)\,d\mu(y)$.
$\one_A$ denotes the indicator function of $A$.

\begin{theorem} \label{thm:RS}
Let $G$ be a compact connected Abelian topological group,
equipped with normalized Haar measure. 
For any measurable subsets $A,B,C\subset G$,
\begin{equation} \label{ineq:RS}
\int_C \one_A*\one_B\,d\mu 
\le \int_{\cstar} \one_{\astar}*\one_{\bstar}\,dm. 
\end{equation}
\end{theorem}

Kneser's inequality \cite{kneser} 
\begin{equation} \label{ineq:kneser}
\mu_*(A+B) \ge \min(\mu(A)+\mu(B),\mu(G))
\end{equation}
may be viewed as a limiting case of \eqref{ineq:RS}. 
A mildly stronger formulation 
is\footnote{\eqref{ineq:kneserzero} follows from \eqref{ineq:kneser}
for $G=\torus^d$  by a simple argument involving points of density,
since $A+B=A+_0B$ if every point of each of $A,B$ is a point of density.
For general groups $G$, \eqref{ineq:kneserzero} follows from the
special case of $\torus^d$ by approximating by elements of the
algebra generated by Bohr sets. Alternatively,
a stronger form of \eqref{ineq:kneserzero} is proved in \cite{taokneser}.}
\begin{equation} \label{ineq:kneserzero}
\mu(A+_0 B) \ge \min(\mu(A)+\mu(B),\mu(G))
\end{equation}
where  $A+_0 B$ is the open set
\[ A+_0 B:=\{x: \one_A*\one_B(x)>0\}.\]
Indeed, $\mu_*(A+B)\ge \mu_*(A+_0 B) = \mu(A+_0 B)$.

Our main theme is the quantitative characterization
of triples $(A,B,C)$ that maximize, or nearly maximize,
the functional $\int_C \one_A*\one_B\,d\mu$
among all sets of specified Haar measures.
However, the inequality \eqref{ineq:RS} seems to have attracted little attention
in the setting of compact groups, so some of its aspects relevant to this characterization
are also developed here. 

In the parameter range of primary interest,
\eqref{ineq:RS} can be restated with an alternative expression for the right-hand side.
\begin{theorem} \label{thm:RSprime}
For any compact connected Abelian topological group $G$
and any measurable subsets $A,B,C\subset G$
satisfying 
\begin{equation}
\label{RSmeasurerestriction}
\left\{
\begin{gathered}
|\mu(A)-\mu(B)|\le \mu(C)\le \mu(A)+\mu(B),
\\
\mu(A)+\mu(B)+\mu(C)\le 2,
\end{gathered} \right. \end{equation}
one has
\begin{equation} \label{ineq:RSprime}
\int_C \one_A*\one_B\,d\mu 
\le 
\tfrac12 (ab + bc + ca) - \tfrac14 (a^2+b^2+c^2)
= ab-\tfrac14(a+b-c)^2
\end{equation}
where $(a,b,c) = (\mu(A),\mu(B),\mu(C))$.
\end{theorem}

The conclusion \eqref{ineq:RSprime} can also be stated
\begin{equation} \label{ineq:RSdoublep}
\int_C \one_A*\one_B\,d\mu \le 
\mu(A)\mu(B)-\tfrac14 (\mu(A)+\mu(B)-\mu(C))^2
\end{equation}
where $\tau$ is defined by $\mu(C) = \mu(A)+\mu(B)-2\tau$.

Both hypotheses 
\eqref{RSmeasurerestriction} 
are invariant under permutations of $(A,B,C)$.
Likewise, the modified form $\int_{-C} \one_A*\one_B\,d\mu$,
where
$-C=\{-x: x\in C\}$, is invariant under permutations of $(A,B,C)$.

Equality holds in \eqref{ineq:RSprime},
under the indicated hypotheses
on $(\mu(A),\mu(B),\mu(C))$,
when $G=\torus$ and $(A,B,C) = (\astar,\bstar,\cstar)$.
Thus \eqref{ineq:RSprime} is a direct
restatement of \eqref{ineq:RS} 
in this parameter regime.

If the hypothesis \eqref{RSmeasurerestriction} 
is violated, then \eqref{ineq:RS} is easily verified directly, using the trivial upper bound
\begin{equation}
\int_C \one_A*\one_B\,d\mu
\le \min(\mu(A)\mu(B),\mu(B)\mu(C),\mu(C)\mu(A))
\end{equation}
which follows from
$\int_C \one_A*\one_B\,d\mu\le\int_G \one_A*\one_B\,d\mu = \mu(A)\mu(B)$
and permutation invariance.
In this paper we will focus primarily on the regime in which
the hypotheses \eqref{RSmeasurerestriction} hold.

The formulation \eqref{ineq:RS} is analogous to the Riesz-Sobolev inequality for $\reals^d$,
but now the symmetrization $\astar$ is a subset of $\torus$, rather than of $G$.
As is the case for $\reals^d$, the inequality for indicator functions implies the
generalization
\begin{equation}\label{rearrangegeneralfns} \langle f*g,h\rangle_G
\le \langle f^\star*g^\star,h^\star\rangle_\torus \end{equation}
for arbitrary nonnegative measurable functions defined on $G$,
with the pairing $\langle \varphi,\psi\rangle_G = \int_G \varphi\psi\,d\mu$
of real-valued functions, and with the natural extension of the definition of symmetrization $f^\star$
from indicator functions to general nonnegative functions.
Thus if $\torus$ is identified with $(-\tfrac12,\tfrac12]$
by identifying each equivalence class in $\reals/\integers$ by its
unique representative in this domain,
then $f^\star$ is even, is nonincreasing on
$[0,\tfrac12]$, and is equimeasurable with $f$.
Theorem~\ref{thm:relaxed} has a stronger conclusion than  \eqref{rearrangegeneralfns}.

For $G=\torus$, Theorem~\ref{thm:RS} was proved by Baernstein \cite{baernstein},
and was stated by Luttinger \cite{luttingerIII}.
For general compact connected Abelian groups,
inequality \eqref{ineq:RS} should be regarded as an equivalent formulation
of an inequality of Tao \cite{taokneser}. The deduction of \eqref{ineq:RS}
as a consequence of the formulation in  \cite{taokneser}
is carried out in \S\S\ref{section:Inter} and \ref{section:RSproved} below.

The formulation of our inverse theorems requires several definitions.

\begin{definition}
Two measurable sets $A,A'\subset G$ are equivalent if $\mu(A\symdif A')=0$.
Likewise, two ordered triples $\bE =(E_1,E_2,E_3)$ and $\bE'=(E'_1,E'_2,E'_3)$
are equivalent if $E_j$ is equivalent to $E'_j$ for each $j\in\{1,2,3\}$. 
\end{definition}

\begin{definition}
For $x\in\torus = \reals/\integers$,
$\norm{x}_\torus = |y|$ where $y\in[-\tfrac12,\tfrac12]$
is congruent to $x$ modulo $1$.
\end{definition}

\begin{definition}
A rank one Bohr set $\scriptb\subset G$ is a set of the form
\begin{equation}
\scriptb = \scriptb(\phi,\rho,c) = \{x\in G: \norm{\phi(x)-c}_\torus \le\rho\},
\end{equation}
where $\phi: G\to\torus$ is a continuous homomorphism,
$c\in\torus$, and $\rho\in[0,\tfrac12]$.
\end{definition}
By a homomorphism $\phi:G\to\torus$, we will always mean
a continuous homomorphism.


\begin{definition}
Two rank one Bohr subsets $\scriptb_1,\scriptb_2$ of $G$ are parallel if
they can be represented as
$\scriptb_j = \scriptb(\phi_j,c_j,\rho_j)$ with $\phi_1=\phi_2$.

An ordered triple $(\scriptb_1,\scriptb_2,\scriptb_3)$
of rank one Bohr subsets of $G$ is parallel if
these three sets are pairwise parallel.

An ordered triple $(\scriptb_1,\scriptb_2,\scriptb_3)$ 
of Bohr sets $\scriptb_j=\scriptb(\phi_j,c_j,\rho_j)$
is compatibly centered if $c_3=c_1+c_2$.
\end{definition}

\begin{definition} Let $(E_1,E_2,E_3)$ be an ordered triple of measurable subsets of $G$.

$(E_1,E_2,E_3)$ is admissible if
$0<\mu(E_i)<1$ for each $i\in\{1,2,3\}$,
$\mu(E_1)+\mu(E_2)+\mu(E_3)<2$,
and
\[\mu(E_k)\le \mu(E_i)+\mu(E_j)\] for each permutation
$(i,j,k)$ of $(1,2,3)$.

$(E_1,E_2,E_3)$ is strictly admissible if
it is admissible and
\begin{equation}
\mu(E_k)< \mu(E_i)+\mu(E_j) 
\end{equation}
for every permutation $(i,j,k)$ of $(1,2,3)$.

For any $\eta>0$, $(E_1,E_2,E_3)$ is $\eta$--strictly admissible if
it is admissible and
\begin{equation}
\mu(E_k)\le \mu(E_i)+\mu(E_j) - \eta \max(\mu(E_1),\mu(E_2),\mu(E_3))
\end{equation}
for every permutation $(i,j,k)$ of $(1,2,3)$. 
\end{definition}

\medskip

Admissibility is a property only of the measures $\mu(E_j)$,
and we will sometimes say instead that $(\mu(E_1),\mu(E_2),\mu(E_3))$ is admissible.
The condition that $\mu(E_k)\le \mu(E_i)+\mu(E_j)$
for every permutation $(i,j,k)$ of $(1,2,3)$ can be equivalently formulated as the condition
\begin{equation}
|\mu(E_i)-\mu(E_j)| \le \mu(E_k)\le \mu(E_i)+\mu(E_j)
\end{equation}
for any single permutation. 
Corresponding equivalences hold for strict admissibility and $\eta$--strict admissibility.
Simple consequences of $\eta$--strict admissibility are 
\begin{align} \mu(E_i) &\ge |\mu(E_j)-\mu(E_k)| + \eta \max(\mu(E_1),\mu(E_2),\mu(E_3)),
\\
\label{etacomparable} \mu(E_i)&\ge\eta\mu(E_j)  \end{align}
for every permutation $(i,j,k)$ of $(1,2,3)$.

\begin{definition}
The ordered triple $(A,B,C)$ of measurable subsets of $G$ is $\eta$--bounded if it satisfies
\begin{gather}
\label{etasmall}
\mu(A) + \mu(B) + \mu(C)\le (2-\eta)\mu(G),
\\
\label{etabig}
\min(\mu(A),\mu(B),\mu(C))\ge\eta \mu(G).
\end{gather}
\end{definition}

If $(A,B,C)$ is $\eta$--strictly admissible
and satisfies \eqref{etasmall}
then \[\max(\mu(A),\mu(B),\mu(C))
\le \tfrac{2-\eta}{2+\eta}\le 1-\tfrac{\eta}{2}.\]
Indeed, suppose that $\mu(C)$ is largest. 
Since $\mu(A) + \mu(B) \ge (1+\eta)\mu(C)$,
 $(2+\eta) \mu(C) \le \mu(A)+\mu(B)+\mu(C)\le 2-\eta$. 
\qed

\begin{theorem}[Uniqueness of maximizers up to symmetries] \label{thm:RSuniqueness}
Let $G$ be a compact connected Abelian topological group equipped with Haar measure $\mu$ satisfying $\mu(G)=1$. Let $(A,B,C)$ be an admissible triple of measurable subsets of $G$.
Then $\int_C \one_A*\one_B\,d\mu = \int_{C^*} \one_\astar*\one_\bstar\,dm$
if and only if $(A,B,C)$ is equivalent
to a compatibly centered parallel ordered triple of rank one Bohr sets.
\end{theorem}

As recognized by Burchard \cite{burchard}, if $\mu(C) > \mu(A)+\mu(B)$
then no characterization
of cases of equality is possible without the admissibility hypothesis,
beyond the trivial necessary and sufficient condition that 
$\one_A*\one_B$ should vanish $\mu$--almost everywhere on the complement of $C$.

Our main stability theorem is the following.

\begin{theorem}[Stability]\label{thm:RSstability}
For each $\eta>0$ there exist $\delta_0>0$ and $\boldC<\infty$ with the following property.
Let $G$ be a compact connected Abelian topological group
equipped with Haar measure $\mu$ satisfying $\mu(G)=1$.
Let $(A,B,C)$ be an $\eta$--strictly admissible and $\eta$-bounded ordered 
triple of measurable subsets of $G$.
Let $0\le \delta \le \delta_0$.
If 
$\int_C \one_A*\one_B\,d\mu \ge
\int_\cstar \one_\astar*\one_\bstar\,dm-\delta$
then there exists a compatibly centered parallel ordered triple
$(\scriptb_A,\scriptb_B,\scriptb_C)$ of rank one Bohr sets
satisfying
\begin{equation}
\mu(A\symdif \scriptb_A) \le \boldC \delta^{1/2}
\end{equation}
and likewise for $(B,\scriptb_B)$ and $(C,\scriptb_C)$.
\end{theorem}

The next two theorems generalize Theorems~\ref{thm:RS} and \ref{thm:RSstability}
from indicator functions of sets to more general functions.
Theorem~\ref{thm:relaxed} will be used in our proof of Theorem~\ref{thm:RSstability}.

\begin{theorem} \label{thm:relaxed}
Let $G$ be a compact connected Abelian topological group
equipped with Haar measure $\mu$ satisfying $\mu(G)=1$.
For any measurable functions $f,g,h:G\to[0,1]$, 
\begin{equation} \label{prec}
\langle f*g,h\rangle_G 
\le \langle \one_\astar*\one_\bstar,\one_\cstar\rangle_\torus
\end{equation}
where $\astar,\bstar,\cstar\subset\torus$ are 
intervals centered at $0$ 
satisfying
\[\big(m(\astar),m(\bstar),m(\cstar)\big)
= 
\big({\textstyle\int_G f\,d\mu,\,\int_G g\,d\mu,\, \int_G h\,d\mu}\big).\]
\end{theorem}

Inequality \eqref{prec} is known for $\reals^d$.
A corresponding stability theorem extends Theorem~\ref{thm:RSstability}
from indicator functions of sets to more general functions.

\begin{theorem} \label{thm:relaxedstability}
For each $\eta>0$ there exists $\bC<\infty$ with the following property.
Let $G$ be a  compact connected Abelian topological group
equipped with Haar measure $\mu$ satisfying $\mu(G)=1$.
Let $f,g,h:G\to[0,1]$ be measurable. 
Let $(\astar,\bstar,\cstar)\subset\torus$ be intervals centered at $0$
with Lebesgue measures $(\int f\,d\mu,\int g\,d\mu,\int h\,d\mu)$.
Let
\begin{equation}
\scriptd = 
 \langle \one_\astar*\one_\bstar,\one_\cstar\rangle_\torus
-\langle f*g,h\rangle_G. 
\end{equation}
Suppose that $(\astar,\bstar,\cstar)$ is $\eta$--strictly admissible
and $\eta$--bounded. 
If $\scriptd$ is sufficiently small as a function of $\eta$ alone
then
there exists a compatibly centered parallel triple $(\scriptb_f,\scriptb_g,\scriptb_h)$
of rank one Bohr subsets of $G$ 
satisfying
\begin{equation}
\norm{f-\one_{\scriptb_f}}_{L^1(G,\mu)}
\le \boldC \scriptd^{1/2}
\end{equation}
and likewise for $(g,\one_{\scriptb_g})$ and $(h, \one_{\scriptb_h})$.
\end{theorem}

Underlying our analysis are analogous results concerning 
Kneser's inequality \eqref{ineq:kneser}.
Continue to consider 
any compact connected Abelian topological group $G$,
equipped with Haar measure $\mu$. 
If $A,B\subset G$ are measurable sets that satisfy $\mu(A)+\mu(B)<\mu(G)$ and
$\min(\mu(A),\mu(B))>0$, 
then according to a theorem of Kneser \cite{kneser}, 
$\mu_*(A+B)=\mu(A)+\mu(B)$ if and only if 
there exists 
a pair of parallel rank one Bohr sets 
satisfying $A\subset\scriptb_A$, $B\subset\scriptb_B$,
and $\mu(\scriptb_A\setminus A) = \mu(\scriptb_B\setminus B)=0$.
For compact Abelian groups that are not necessarily connected, matters
are more complicated; see for instance \cite{griesmer2014}.
Moreover, Tao \cite{taokneser} and Griesmer \cite{griesmer} have proved associated stability,
or quantitative uniqueness, theorems.
Most relevant to our considerations is this result from \cite{griesmer}:
For every $\eps,\eta>0$ there exists $\delta>0$ such that
if $A,B\subset G$ are measurable sets satisfying
the auxiliary hypotheses
$\mu(A)\ge \eta\mu(G)$, $\mu(B)\ge \eta\mu(G)$, $\mu(A)+\mu(B)\le(1-\eta)\mu(G)$ and
the main hypothesis $\mu_*(A+B)\le \mu(A)+\mu(B)+\delta\mu(G)$,
then there exists a pair of parallel rank one Bohr sets $(\scriptb_A,\scriptb_B)$ satisfying
$A\subset\scriptb_A$,
$B\subset\scriptb_B$,
and 
\[\mu(\scriptb_A\setminus A) + \mu(\scriptb_B\setminus B)<\eps\mu(G).\] 
This is the result required for our analysis on general compact connected groups. 
We also prove a more quantitative version, Theorem~\ref{thm:kneserstability} below. 

Both \cite{taokneser} and \cite{griesmer} 
extend this result by weakening the hypothesis to one 
which involves an upper bound only on the Haar measure of 
$\{x\in A+B: \one_A*\one_B(x)\ge\rho\}$ for sufficiently small $\rho>0$.
Our development does not rely on this extension.

A more quantitative stability theorem for sumsets
is the following.

\begin{theorem} \label{thm:kneserstability}
For each $\eta,\eta'>0$ there exist $\delta_0>0$
and $\boldC<\infty$ with the following property.
Let $G$ be any compact connected Abelian topological group 
equipped with normalized Haar measure $\mu$.
Let $A,B\subset G$ be a pair of measurable sets satisfying
$\min(\mu(A),\mu(B)) \ge \eta'$
and $\mu(A)+\mu(B)\le 1-\eta$.
If  $\mu(A+_0 B)\le\mu(A)+\mu(B)+\delta\min(\mu(A),\mu(B))$
and $\delta\le\delta_0$, 
then there exists a pair of parallel rank one Bohr sets
$(\scriptb_A,\scriptb_B)$ such that
$A\subset\scriptb_A$, $B\subset\scriptb_B$, and
\begin{equation}
\mu(\scriptb_A\setminus A) + \mu(\scriptb_B\setminus B) \le \boldC \delta \min(\mu(A),\mu(B)).
\end{equation}
\end{theorem}

Candela and de~Roton \cite{candela+deroton}
have proved a theorem of this type for the special case $G=\torus$
in which the relationship between $m(\scriptb_A\setminus  A)$
and $m_*(A+B)-m(A)-m(B)$ is made quite precise, for an interesting range
of parameters.
We believe that their theorem extends to arbitrary compact connected Abelian groups, with the same relationship between parameters.

\medskip \noindent
{\bf Organization of the paper.}\ 
We begin by reviewing in \S\ref{section:Inter} 
an inequality of Tao \cite{taokneser},
stating several equivalent reformulations and establishing a refinement.
This refinement is used in \S\ref{section:RSproved}
to prove the Riesz-Sobolev--type inequality of Theorem~\ref{thm:RS}. The defect 
$\scriptd(A,B,C) = \langle \one_\astar*\one_\bstar,\one_\cstar\rangle 
- \langle \one_A*\one_B,\one_C\rangle$
and a related defect $\scriptd'(A,B,\tau)$, 
in terms of which much of our analysis is naturally phrased,
are introduced in \S\ref{section:RSproved}.

In \S\ref{section:twokey} we
discuss two key principles, submodularity and complementation. 
At the heart of our analysis of stability for the Riesz-Sobolev-type inequality \eqref{ineq:RS}
is a connection, developed in \cite{christRS} for $G=\reals$,
between (near) equality in the Riesz-Sobolev equality and
(near) equality in the sumset inequality
for certain associated sets.
This connection only applies directly in the case in which
two of the three sets $A,B,C$ have equal measures.
\S\ref{section:alink} reviews this connection 
and adapts it to general connected compact Abelian groups. 
\S\ref{section:towards} begins a reduction of the general case to the special case
of two sets of equal measures. This reduction does not
proceed in the same way as the corresponding reduction in \cite{christRS} for the Euclidean
case.

\S\ref{section:RSperturbative} establishes the conclusion of 
Theorem~\ref{thm:RSstability} in its quantitative form,
for the perturbative regime in which
$(A,B,C)$ is assumed to be within a certain threshold distance of 
a compatibly centered parallel triple of rank one Bohr sets. 
\S\ref{section:sumsetperturbative}
digresses to establish quantitative stability for
Kneser's inequality  in the perturbative regime in which $A,B$ are assumed
to be moderately close to a pair of parallel rank one Bohr sets.
These perturbative results are elements of our more general analysis of stability
in the absence of any perturbative hypotheses.

Theorem~\ref{thm:relaxed} and Theorem~\ref{thm:relaxedstability}
concern relaxed variants of the Riesz-Sobolev-type inequality
and its companion inverse stability theorem.
They are proved in \S\ref{section:relax} and \S\ref{section:relax2},
respectively. 

\S\ref{section:specialcase} and \S\ref{section:When} 
analyze the special case in which the defect $\scriptd(A,B,C)$
is small and one of the three sets is well
approximated by a rank one Bohr set. 
\S\ref{section:specialcase} treats the sub-subcase in which $G=\torus$ and $C$ is an interval. 
In \S\ref{section:When}, we reduce matters
from general groups $G$ to $\torus$.
The situation that arises on $\torus$ in this way
belongs to the more general framework of Theorems~\ref{thm:relaxed} and \ref{thm:relaxedstability}.
That framework comes into play at this juncture.
The proof of Theorem~\ref{thm:RSstability} is completed in \S\ref{section:finishRSstability}.

Another thread is taken up in \S\ref{section:flow} and \S\ref{section:finishingtorus},
which are concerned with the important group $G=\torus$.
This thread is founded on the monotonicity
of a normalized version of the functional
$\int_C \one_A*\one_B\,dm$
under a certain continuous deformation of $A,B,C$.
This deformation is developed in \S\ref{section:flow}.
As an application, 
in \S\ref{section:finishingtorus} we 
establish Theorem~\ref{thm:RS_Tstability}, a refinement for $G=\torus$ of
Theorem~\ref{thm:RSstability} which in an appropriate sense
eliminates the dependence of the conclusion 
on a lower bound for $\min(m(A),m(B),m(C))$.

One could alternatively bypass the 
analysis in \S\ref{section:specialcase} of the situation 
in which $G=\torus$ and $C$ is an interval, 
by invoking the theory for $\torus$ established in \S\ref{section:finishingtorus}.

\begin{notation}
For the sake of economy, we will often refer to \eqref{ineq:RS}
as the Riesz-Sobolev inequality, or the Riesz-Sobolev inequality for $G$.
Throughout the remainder of the paper,
$G$ denotes a compact connected Abelian topological group 
equipped with a complete Haar measure $\mu$
that is normalized in the sense that $\mu(G)=1$.
This is a hypothesis of all lemmas and propositions, though
it is not included in their statements.
It is implicitly asserted that all constants
in upper and lower bounds in theorems, propositions, lemmas, 
and inequalities are independent of $G$,
except when the special case $G=\torus$ is explicitly indicated.

$m$ denotes Lebesgue measure for $\torus = \reals/\integers$. 
$m(E)$ is alternatively denoted by $|E|$ in some parts of the discussion.
$C$ with no subscript is used to denote a subset of $G$, rather than a constant.
$c$, $c'$, and $\bC$ denote unspecified positive finite constants, 
whose values may change freely from one occurrence to the next. 
$C_n$ denotes a constant that is fixed for a relatively short portion
of the discussion.

It will be convenient in the analysis of the functional
$\langle \one_{E_1}*\one_{E_2},\one_{E_3}\rangle$ 
to be able to freely interchange the sets $E_j$.
For that purpose, we work with a more symmetric variant.
For measurable functions $f_j:G\to[0,\infty)$,
\begin{equation}
\scriptt_G(f_1,f_2,f_3) = \iint_{x+y+z=0} f_1(x)f_2(y)f_3(z)\,d\lambda(x,y,z)
\end{equation}
where $\lambda$ is the measure on $\{(x,y,z)\in G^3: x+y+z=0\}$
defined by pulling back the measure $\mu\times\mu$ on $G\times G$
via the mapping $(x,y,z)\mapsto (x,y)$. This definition of $\lambda$ is invariant
with respect to permutation of the three coordinates.
Equivalently,
\[ \scriptt_G(\bff)
= \scriptt_G(f_1,f_2,f_3) = \iint_{G^2} f_1(x)f_2(y)f_3(-x-y)\,d\mu(x)\,d\mu(y).  \]
For a three-tuple $\bE = (E_j: j\in\{1,2,3\})$ of sets,
we write $\scriptt_G(\bE) = \scriptt_G(\bff)$
with $f_j=\one_{E_j}$.

We sometimes work simultaneously on a general group $G$ and on $\torus$,
and write $\scriptt_G$ and/or $\scriptt_\torus$ to distinguish between
the functionals associated to the two groups. Defining 
$$\overline{\scriptd}(A,B,C):=\scriptt_\torus(A^\star,B^\star,C^\star)-\scriptt_G(A,B,C),
$$
one has $\overline{\scriptd}(A,B,C)=\scriptd(A,B,-C)$.
\end{notation}

\medskip

The authors are grateful to Rupert Frank, who kindly called their attention to a slip 
in the proof of Theorem~\ref{thm:relaxedstability} in an earlier draft. A version of Theorem \ref{thm:relaxedstability} is proved by Frank and Lieb in the Euclidean setting in arbitrary dimension \cite{franklieb}, in the special case where two of the functions are identical and the third is the indicator function of a ball.

\section{Refinement of a related inequality} \label{section:Inter}

In this section we review an inequality of Tao \cite{taokneser},
discuss multiple equivalent reformulations,
and formulate and prove a sharper inequality,
from which the Riesz-Sobolev inequality \eqref{ineq:RS} for $G$ 
will subsequently be derived.

The inequality of \cite{taokneser} states that
for any compact connected Abelian group $G$ with normalized Haar measure $\mu$,
for any measurable $A,B\subset G$,
\begin{equation}\label{ineq:taokneser1} 
\int_G \min(\one_A*\one_B,\tau)\,d\mu \ge \tau \min(\mu(A)+\mu(B)-\tau,1)
\ \forall\, 0\le \tau \le \max(\mu(A),\mu(B)).
\end{equation}
The inequality is trivial in the range
$\min(\mu(A),\mu(B))\le\tau\le \max(\mu(A),\mu(B))$,
in the sense that for arbitrary $A,B$ equality holds 
when $\tau$ is equal to the minimum or maximum, 
while for $\tau$ in the open interval $\big(\min(\mu(A),\mu(B)),\,\max(\mu(A),\mu(B))\big)$,
the left-hand side is equal to $\mu(A)\cdot\mu(B)$ 
and \eqref{ineq:taokneser1} holds with strict inequality.
\eqref{ineq:taokneser1} also holds with equality whenever $\mu(A)+\mu(B)\ge 1+\tau$,
for in that case,
\[\one_A*\one_B(x) = \mu(A \cap (x-B)) \ge \mu(A)+\mu(B)-1\ge \tau\] for every $x\in G$,
so both the left-- and right--hand sides are equal to $\tau$.
\eqref{ineq:taokneser1} never holds when $\tau > \max(\mu(A),\mu(B))$.

If $G=\torus$ and $A,B\subset\torus$ are intervals centered at $0$ then equality holds
in \eqref{ineq:taokneser1} whenever $\tau\le\min(\mu(A),\mu(B))$. Therefore this inequality
can be equivalently restated as
\begin{multline}\label{ineq:taokneser2} 
\int_G \min(\one_A*\one_B,\tau)\,d\mu \ge 
\int_\torus \min(\one_{\astar}*\one_{\bstar},\tau)\,dm
\ \forall\, 0\le \tau\le \min(\mu(A),\mu(B)).
\end{multline}

By virtue of the identities 
\begin{equation} \label{identity} \int_G \one_A*\one_B\,d\mu = \mu(A)\cdot\mu(B) \end{equation}
and $\max(f,g)+\min(f,g) = f+g$,
\eqref{ineq:taokneser2} can be equivalently reformulated as
\begin{multline}
\int_G \max(\one_A*\one_B-\tau,0)\,d\mu \le (\mu(A)-\tau)(\mu(B)-\tau)
\\
\ \forall\, \tau\in[\mu(A)+\mu(B)-1,\, \min(\mu(A),\mu(B))]
\label{ineq:taokneser3}
\end{multline}
with $\int_G \max(\one_A*\one_B-\tau,0)\,d\mu = \mu(A)\mu(B)-\tau$
for all $\tau\in[0,\mu(A)+\mu(B)-1]$. 
Likewise, \eqref{ineq:taokneser2} can be reformulated as
\begin{multline}
\int_G \max(\one_A*\one_B-\tau,0)\,d\mu 
\le 
\int_\torus \max(\one_{\astar}*\one_{\bstar}-\tau,0)\,dm
\\
\forall\,0\le \tau\le \min(\mu(A),\mu(B)).
\label{ineq:taokneser4}
\end{multline}
These four inequalities are equivalent in the sense that any one of them follows from 
any other one by simple manipulations augmented by 
the above discussion of the cases in which $\min(\mu(A),\mu(B))\le \tau\le \max(\mu(A),\mu(B))$
or $\tau\le \mu(A)+\mu(B)-1$.

The inequalities \eqref{ineq:taokneser1} through \eqref{ineq:taokneser4} can be reformulated
in terms of superlevel sets and associated distribution functions.
The following notation \eqref{Superlevel} will be used throughout the paper.

\begin{definition}
For measurable sets $A,B\subset G$ and for $t\ge 0$,
the associated superlevel set is
\begin{equation} \label{Superlevel}
S_{A,B}(t) = \{x\in G: \one_A*\one_B(x)>t\}.
\end{equation}
\end{definition}
Superlevel sets appear in fundamental formulae for the functionals of interest here:
\begin{align}
&\int_G \one_A*\one_B\,d\mu = \mu(A)\cdot\mu(B) = \int_0^\infty \mu(S_{A,B}(t))\,dt,
\\
&\int_{S_{A,B}(\tau)}\one_A*\one_B\,d\mu
= \tau \mu(S_{A,B}(\tau)) + 
\int_\tau^\infty \mu(S_{A,B}(t))\,dt,
\\
&\int_G \max(\one_A*\one_B- \tau,0)\,d\mu
= \int_\tau^\infty \mu(S_{A,B}(t))\,dt.
\end{align}

Thus \eqref{ineq:taokneser3} can be written
\begin{multline} \label{ineq:taokneser5}
\int_\tau^\infty \mu(S_{A,B}(t))\,dt
\le 
(\mu(A)-\tau)(\mu(B)-\tau)
\\
\ \forall\, \tau\in[\mu(A)+\mu(B)-1,\, \min(\mu(A),\mu(B))].
\end{multline}

The next result sharpens \eqref{ineq:taokneser5} and will be the basis of
our proof of Theorem~\ref{thm:RS}.

\begin{theorem} \label{thm:sharpKPRGT}
Let $G$ be a compact connected Abelian topological group,
equipped with normalized Haar measure $\mu$.
Suppose that
\begin{equation} \label{tauhypothesis}
0\le\tau\le\min (\mu(A),\mu(B))
\end{equation} 
and that 
\begin{equation} \label{SABhypothesis}
\mu(A)+\mu(B) + \mu(S_{A,B}(\tau))\le 2.
\end{equation}
Let $\sigma = \tfrac12 (\mu(A)+\mu(B)-\mu(S_{A,B}(\tau)))$.
Then
\begin{equation} \label{eq:sharpKPRGTconclusion}
\int_\tau^\infty \mu(S_{A,B}(t))\,dt 
\le (\mu(A)-\tau)(\mu(B)-\tau)-h,
\end{equation}
where
\begin{equation}
h=
\left\{
\begin{alignedat}{2}
&(\sigma-\tau)^2 
&&\text{if $\sigma\le\min(\mu(A),\mu(B))$}
\\ 
&(\min(\mu(A),\mu(B))-\tau)^2
\qquad &&\text{if $\sigma > \min(\mu(A),\mu(B))$.}
\end{alignedat} \right.
\end{equation}

In particular, if 
$\mu(A)+\mu(B)\le 1+\tau$
then 
\begin{equation}
\int_\tau^\infty \mu(S_{A,B}(t))\,dt 
\le
\int_\tau^\infty m(S_{\astar,\bstar}(t))\,dt -h.
\end{equation}
\end{theorem}

The form of the right-hand side of \eqref{eq:sharpKPRGTconclusion}
is unnatural when $\mu(A)+\mu(B)>1+\tau$,
in the sense that 
$\int_\tau^\infty m(S_{\astar,\bstar}(t))\,dt=\mu(A)\mu(B)-\tau$ 
is strictly smaller than $(\mu(A)-\tau)(\mu(B)-\tau)$ 
for such values of $\tau$.

\begin{proof}[Proof of Theorem~~\ref{thm:sharpKPRGT}]
Write $S(t)=S_{A,B}(t)$ to simplify notation.
With $\sigma$ as defined above, the hypothesis 
$\mu(S(\tau)) +\mu(A) + \mu(B)\le 2$
can be equivalently written as
$\mu(A)+\mu(B)-1\le\sigma$.
This is one of two conditions
needed to apply \eqref{ineq:taokneser5}
to $\int_\sigma^\infty \mu(S(t))\,dt$.
The second condition is that $\sigma\le\min(\mu(A),\mu(B))$, 
which need not hold under the hypotheses of Theorem~\ref{thm:sharpKPRGT}, in general.
The proof is consequently organized into cases.

If $\sigma\le\tau$ 
then indeed $\sigma\le\min(\mu(A),\mu(B))$, so \eqref{ineq:taokneser5} 
may be applied to obtain
\begin{align*}
\int_\tau^\infty \mu(S(t))\,dt
&= 
\int_\sigma^\infty \mu(S(t))\,dt
- \int_\sigma^\tau \mu(S(t))\,dt
\\& 
\le (\mu(A)-\sigma)(\mu(B)-\sigma) - (\tau-\sigma) \mu(S(\tau))
\\& 
= (\mu(A)-\tau)(\mu(B)-\tau) -(\sigma-\tau)^2 
\end{align*}
as follows by expanding $\tau = \sigma - (\tau-\sigma)$
in the product $(\mu(A)-\tau)(\mu(B)-\tau)$
and invoking the relation $\mu(A)+\mu(B) = 2\sigma + \mu(S(\tau))$.

If $\tau\le\sigma$ and if $\sigma$ does satisfy 
$\sigma \le\min(\mu(A),\mu(B))$, then again by \eqref{ineq:taokneser5},
\begin{align*}
\int_\tau^\infty \mu(S(t))\,dt
&= 
\int_\sigma^\infty \mu(S(t))\,dt
+ \int_\tau^\sigma \mu(S(t))\,dt
\\& 
\le (\mu(A)-\sigma)(\mu(B)-\sigma) 
+(\sigma-\tau) \mu(S(\tau))
\end{align*}
which we have already stated to be equal to $(\mu(A)-\tau)(\mu(B)-\tau) -(\sigma-\tau)^2$.

If on the other hand 
$\sigma \ge \min(\mu(A),\mu(B))$
then by permutation invariance, we may assume without loss of generality that $\mu(A)\le\mu(B)$.
Thus $\tfrac12(\mu(A)+\mu(B)-\mu(S(\tau)))=\sigma \ge \mu(A)$,
so $\mu(S(\tau)) \le  \mu(B)-\mu(A)$.
Since $\one_A*\one_B\le \mu(A)$,
\begin{equation*}
\int_\tau^\infty \mu(S(t))\,dt
= \int_\tau^{\mu(A)} \mu(S(t))\,dt
\le 
(\mu(A)-\tau) \mu(S(\tau))
\end{equation*}
since the integrand is a nonincreasing function of $t$.
The right-hand side is
\begin{equation*}
\le (\mu(A)-\tau)(\mu(B)-\mu(A))
= (\mu(A)-\tau)(\mu(B)-\tau) - (\mu(A)-\tau)^2.
\end{equation*}
\end{proof}

\begin{corollary}
Let $G$ be a compact connected Abelian topological group,
equipped with Haar measure $\mu$ satisfying $\mu(G)=1$.
Let $A,B\subset G$ be measurable sets. 
Suppose that
\begin{equation}
\mu(A)+\mu(B)-1 < t < \min(\mu(A),\mu(B)).
\end{equation}
If $(A,B,t)$ achieves equality in \eqref{ineq:taokneser1}
(equivalently in any or all of \eqref{ineq:taokneser2},
\eqref{ineq:taokneser3},
\eqref{ineq:taokneser4}),
then \begin{equation}\label{equalityintaokneser} \mu(S_{A,B}(t)) = \mu(A)+\mu(B) -2t.\end{equation}
\end{corollary}

We remark that $\mu(A)+\mu(B) -2\tau$ is not an extremal value for $\mu(S_{A,B}(\tau))$
for any single value of $\tau$;
$\mu(S_{A,B}(\tau))$ can in general be either larger, or smaller.

\begin{proof}
If $\mu(A)+\mu(B)+\mu(S_{A,B}(t))\le 2$ then all hypotheses of Theorem~\ref{thm:sharpKPRGT}
are satisfied, and \eqref{equalityintaokneser} follows from its conclusion
since $t$ is strictly less than $\min(\mu(A),\mu(B))$.

We claim that $\mu(S_{A,B}(t))\leq 1-t$, whence
$\mu(A)+\mu(B)+\mu(S_{A,B}(t))\leq 1+t+1-t =2$,
completing the proof of the corollary.
Suppose to the contrary that $\mu(S_{A,B}(t)) > 1-t$. 
Define $\tau\in(0,t)$ by $\mu(A)+\mu(B)=1+\tau$. 

For every $x\in G$, $\one_A*\one_B(x)=\mu\big(A\cap(x-B)\big)\geq \mu(A)+\mu(B)-1=\tau$. Thus, for every $r\in [0,\tau)$, $S_{A,B}(r)=G$, so
\begin{equation}\label{eq:strong}\mu(S_{A,B}(r))=1\text{ for every }r\in [0,\tau).  \end{equation}
For any $r\in [\tau,t]$, $S_{A,B}(r)\supset S_{A,B}(t)$, so
\begin{equation}\label{eq:small level sets}\mu(S_{A,B}(r))
\ge \mu(S_{A,B}(t))
> 1-t\text{ for every }r\in [\tau,t].  \end{equation}
The assumption that $(A,B,t)$ satisfies equality in \eqref{ineq:taokneser3} means that
$$\int_G \min(1_A*1_B,t)\,d\mu =t(\mu(A)+\mu(B)-t)=t(1+\tau-t). $$
Substituting
\[\int_G \min(1_A*1_B,t)\,d\mu
= \int_{0}^{t}\mu(S_{A,B}(r))\,dr\]
in the left-hand side and invoking \eqref{eq:strong} and \eqref{eq:small level sets} gives
\begin{multline*}
t(1+\tau-t)
=\int_{0}^{t}\mu(S_{A,B}(r))\,dr 
=\int_{0}^{\tau}\mu(S_{A,B}(r))\,dr+\int_{\tau}^t\mu(S_{A,B}(r))\,dr\\
> \int_{0}^{\tau}1+\int_{\tau}^t(1-t)\,dr
=\tau+(t-\tau)(1-t)
=t(1+\tau-t),
\end{multline*}
which is a contradiction. 
Therefore $\mu(S_{A,B}(t))\le 1-t$, and the proof of the corollary is complete.
\end{proof}

\section{On the Riesz-Sobolev for $G$} \label{section:RSproved}

In this section we prove the Riesz-Sobolev inequality \eqref{ineq:RS}
for $G$ using Theorem~\ref{thm:sharpKPRGT}. The sharpened form \eqref{eq:sharpKPRGTconclusion}
of \eqref{ineq:taokneser3} for $\sigma\le\min(\mu(A),\mu(B))$
is exactly what is needed in this derivation. Also, for the defects $\mathcal{D}$ and $\mathcal{D}'$ corresponding to these inequalities, defined below, we discuss approximation of the set $C$ by
superlevel sets $S_{A,B}(t)$, under the assumption that $\scriptd(A,B,C)$ is small.
We also discuss majorization of $\scriptd(A,B,C)$ by $\scriptd'(A,B,\tau)$ and vice versa,
under appropriate hypotheses linking $C$ to $\tau$.

The defects $\scriptd(A,B,C)$ and $\scriptd'(A,B,\tau)$ are defined as follows.

\begin{definition}
\begin{align}
&\scriptd(A,B,C) = \int_\cstar \one_\astar*\one_\bstar\,dm 
- \int_C \one_A*\one_B\,d\mu.
\\
&\scriptd'(A,B,\tau) = \int_\torus \max(\one_\astar*\one_\bstar-\tau,0)\,dm
- \int_G \max(\one_A*\one_B-\tau,0)\,d\mu.
\end{align}
\end{definition}
Theorem~\ref{thm:RS} states that $\scriptd(A,B,C)\ge 0$ for any ordered triple,
while inequality \eqref{ineq:taokneser4} states that
$\scriptd'(A,B,\tau)\ge 0$ for all $\tau\in[0,\min(\mu(A),\mu(B))]$.
These defects can usefully be expressed in terms of distribution functions
$\mu(S_{A,B}(t))$, as discussed in \S\ref{section:Inter}.

The following quantity arises throughout our analysis.
\begin{definition}
To sets $A,B,C\subset G$ satisfying $\mu(C)\le\mu(A)+\mu(B)$
is associated
\begin{equation}
\tau_C = \tfrac12 (\mu(A)+\mu(B)-\mu(C)).
\end{equation}
\end{definition}
This quantity satisfies $m(S_{\astar,\bstar}(\tau_C))=m(\cstar)= \mu(C)$;
it represents the parameter $\tau$ for which $\cstar$
equals the superlevel set $S_{\astar,\bstar}(\tau)$,
provided that $(\mu(A),\mu(B),\mu(C))$ 
is admissible.

\begin{lemma} \label{lemma:RSforsupers}
Suppose that $A,B\subset G$ and $\tau\in[0,1]$ satisfy 
$$
0\le\tau\le\min (\mu(A),\mu(B)),
$$
$$
\mu(A)+\mu(B) + \mu(S_{A,B}(\tau))\le 2.
$$
Then
\begin{equation} \label{RSsupersconclusion}
\tau\mu(S_{A,B}(\tau)) + \int_\tau^\infty \mu(S_{A,B}(\alpha))\,d\alpha
\le \mu(A)\mu(B)-\tfrac14 (\mu(A)-\mu(B)-\mu(S_{A,B}(\tau)))^2.
\end{equation}
\end{lemma}

That is, $(A,B,C)= (A,B,S_{A,B}(\tau))$ satisfies \eqref{ineq:RSdoublep} (and thus \eqref{ineq:RS}, under some additional hypotheses).

\begin{proof}
Define $\sigma = \tfrac12(\mu(A)+\mu(B)-\mu(S_{A,B}(\tau)))$.
Equivalently,  $\mu(S_{A,B}(\tau)) = \mu(A)+\mu(B)-2\sigma$.
Calculate
\begin{align*}
(\mu(A)-\tau)(\mu(B)-\tau)  - \big(\mu(A)\mu(B)-\sigma^2\big)
&= -\tau(\mu(A)+\mu(B))+\tau^2+\sigma^2
\\&
= -\tau(\mu(A)+\mu(B)-2\sigma) + (\sigma-\tau)^2
\\&
= -\tau \mu(S(\tau)) + (\sigma-\tau)^2.
\end{align*}
Thus
\begin{equation} \label{muSbound}
\tau\mu(S_{A,B}(\tau))
=
-(\mu(A)-\tau)(\mu(B)-\tau)  + \big(\mu(A)\mu(B)-\sigma^2\big) + (\sigma-\tau)^2.
\end{equation}
Note that $(A,B,\tau)$ satisfies the hypotheses of Theorem~\ref{thm:sharpKPRGT}. Applying Theorem \ref{thm:sharpKPRGT} to the second term on the left-hand side of
\eqref{RSsupersconclusion} and then invoking \eqref{muSbound}  gives
the desired upper bound
\begin{align*}
\tau \mu(S_{A,B}(\tau)) + (\mu(A)-\tau)(\mu(B)-\tau) - (\sigma-\tau)^2
= \mu(A)\mu(B) -\sigma^2.
\end{align*}
\end{proof}

\begin{proof}[Proof of Theorem~\ref{thm:RS}]
Let $A,B,C\subset G$. 
Consider first the case
in which $\mu(A)+\mu(B)+\mu(C)\geq 2$.
Define $t$ by $\mu(A)+\mu(B)=1+t$; note that $t\geq 0$. 
Then $\one_A*\one_B(x)\ge t$ for every $x\in G$.
Indeed,
\[\one_A*\one_B(x) = \mu(A\cap(x-B)) \ge \mu(A)+\mu(x-B)-\mu(G)
= \mu(A)+\mu(B)-1 = t.\]
Therefore
\[ \int_C \one_A*\one_B\,d\mu
\le
\int_G \one_A*\one_B\,d\mu
-t\mu(G\setminus C)
= \mu(A)\mu(B)-t(1-\mu(C)).
\]
On the other hand, $\one_\astar*\one_\bstar\equiv t$
on $\torus\setminus\cstar$, and so the same calculation gives
\[\int_\cstar \one_\astar*\one_\bstar\,dm
= m(\astar)m(\bstar)-t(1-m(\cstar)) 
= \mu(A)\mu(B)-t(1-\mu(C)).\]
Thus the stated conclusion holds in this case.

If $\mu(C)\le |\mu(A)-\mu(B)|$
then, while $\one_A*\one_B\leq \min(\mu(A),\mu(B))$ on $C$, it also holds that $\one_\astar*\one_\bstar\equiv \min(m(\astar),m(\bstar))$
on $\cstar$. Therefore \eqref{ineq:RS} holds.
If $\mu(C)\ge \mu(A)+\mu(B)$
then either $\mu(A)\le |\mu(B)-\mu(C)|$ or $\mu(B)\le |\mu(A)-\mu(C)|$. 
\eqref{ineq:RS} thus follows by permutation invariance 
from the case in which $\mu(C)\le |\mu(A)-\mu(B)|$.

Assume henceforth that $\mu(A)+\mu(B)+\mu(C) < 2$,
and that $|\mu(A)-\mu(B)| < \mu(C) < \mu(A)+\mu(B)$.

If there exists $t\in[0,1]$ for which 
the superlevel set $S=S_{A,B}(t)$
satisfies $\mu(S)=\mu(C)$, then the desired inequality \eqref{ineq:RS} holds for $(A,B,C)$. More precisely,
$\int_C \one_A*\one_B \le \int_S \one_A*\one_B$.
The parameter $t$ satisfies $t\le\min(\mu(A),\mu(B))$, since $\norm{\one_A*\one_B}_{C^0}\le
\min(\mu(A),\mu(B))$ and $\mu(C)>0$.
It also satisfies $\mu(A)+\mu(B)\le 1+t$.
Indeed, if $\mu(A)+\mu(B)>1+t$ then $\one_A*\one_B(x)>t$ for every $x\in G$
as noted above,
so $S =S_{A,B}(t)=G$, so $\mu(C)=\mu(S)=1$, forcing
$\mu(A)+\mu(B)+\mu(C)=\mu(A)+\mu(B)+1 > 2+t\ge 2$ 
and thereby contradicting the assumption that $\mu(A)+\mu(B)+\mu(C)< 2$.

Thus the hypotheses of Lemma~\ref{lemma:RSforsupers}
are satisfied by $A,B,t$ and $S_{A,B}(t)$.
Applying that lemma to $S_{A,B}(t)$ gives the desired 
upper bound for $\int_{S_{A,B}(t)} \one_A*\one_B$, hence for $\int_C \one_A*\one_B$.

\medskip
It remains to reduce the general case to that in which there exists $t\in[0,1]$
satisfying $\mu(S_{A,B}(t))=\mu(C)$, under the hypotheses $\mu(A)+\mu(B)+\mu(C)< 2$ and $|\mu(A)-\mu(B)|<\mu(C)<\mu(A)+\mu(B)$.
We may also assume the auxiliary condition
\begin{equation} \label{auxcondition} \mu(\{x: \one_A*\one_B(x)>0\})\ge \mu(C).\end{equation}
Indeed, if this fails, set $\tilde C = C\cap \{x: \one_A*\one_B(x)>0\}$.
The value of the integral $\int_C \one_A*\one_B\,d\mu$
is unchanged when $C$ is replaced by $\tilde C$.
If $\mu(\tilde C)<|\mu(A)-\mu(B)|$ then
we have already observed that
$$\int_{\tilde C} \one_A*\one_B\,d\mu
\le
\int_{\tilde C^\star} \one_\astar*\one_\bstar\,dm
$$
(that is, $(A,B,\tilde{C})$ satisfies \eqref{ineq:RS}). Since $\mu(\tilde{C})\leq \mu(C)$, the right-hand side is in turn majorized by
$\int_{C^\star} \one_\astar*\one_\bstar\,dm$, so \eqref{ineq:RS} holds for $(A,B,C)$.
If $\mu(\tilde C) \ge |\mu(A)-\mu(B)|$
then it suffices to prove that $(A,B,\tilde C)$ satisfies \eqref{ineq:RS}.
Thus matters are reduced to the case in which
$(A,B,C)$ satisfies \eqref{auxcondition}.

Given \eqref{auxcondition}, a sufficient condition for the existence of $t$
satisfying $\mu(C)=\mu(S_{A,B}(t))$
is that all level sets of $\one_A*\one_B$ should be null sets, that is,
for every $r>0$, 
$\mu(\{x: \one_A*\one_B(x)=r\})=0$.
Moreover, because $(A,B,C)\mapsto \int_C \one_A*\one_B\,d\mu$
is continuous in the sense that
\[\int_{C_n}\one_{A_n}*\one_{B_n}
\,d\mu\to\int_C \one_A*\one_B\,d\mu
\text{ if } \mu(A_n\symdif A) + \mu(B_n\symdif B) + \mu(C_n\symdif C)\to 0,\]
it would suffice to construct $(A_n,B_n,C_n)$, converging
to $(A,B,C)$ in this sense, such that
all level sets of $\one_{A_n}*\one_{B_n}$ are $\mu$--null.

Such a construction does not necessarily exist in $G$, but it does in the auxiliary group $\tilde G = G\times\torus$ with normalized
Haar measure $\tilde\mu$. 
Consider a sequence of triples $(\alpha_n,\beta_n,\gamma_n)$ of Lebesgue measurable subsets of $\torus$
satisfying $\mu(\alpha_n)\to 1$ as $n\to\infty$ and likewise for $\mu(\beta_n),\mu(\gamma_n)$,
such that 
all level sets of $\one_{\alpha_n}*\one_{\beta_n}$ on $\torus$ are Lebesgue null sets.
The existence of such sequences can be proved in various ways.

Consider $(\tilde A,\tilde B,\tilde C) = (A\times\alpha_n, B\times\beta_n,C\times\gamma_n)$.
Then
$\one_{\tilde A_n}*\one_{\tilde B_n}$
is the product function 
$G\times\torus\owns (x,y)\mapsto (\one_A*\one_B(x))\cdot(\one_{\alpha_n}*\one_{\beta_n}(y))$,
so
\[ \int_{\tilde C_n} \one_{\tilde A_n}*\one_{\tilde B_n}\,d\tilde\mu
= \Big(\int_{C_n} \one_A*\one_B\,d\mu\big)
\cdot
\Big(\int_{\gamma_n} \one_{\alpha_n}*\one_{\beta_n}\,dm\big)\]
converges to $\int_C \one_A*\one_B\,d\mu$ as $n\to\infty$.
Moreover, 
all level sets of $\one_{\tilde A_n}*\one_{\tilde B_n}$ are null sets;
this is a simple consequence of Fubini's theorem and the
corresponding property of $\one_{\alpha_n}*\one_{\beta_n}$.
Therefore the conclusion of Theorem~\ref{thm:RS}, or equivalently
that of Theorem~\ref{thm:RSprime} (whose hypotheses are satisfied by $(\tilde{A}_n,\tilde{B}_n,\tilde{C}_n)$ for large $n$), holds for 
$(\tilde A_n,\tilde B_n,\tilde C_n)$
for all sufficiently large $n$.
Since $\tilde\mu(A_n)=\mu(A)m(\alpha_n)\to\mu(A)$
and likewise for $\tilde B_n,\tilde C_n$,
it follows from passage to the limit that the conclusion also holds 
for $(A,B,C)$.
\end{proof}

The proofs of the subsequent statements are not provided, as they are direct adaptations of proofs in \cite{christRS}.

The next lemma states that if $(A,B,C)$
nearly maximizes the Riesz-Sobolev functional $\int_C \one_A*\one_B\,d\mu$,
then $C$ nearly coincides with a superlevel set $S_{A,B}(\tau)$ (as long as $(A,B,C)$ is appropriately admissible).

\begin{lemma} \cite{christRS} \label{lemma:Staulowerbound}
\label{C almost level set} 
Let $A,B,C\subset G$ be measurable sets with $\mu(A),\mu(B),\mu(C)>0$. Suppose that
\begin{gather} 
 \big|\,\mu(A)-\mu(B)\,\big|  + 2\scriptd(A,B,C)^{1/2}< \mu(C) < \mu(A)+\mu(B) -2\scriptd(A,B,C)^{1/2}
\\
\mu(A)+\mu(B)+\mu(C) < 2-2\scriptd(A,B,C)^{1/2}.
\end{gather}
Define $\tau$ by $\mu(C) = \mu(A)+\mu(B)-2\tau$.  Then 
the superlevel set $S_{A,B}(\tau)$ satisfies
\begin{gather} \label{ineq:symmetricdifferencebound}
\mu(S_{A,B}(\tau)\bigtriangleup C) \le 4\scriptd(A,B,C)^{1/2}
\\
\label{RSsharpeneddisguised}
\big|\mu(S_{A,B}(\tau))-\mu(C)\big| \leq 2\mathcal{D}(A,B,C)^{1/2}
\\
\scriptd(A,B,S_{A,B}(\tau)) \le \scriptd(A,B,C).
\end{gather}
\end{lemma}

The next result sharpens Theorem~\ref{thm:RS}
in the same way that Theorem~\ref{thm:sharpKPRGT} sharpens \eqref{ineq:taokneser3}.
It is simply a restatement of \eqref{ineq:symmetricdifferencebound} in alternative terms.

\begin{theorem}
Under the hypotheses of Lemma~\ref{lemma:Staulowerbound},
\begin{equation}
\int_C \one_A*\one_B\,d\mu
\le \int_\cstar \one_\astar*\one_\bstar\,dm
- \tfrac1{16} \mu\big(C\symdif S_{A,B}(\tau_C)\big)^2
\end{equation}
where $\tau_C = \tfrac12 (\mu(A)+\mu(B)-\mu(C))$.
\end{theorem}

The next two lemmas relate the two defects $\scriptd,\scriptd'$ to one another.

\begin{lemma} \emph{\textbf{\cite{christRS}}} \label{returning to level set} 
Let $A,B$ be measurable subsets of $G$ of positive Haar measures, and suppose that
$\tau \in [0,\min (\mu(A),\mu(B))]$ and $\mu(A)+\mu(B)<1+\tau$. 
Then
$$\mathcal{D}(A,B,S_{A,B}(\tau))\leq \mathcal{D}'(A,B,\tau).  $$
\end{lemma}

\begin{lemma} \emph{\textbf{\cite{christRS}}} \label{lemma:726}
Let $A,B,C\subset G$ be measurable sets with positive Haar measures. Let
$\tau_C = \tfrac12 (\mu(A)+\mu(B)-\mu(C))$. 
If 
\begin{equation} \label{726lemmahypothesis}
\big|\,\mu(A)-\mu(B)\,\big|  + 2\scriptd(A,B,C)^{1/2}< \mu(C) < \mu(A)+\mu(B)-2\scriptd(A,B,C)^{1/2}
\end{equation}
and $\mu(A)+\mu(B)+\mu(C) \le 2-2\scriptd(A,B,C)^{1/2}$
then
\begin{equation} \scriptd'(A,B,\tau_C) \le 2\scriptd(A,B,C). \end{equation}
\end{lemma}

\begin{corollary} \emph{\textbf{\cite{christRS}}} 
\label{cor:KTsharpened} 
Let $G$ be a compact connected Abelian topological group, equipped with normalized Haar measure $\mu$.
Let $A,B\subset G$ be measurable sets with positive Haar measures. 
Let $\tau\in [0,\min(\mu(A),\mu(B))]$, and suppose that
$\mu(A)+\mu(B)\le 1+\tau$ and
\begin{gather*}
\big|\,\mu(A)-\mu(B)\,\big| \le\mu(S_{A,B}(\tau)),
\\
\mu(A)+\mu(B)+\mu(S_{A,B}(\tau))\le 2.
\end{gather*}
Then
\begin{equation}  \label{corKTsharpconclusion}
\big|\mu(S_{A,B}(\tau))-(\mu(A)+\mu(B)-2\tau)\big|
\le 2\scriptd'(A,B,\tau)^{1/2}.
\end{equation}
\end{corollary}

\begin{proof}
The hypotheses of Theorem~\ref{thm:sharpKPRGT} are satisfied.
The hypothesis 
$\big|\,\mu(A)-\mu(B)\,\big| \le\mu(S_{A,B}(\tau))$
of the corollary is equivalent to $\sigma\le\min(\mu(A),\mu(B))$, where $\sigma$ is defined by $\mu(S_{A,B}(\tau))=\mu(A)+\mu(B)-2\sigma$. Thus, \eqref{corKTsharpconclusion} holds by being a restatement of the conclusion of Theorem \ref{thm:sharpKPRGT}
for $\sigma$ in this range.
\end{proof}

\section{Two key principles} \label{section:twokey}

In analyzing near-maximizers $(A,B,C)$ of the Riesz-Sobolev functional,
we have found it to be useful to transform $(A,B,C)$ in several different ways.
Two of these are based on the principles of submodularity and complementation,
which are developed in this section
as Proposition~\ref{submodularity} and Lemma~\ref{lemma:complementD}, respectively.
A third is the transformation of $(A,B,C)$ to a triple $(A,B,\tau)$,
based on the relationship between $\scriptd(A,B,S_{A,B}(\tau))$ and $\scriptd'(A,B,\tau)$
explored in \S\ref{section:RSproved}.
A fourth is the flow $(A,B,C)\mapsto(A(t),B(t),C(t))$
introduced in \S\ref{section:flow}.
A fifth arises when $C\subset G$ is a rank one Bohr set or is well approximated
by such a set, and relates $\int_{C} \one_A*\one_B\,d\mu$
to a relaxed version of this functional for associated data on $\torus$.
This connection is developed in \S\ref{section:When}.
\medskip

At certain stages of the analysis we will pass from a triple $(A,B,C)$ 
to a related triple $(A',B',C')$ with certain more advantageous properties,
or from $(A,B,\tau)$ to $(A',B',\tau')$.
We want to do this without sacrificing smallness of $\scriptd(A,B,C)$ or
of $\scriptd'(A,B,\tau)$, respectively.
Two principles that make this possible are submodularity and complementation.

Let $G$ be a compact connected Abelian group $G$, with normalized Haar measure $\mu$.

\begin{proposition}[Submodularity] \emph{{(Tao \cite{taokneser})}}\label{submodularity} 
Let $A, B_1, B_2$ be measurable subsets of $G$, and let $\tau\in[ 0,\min(\mu(A),\mu(B_1\cap B_2))]$ with $\mu(A)+\mu(B_1\cup B_2)-\tau\leq 1$. Then
$$\mathcal{D}'(A,B_1\cap B_2,\tau)+\mathcal{D}'(A,B_1\cup B_2,\tau)\leq \mathcal{D}'(A,B_1,\tau)+\mathcal{D}'(A,B_2,\tau)
$$
and the above four quantities $\scriptd'$ are all nonnegative.
\end{proposition}

\begin{lemma} \label{lemma:complementadmiss}
Suppose that each of $A,B,C$ has Haar measure strictly $>0$
and strictly $<1$.
$(A,B,C)$ is admissible
and satisfies $\mu(A)+\mu(B)+\mu(C)\le 2$
if and only if 
$(G\setminus A,G\setminus B,C)$ is admissible
and satisfies $\mu(G\setminus A)+\mu(G\setminus B)+\mu(C)\le 2$.
\end{lemma}

\begin{proof}
The relation $\mu(C) \le \mu(G\setminus A) + \mu(G\setminus B)$
is equivalent to $\mu(A)+\mu(B)+\mu(C)\le 2$,
and by symmetry 
$\mu(C) \le \mu(A) + \mu(B)$
is equivalent to $\mu(G\setminus A)+\mu(G\setminus B)+\mu(C)\le 2$.

The relation $\mu(G\setminus A)\le \mu(G\setminus B)+\mu(C)$
is equivalent to $\mu(B)\le \mu(A)+\mu(C)$,
and interchanging $A,B$ in this equivalence yields the equivalence
of the remaining two relations.
\end{proof}

\begin{lemma}
For each $\eta>0$
there exists $\eta'>0$ with the following property.
Suppose that each of $A,B,C$ has Haar measure strictly $>0$
and strictly $<1$, and that
$(A,B,C)$ is $\eta$--strictly admissible 
and $\eta$-bounded.
Then 
$(G\setminus A,G\setminus B,C)$ is $\eta'$--strictly admissible
and $\eta'$-bounded.
\end{lemma}

This is proved in the same way as Lemma~\ref{lemma:complementadmiss}.
\qed

\begin{lemma} \label{lemma:complementsumset}
Suppose that each of $A,B$ has Haar measure strictly $>0$,
that $\mu(A)+\mu(B)<1$, and that $A+B$ is measurable.
Then
\begin{equation}
 \mu_*(A+\tilde B)-\mu(A)-\mu(\tilde B)
\le
\mu(A+B)-\mu(A)-\mu(B)
\end{equation}
where
$\tilde B = -\big(G\setminus(A+B)\big)$.
\end{lemma}

\begin{proof} It holds that $(G\setminus(A+B)) -A\subset  G\setminus B$.
Indeed, let $x\in A$ and $z\notin A+B$. If $y=z-x$ belongs to $B$ then
$x+y=z$, whence $z\in  A+B$, a contradiction.

Therefore $\mu_*\big(A-G\setminus(A+B) \big)\le 1-\mu(B)$
and consequently
\begin{multline*}
\mu_*\big(A-G\setminus(A+B) \big) -\mu(A)-\mu(G\setminus(A+B))
\\ \le 1-\mu(B)-\mu(A)-[1-\mu(A+B)]
\\
= \mu(A+B)-\mu(A)-\mu(B).\end{multline*}
\end{proof}

\begin{lemma}[Complementation] \label{lemma:complementD}
If $(A,B,C)$ is admissible and $\mu(A)+\mu(B)+\mu(C)\le 2$ then
\begin{equation} \label{eq:complementD}
\scriptd(A,B,C) = \scriptd(G\setminus A,\,G\setminus B,\,C).
\end{equation}
\end{lemma}

\begin{proof}
Writing $\one_{G\setminus A} = 1-\one_A$ and likewise for $B$,
then expanding the integrand, gives
\begin{align*}
\int_C \one_{G\setminus A}*\one_{G\setminus B} \,d\mu
&= \int_C (1-\mu(A)-\mu(B)+\one_A*\one_B)\,d\mu
\\&
= \int_C \one_A*\one_B\,d\mu +  \mu(C)(1-\mu(A)-\mu(B)).
\end{align*}
Since 
\[
(G\setminus A)^\star =\{\tfrac12-x: x\in \torus\setminus \astar\}
\]
and likewise for $(G\setminus B)^\star$,
and since $\astar,\bstar,\cstar$ are symmetric under $x\mapsto -x$,
the same calculation gives
\[
\int_\cstar \one_{(G\setminus A)^\star}*\one_{(G\setminus B)^\star}\,dm
= \int_\cstar \one_\astar*\one_\bstar\,dm
+  \mu(C)(1-\mu(A)-\mu(B))
\]
since $m(\astar)=\mu(A)$ and likewise for $B$.
Subtracting these two relations gives 
$\scriptd(A,B,C) = \scriptd(G\setminus A,\,G\setminus B,\,C)$.
\end{proof}

\section{A link between Riesz-Sobolev and sumset inequalities} \label{section:alink}

The next lemma lies at the heart of this part of the analysis.
It states that if $(A,B,C)$ is nearly a maximizer for 
the functional $\int_C \one_A*\one_B\,d\mu$,
then a certain associated superlevel set $S=S_{A,B}(\beta)$
has small sumset in the sense that  $\mu(S-S)$ is nearly equal to $2\mu(S)$. 
A theorem of Tao \cite{taokneser} then implies that S is nearly a rank one Bohr set. 
The same holds for $C$, as, by Lemma~\ref{lemma:Staulowerbound}, 
$C\symdif S$ has small Haar measure. 

However, the proof of Lemma~\ref{lemma:equalitycase} 
requires its very restrictive hypothesis that $\mu(A)=\mu(B)$. 
In an analysis of the Riesz-Sobolev equality for $\reals^1$ in
\cite{christRS}, this hypothesis was removed in a subsequent step, 
by a method that does not apply to compact groups $G$.
In the present paper
we will accomplish this removal for compact connected Abelian groups by an unrelated and somewhat lengthy 
alternative method based in part on ideas of Tao \cite{taokneser}. 
This necessitates the reductions carried out in \S\ref{section:towards}.

\begin{lemma} \emph{\textbf{\cite{christRS}}} \label{lemma:equalitycase}
Let $(A,B,C)$ be an $\eta$--strictly admissible 
ordered triple of measurable subsets of $G$ with positive Haar measures. 
Suppose that  
\begin{align} &\mu(A)=\mu(B) \le \tfrac12 \label{unnaturalequality},
\\ & \mu(C) \le \mu(A)-  4\scriptd(A,B,C)^{1/2} \label{unnatural},
 \\ &\scriptd(A,B,C)^{1/2} < \tfrac1{28}\eta \mu(A). \label{28appears}  \end{align} 
Let $\beta = \tfrac{1}{2}\big(\mu(A)+\mu(B)-\mu(C)\big)$.
Then
\begin{equation} \label{nearlyadditive}
\mu\big(S_{A,B}(\beta)-S_{A,B}(\beta)\big)
\le 2\mu(S_{A,B}(\beta)) + 12\scriptd(A,B,C)^{1/2}.
\end{equation}
\end{lemma}

The proof of this lemma is essentially identical to the proof of the corresponding result
in \cite{christRS}, so it is not included here. \qed
 
Under certain hypotheses, 
it can be concluded
that the set $S_{A,B}(\beta)$ above is nearly a rank one Bohr set.

\begin{corollary} \label{cor:equalitycase}
For each $\eps,\eta>0$ there exists $\rho>0$ with the following property.
Let $(A,B,C)$ be an $\eta$--strictly admissible $\eta$--bounded
ordered triple of measurable subsets of $G$
satisfying  the hypotheses
\eqref{unnaturalequality} and \eqref{unnatural} of Lemma~\ref{lemma:equalitycase}.
If $\scriptd(A,B,C)\le\rho$
then there exists a rank one Bohr set $\scriptb\subset G$ satisfying
\begin{equation} \label{CnearlyBohr} 
\mu(C\bigtriangleup \scriptb) \le \eps. \end{equation}

If $\mu(C)\le \mu(A)=\mu(B)\le \tfrac12-\eta$
and $\scriptd(A,B,C)=0$,
then there exists a rank one Bohr set $\scriptb\subset G$ satisfying
\begin{equation} \label{CreallyBohr} 
\mu(C\bigtriangleup \scriptb) =0. \end{equation}
\end{corollary}

\begin{proof}
Let $\eps>0$.
By \eqref{ineq:symmetricdifferencebound}, $\mu(S_{A,B}(\beta)\symdif C)\le 4\scriptd(A,B,C)^{1/2}$.
Moreover, if $\scriptd(A,B,C)$ is sufficiently small as a function of $\eta$, then
the conclusion  \eqref{nearlyadditive} of Lemma~\ref{lemma:equalitycase}
states that $S_{A,B}(\beta)$ satisfies a strong form of the hypothesis 
of the theorems of Tao \cite{taokneser} and Griesmer \cite{griesmer} discussed in \S1.
The conclusion of those theorems is the existence of a rank one Bohr set
satisfying $\mu(\scriptb\symdif S_{A,B}(\beta))\le\eps$,
where $\eps\to 0$ as $\scriptd(A,B,C)\to 0$ with $\eta$ fixed. 
Therefore
\[ \mu(\scriptb\symdif C) 
\le
\mu(\scriptb\symdif S_{A,B}(\beta))
+
\mu(S_{A,B}(\beta)\symdif C)
\le \eps + 4\scriptd(A,B,C)^{1/2}.\]

If $\eps=0$ and the measures of $A,B,C$ satisfy the indicated
hypotheses, then by Lemma~\ref{lemma:equalitycase},
$S=S_{A,B}(\beta)$ satisfies $\mu(\setminus C)=0$
and $\mu(S-S)\le 2\mu(S)$.
Therefore $\mu(S_{A,B}(\beta))=\mu(C)\le \tfrac12-\eta$,
and $S$ achieves equality in Kneser's inequality.
Therefore by Kneser's inverse theorem,
there exists a rank one Bohr set $\scriptb$ satisfying $\mu(\scriptb)=\mu(S)$
and $\mu(\scriptb\setminus S)=0$.
Thus $\mu(\scriptb\symdif C)=0$ also.
\end{proof}

\section{Two reductions} \label{section:towards}

This section is devoted to two auxiliary results, of limited if any intrinsic interest, 
whose purpose is to reduce the analysis of triples that nearly saturate 
the Riesz-Sobolev inequality
to triples that satisfy the hypotheses of Corollary~\ref{cor:equalitycase}.
In particular, we show that
if $(A,B,C)$ nearly maximizes the Riesz-Sobolev functional among triples of sets with specified
Haar measures, then
there exists a closely related near maximizing triple $(\tilde A,\tilde B,\tilde C)$ 
satisfying supplementary properties, including the hypotheses of Corollary \ref{cor:equalitycase}.
Those properties will subsequently be used to deduce that $(\tilde A,\tilde B,\tilde C)$
is nearly a compatibly centered parallel triple of rank one Bohr sets.
From that we will deduce the same property for $(A,B,C)$.
This will be achieved by ultimately applying this reasoning to a short chain
of triples $(A_n,B_n,C_n)$, 
with $(A_n,B_n,C_n)$ constructed recursively from $(A_{n-1},B_{n-1},C_{n-1})$
beginning with $(A_0,B_0,C_0)= (A,B,C)$,
and with conclusions propagated in 
reverse from $(A_{n},B_{n},C_{n})$ to $(A_{n-1},B_{n-1},C_{n-1})$, 

\begin{lemma} \label{reducing to equal} Let $(A,B,C)$ be an $\eta$-strictly admissible
and $\eta$--bounded triple of $\mu$-measurable subsets of $G$, satisfying
$$\mu(C)\leq \mu(A)\leq \mu(B),
$$
$$ \mu(A) \le \tfrac{1}{2} ,
$$
$$\mathcal{D}(A,B,C)^{1/2}\leq \tfrac{1}{400}\eta^2 \mu(B).
$$

Define $\tau$ by $\mu(C)=\mu(A)+\mu(B)-2\tau$. 
Then there exists a measurable set
$B'\subseteq G$ with $\mu(A)=\mu(B')$ such that 
\begin{gather*}
(A,B',S_{A,B'}(\tau))\text{ is $\eta/2$--strictly admissible
and $\eta^2/2$--bounded},
\\
\mathcal{D}(A,B',S_{A,B'}(\tau))\leq \frac{1}{\eta}\mathcal{D}(A,B,C).
\end{gather*}
Moreover, if 
$\mu(C) \le (1-\tfrac{\eta}{50})\mu(B)$ then
\begin{equation} \label{eq:scriptseta}
\mu(S_{A,B'}(\tau)) \le \mu(A)-4\scriptd(A,B',S_{A,B'}(\tau))^{1/2}.
\end{equation}
\end{lemma}

\begin{proof} The set $B'$ is constructed via an iterative process, in the course of
which $B$ is recursively replaced by successively smaller sets $B_j$, 
finally arriving at a set $B'$ with the same Haar measure as $A$. 
The quantity $\mathcal{D}'(A,B_j,\tau)$ is controlled by induction on $j$,
yielding control of $\mathcal{D}'(A,B',\tau)$. 

Before starting this process, 
recall that $C$ is essentially equal to $S_{A,B}(\tau)$ in the sense that
\begin{gather}
\label{eq:C almost level set}
\big|\mu\big(S_{A,B}(\tau)\big)-\mu(C)\big|\leq 2\mathcal{D}(A,B,C)^{1/2},
\\
\label{eq:control at step 1}
\scriptd(A,B,S_{A,B}(\tau))
\le \scriptd(A,B,C),\\
\label{eq:control at step 1 for D'}
\scriptd'(A,B,\tau) \le2\scriptd(A,B,C),
\end{gather}
with these inequalities justified 
by Lemmas~\ref{lemma:Staulowerbound} and \ref{lemma:726}.

The following lemma will be useful. 

\begin{lemma}\label{continuous overlap} 
Let $B$ be a 
measurable subset of $G$.
For any $t\in [\mu(B)^2,\mu(B)]$, there exists $x_t\in B$ satisfying $\mu\big(B\cap (x_t+B)\big)=t$. 
\end{lemma}

This is a direct consequence of the connectivity of $G$,
since $x\mapsto \mu(B\cap (B+x))$ is a continuous function from $G$ to $\reals$.
\qed

Iteratively invoking Lemma \ref{continuous overlap}, a nested sequence of subsets of $B$ will be constructed; the last set in the sequence will be the desired $B'$. The properties of this sequence are described in the following Claim, the proof of which is postponed until after the proof of Lemma \ref{reducing to equal}.

\begin{claim} \label{claim:algorithm} There exists a nested sequence $B=:B_0\supseteq B_1\supseteq B_2\supseteq\ldots\supseteq B_J$ of subsets of $G$, with 
$$\mu(B_J)=\mu(A),
$$
\begin{equation}\label{eq:next defect}\mathcal{D}'(A,B_j,\tau)\leq 2\mathcal{D}'(A,B_{j-1},\tau)\text{ for each }j\leq J,
\end{equation}
$$2^J\leq \frac{2}{\eta^2}.
$$
\end{claim}

It follows that
$$\mathcal{D}'\big(A,B',\tau\big)\leq 2^J\cdot 
2 \mathcal{D}(A,B,C)\leq \frac{4}{\eta^2}\mathcal{D}(A,B,C), $$
whence
\begin{equation}\label{eq:control on D' last step} 
\mathcal{D}'\big(A,B',\tau\big)^{1/2}\leq \tfrac{1}{200}\eta\mu(B)  
\end{equation}
by the hypothesis on $\scriptd(A,B,C)$.

We claim that $(A,B',\tau)$ satisfies the hypotheses of Corollary~\ref{cor:KTsharpened}.
Firstly,
$\tau = \tfrac{1}{2}(\mu(A)+\mu(B)-\mu(C))\le 
\min(\mu(A),\mu(B'))=\mu(A)$ 
is equivalent to $\mu(B)\le \mu(A)+\mu(C)$, which 
holds since $(A,B,C)$ is admissible.
Secondly,
the superlevel set $S_{A,B'}(\tau)$ satisfies
$|\mu(A)-\mu(B')|\le \mu(S_{A,B'}(\tau))$, since 
$\mu(A)-\mu(B')=0$.
Thirdly, $\mu(A)+\mu(B')\le 1+\tau = 1 + \tfrac12(\mu(A)+\mu(B)-\mu(C))$
is equivalent to 
$ \mu(A) + \mu(C) + (2\mu(B')-\mu(B))\le 2$, which holds since $\mu(A)+\mu(B)+\mu(C)\le 2$ and $\mu(B')\le\mu(B)$. Fourthly, $\mu(A)+\mu(B')+\mu(S_{A,B'}(\tau))\leq 2$, as $\mu(A)=\mu(B')\leq \tfrac{1}{2}$.

Invoking Corollary~\ref{cor:KTsharpened} for the triple $(A,B',S_{A,B'}(\tau))$ gives
\begin{equation}\label{eq:control of last level set}
|\mu(S_{A,B'}(\tau))-(\mu(A)+\mu(B')-2\tau)|\leq 2\mathcal{D}'(A,B',\tau)^{1/2}.
\end{equation}
Since 
$\mu(A)+\mu(B')-2\tau
= \mu(A)+\mu(C)-\mu(B)$
is $\ge \eta\mu(B)\ge\eta\mu(A)$ by the $\eta$--strict admissibility hypothesis, while also $2\tau\geq \mu(B)$, it follows from \eqref{eq:control on D' last step} 
that $(A,B',S_{A,B'}(\tau))$ is $\eta/2$-strictly admissible
and satisfies the estimates
$\mu(A)+\mu(B')+\mu(S_{A,B'}(\tau))\le 2-\tfrac12\eta$
and 
$\min(\mu(A),\mu(B'),\mu(S_{A,B'}(\tau)))\ge \eta^2/2$.

Moreover,
if $\mu(C)\le \mu(B)- \tfrac1{50}\eta\mu(B)$ then
\begin{eqnarray*}
\begin{aligned}
\mu(S_{A,B'}(\tau))&\leq \mu(A)+\mu(B')-2\tau+2\mathcal{D}'(A,B',\tau)^{1/2}\\
&= \mu(A)+\mu(C)-\mu(B) + 2\mathcal{D}'(A,B',\tau)^{1/2}\\
&\leq \mu(A)- \big(\mu(B)-\mu(C)\big)+\tfrac{1}{100}\eta\mu(B) \\
& \le \mu(A)-\tfrac1{50}\eta\mu(B) + \tfrac1{100} \eta \mu(B).
\end{aligned}
\end{eqnarray*}
Therefore
$\mu(S_{A,B'}(\tau)) \le \mu(A)-4\scriptd(A,B',\tau)^{1/2}$,
establishing together with Lemma \ref{returning to level set} the final assertion of Lemma~\ref{reducing to equal}.
\end{proof}

\begin{proof}[Proof of Claim \ref{claim:algorithm}] The sets $B_j$ will be constructed by an iterative use of Lemma \ref{continuous overlap}, in such a way that Proposition \ref{submodularity} can be invoked to control each $\mathcal{D}'(A,B_j,\tau)$. More precisely, for each $j=1,\ldots,J$, define 
$$B_j:=B_{j-1}\cap (x_j+B_j),
$$
with $x_j\in G$ chosen to ensure that
\begin{equation}\label{eq:exact size}\mu(B_j)=\mu(B_{j-1})-b_j
\end{equation}
for appropriate quantities $b_j\in [0,\mu(B_{j-1})-\mu(B_{j-1})^2]$ that will be specified later (such $x_j$ exists by Lemma \ref{continuous overlap}), where $J$ is defined as the smallest non-negative integer such that $\mu(B_J)=\mu(A)$. (The quantities $b_j$ will be such that such $J$ will exist.)

Now, the $b_j\in [0,\mu(B_{j-1})-\mu(B_{j-1})^2]$ are chosen so that $\mu(B_j)\geq \mu(A)$ for all $j$, i.e.
\begin{equation}\label{eq:simplest bound}b_j\leq \mu(B_{j-1})-\mu(A),
\end{equation}
and so that Proposition \ref{submodularity} can be applied for $(A,B_j,\tau)$ and $\big(A,B_{j-1}\cup (x_j+B_{j-1}),\tau\big)$, to deduce \eqref{eq:next defect}. To that end, for each $j$ the estimate
$$\mu(A)+\mu\big(B_{j-1}\cup (x_j+B_{j-1})\big)\leq 1+\tau
$$
should hold, i.e. $\mu(A)+\mu(B_{j-1})+b_j\leq 1+\tau$ for all $j$. By \eqref{eq:exact size}, this is equivalent to
\begin{equation*}
\left\{ \begin{aligned}
&\mu(A)+\mu(B)+b_1\leq 1+\tau\text{ (for }j=1\text{)},
\\
&\mu(A)+\mu(B)-(b_1+b_2+\ldots+b_{j-1})+b_j\leq 1+\tau\text{ for }j\geq 2,
\\
\end{aligned} \right.
\end{equation*}
that is
\begin{equation}\label{eq:recursive submodularity}
\left\{ \begin{aligned}
&b_1\leq d,
\\
&b_j-(b_1+b_2+\ldots+b_{j-1})\leq d\text{ for }j\geq 2,
\\
\end{aligned} \right.
\end{equation}
where $d:=\tfrac12\big(2-\mu(A)-\mu(B)-\mu(C))\big).$ 

Therefore, it suffices to find $b_j\in [0,\mu(B_{j-1})-\mu(B_{j-1})^2]$ that satisfy \eqref{eq:simplest bound}, that are small enough for \eqref{eq:recursive submodularity} to hold, but also large enough for $\mu(B_J)=\mu(A)$ to hold for some $J$ with $2^J\leq \tfrac{2}{\eta^2}$.

Observe that, if not for the condition $b_j\in [0,\mu(B_{j-1})-\mu(B_{j-1})^2]$, the quantities $b_j=2^jd$ for all $j=1,\ldots,J-1$ and $b_J=\mu(B_{J-1})-\mu(A)$, where $J$ is the smallest positive integer with $\mu(B)-d-2d-\ldots-2^Jd<\mu(A)$, would work as they satisfy \eqref{eq:simplest bound} and \eqref{eq:recursive submodularity}, while also $2^J\leq \tfrac{2}{\eta^2}$.

In order to achieve the additional condition $b_j\in [0,\mu(B_{j-1})-\mu(B_{j-1})^2]$, more care needs to be taken. For simplicity, once $B_{j-1}$ has been defined, denote
$$m_j:=\min\big(\mu(B_{j-1})-\mu(B_{j-1})^2,\mu(B_{j-1})-\mu(A)\big).
$$
Define
$$b_j:=2^jd\text{ for all }j=1,\ldots,J_1-1,
$$
where $J_1$ is the smallest non-negative integer $j$ such that $2^jd>m_j$. Observe that the so far defined $b_j$ satisfy the required conditions.

If $2^{J_1}d\leq\mu(B_{J_1-1})-\mu(B_{J_1-1})^2$, then $2^{J_1}d>\mu(B_{J_1-1})-\mu(A)$, so $\mu(B_{J_1-1})-\mu(A)\in [0,  \mu(B_{J_1-1})-\mu(B_{J_1-1})^2]$. In this case, define $b_{J_1}:=\mu(B_{J_1-1})-\mu(A)$ and terminate the process. The $b_j$ satisfy all the required conditions.

Otherwise, $2^{J_1}d>\mu(B_{J_1-1})-\mu(B_{J_1-1})^2$. Define
$$b_j:=m_j\text{ for all }J=J_1+1,\ldots, \bar{J}_2-1,
$$
where $\bar{J}_2$ is the smallest integer larger than $J_1$ with $2^{J_1}d\leq m_{\bar{J}_2}$, having terminated the process at the smallest $j$ along the way for which $m_j=0$, if such a $j$ exists. Observe that the so far defined $b_j$ satisfy the required conditions.

If the process has not been terminated, define
$$b_j:=2^{J_1+j-\bar{J}_2}d\text{ for all }j=\bar{J}_2,\ldots, J_2-1,
$$
where $J_2$ is the smallest integer $j$ larger than $\bar{J}_2$ with $2^{J_1+j-\bar{J}_2}d>m_j$. The so far defined $b_j$ satisfy the required conditions.

Now, working as above, if $2^{J_1+J_2-\bar{J}_2}d\leq\mu(B_{J_1-1})-\mu(B_{J_1-1})^2$ define $b_{J_2}:=\mu(B_{J_2-1})-\mu(A)$ and terminate the process. Otherwise, define
$$b_j:=m_j\text{ for all }J=J_2+1,\ldots, \bar{J}_3-1,
$$
where $\bar{J}_3$ is the smallest integer larger than $J_2$ with $2^{J_1+J_2-\bar{J}_2}d\leq m_2$, having terminated the process at the smallest $j$ along the way for which $m_j=0$, if such a $j$ exists. Continuing this way, one definitely finds $J\in \mathbb{N}$ with $\mu(B_J)=\mu(A)$; that is when the process terminates. The $b_j$ satisfy \eqref{eq:simplest bound} and \eqref{eq:recursive submodularity}. Therefore, it remains to show that $2^J\leq \tfrac{2}{\eta^2}$.

Indeed, $b_1+\ldots +b_J=\mu(B)-\mu(A)$. Now, let $\mathcal{M}$ be the set of $j$ for which $b_j=m_j$, and $\mathcal{M}':=\{1,\ldots,J\}\setminus \mathcal{M}$. On the one hand,
$$\sum_{j\in\mathcal{M}'}b_j=d+2d+2^2d+\ldots +2^{m'}d\geq 2^{m'}d,
$$
where $m'=\#\mathcal{M}'$. Therefore, $2^{m'}d\leq \mu(B)-\mu(A)$, so
$$2^{m'}\leq \tfrac{1}{\eta}.
$$
On the other hand, $m$ equals at most the number of consecutive intervals of the form $[c^2,c]$ needed to cover $[\mu(A),\mu(B)]$ (with the right-most interval being $[\mu(B)^2,\mu(B)]$). This in turn equals the smallest positive integer $k$ with $\mu(B)^{2^k}\leq \mu(A)$. Since $\mu(B)^{2^{k-1}}\geq \mu(A)$, it follows that
$$2^m\leq 2\tfrac{\ln\big(\tfrac{1}{\mu(A)}\big)}{\ln\big(\tfrac{1}{\mu(B)}\big)}\leq 2\tfrac{\ln \big(\tfrac{1}{\eta}\big)}{\ln 2}\leq\tfrac{2}{\eta}.
$$
So, $2^J=2^{m+m'}\leq \frac{2}{\eta^2}$.

\end{proof}

The next lemma will be used to deduce properties of more general triples
from properties of triples that satisfy the hypotheses of Lemma~\ref{reducing to equal}.

\begin{lemma} \label{for not strongly admissible} Let $(A,B,C)$ be $\eta$-strictly admissible 
and $\eta$--bounded and satisfy
\begin{gather*}
\mu(C)\leq \mu(A)\leq \mu(B),
\\
\mu(A)\leq\tfrac{1}{2},
\\
\mathcal{D}(A,B,C)^{1/2}\leq \tfrac{1}{800} \eta\mu(B).
\end{gather*}
Define $\tau$ by $\mu(B)=\mu(A)+\mu(C)-2\tau$.
If 
$\mu(C) > (1-\tfrac{\eta}{50})\mu(B)$
then there exist measurable sets $C'\subseteq C$ and $A'\subseteq A$ that satisfy
\begin{equation*}
\left\{ \begin{aligned}
&(S_{C',A}(\tau),C',A)\text{ is $\eta/4$--strictly admissible
and $\eta/4$--bounded}
\\
&\mathcal{D}(S_{C',A}(\tau),C',A)\leq 16\mathcal{D}(C,B,A)
\\
&\mu(C')=\mu(A')=\mu(C)-\tfrac{1}{10} \eta\mu(B),
\end{aligned} \right. 
\end{equation*}
while 
\begin{equation*}
\left\{ \begin{aligned}
&(S_{C',A'}(\tau),C',A')\text{ is $\eta/2$--strictly admissible and $\eta/2$--bounded}
\\
&\mathcal{D}(S_{C',A'}(\tau),C',A')\leq 16\mathcal{D}(C,B,A)
\\
&\mu(S_{A',C'}(\tau)) \le (1-\tfrac{\eta/2}{50}) \mu(C').
\end{aligned} \right.
\end{equation*}
\end{lemma}

\begin{proof} 
Define $\tau = \tfrac12 (\mu(A)+\mu(C)-\mu(B))$.
Then $\tau\ge \tfrac12 \eta \mu(B)\ge \tfrac12 \eta^2$ by the $\eta$--strict admissibility hypothesis,
while $\tau \le \tfrac12\mu(C)\le \tfrac14$ since $\mu(B)\ge\mu(A)$.

Since $(A,B,C)$ is $\eta$-strictly admissible and $\mathcal{D}(A,B,C)$ is small relative to $\eta\mu(B)$, 
Lemma~\ref{lemma:Staulowerbound} gives
\begin{equation}\label{eq:B almost level set}\big|\mu\big(S_{C,A}(\tau)\big)-\mu(B)\big|\leq 2\mathcal{D}(A,B,C)^{1/2},
\end{equation}
whence $(C,A,S_{C,A}(\tau))$ is $\tfrac{1}{2}\eta$-strictly admissible. Lemma~\ref{lemma:Staulowerbound} also gives
\begin{equation}\label{eq:control of D}\mathcal{D}(C,A,S_{C,A}(\tau))\leq \mathcal{D}(A,B,C).
\end{equation}
By Lemma~\ref{lemma:726},
\begin{equation*}
\mathcal{D}'(C,A,\tau)
\leq 2\mathcal{D}(A,B,C).
\end{equation*}

Now, there exist $x_C,x_A\in G$ such that 
$C':=C\cap (x_C +C)$ and
$A':=A\cap (x_A +A)$ satisfy
\begin{equation*} \left\{
\begin{aligned}
\mu(C')&=\mu(C)-\tfrac{\eta}{10}\mu(B) \ \in [\mu(C)^2,\mu(C)]
\\
\mu(A')&=\mu(C') \  \in [\mu(A)^2,\mu(A)].
\end{aligned} \right. \end{equation*}
(Observe that $\mu(C)-\tfrac{\eta}{10}\mu(B)\geq \mu(A)^2$ ($\geq \mu(C)^2$) because $\mu(A)\leq \tfrac{1}{2}$, thus $\mu(A)^2\leq \tfrac{1}{2}\mu(A)\leq \tfrac{1}{2}\mu(B)$; combining this with the lower bound assumption on $\mu(C)$, one obtains $\mu(C)-\mu(A)^2\geq (1-\tfrac{\eta}{50}-\tfrac{1}{2})\mu(B)\geq \tfrac{\eta}{10}\mu(B)$.)

It holds that
$$0\leq\tau\leq\mu(C')=\min\big\{\mu(C'),\mu(A)\big\}=\min\big\{\mu(C'),\mu(A')\big\}
$$
and
$$\mu(C')+\mu(A\cup A')-\tau\leq \mu(A)+\mu(C\cup C')-\tau<1
$$
(as $2\frac{\eta}{10}\mu(B)<2-(\mu(A)+\mu(B)+\mu(C))$).
Therefore,
\begin{eqnarray*}
\begin{aligned}
0\leq\mathcal{D}'(C',A',\tau)\leq 2\mathcal{D}'(C',A,\tau)&\leq 4\mathcal{D}'(C,A,\tau)\leq 8\mathcal{D}(A,B,C)
\end{aligned}
\end{eqnarray*}
by the submodularity principle, Proposition~\ref{submodularity}.

We apply Corollary~\ref{cor:KTsharpened}
to the triple $(A',C',\tau)$.  
Its hypotheses are satisfied.
First, $0\leq \tau\leq\min(\mu(C'),\mu(A'))=\mu(C')$; also, $\mu(A')+\mu(C')<1+\tau$ holds,
since $\mu(C')=\mu(A')\le \mu(A)\le\tfrac12$ while $\tau>0$.
Second, $\mu(S_{A',C'}(\tau)) \ge 0 = |\mu(A')-\mu(C')|$.
Third,
$\mu(A')+\mu(C')+\mu(S_{A',C'}(\tau))\le 2$
because $\mu(A')=\mu(C')\le\mu(A)\le\tfrac12$
while $\mu(S_{A',C'}(\tau))\le 1$.
Therefore the Corollary may be applied to obtain
\begin{eqnarray}
\begin{aligned}\label{eq:final level set}
|\mu(S_{C',A'}(\tau))-(\mu(A')+\mu(C')-2\tau)|
\leq 2\mathcal{D}'(C',A',\tau)^{\frac{1}{2}}
\leq \tfrac{\eta}{100}\mu(B).
\end{aligned}
\end{eqnarray}

We next show that
$(S_{C',A'}(\tau),C',A')$ is $\frac{\eta}{2}$-strictly admissible. Inserting the definition of $\tau$ into \eqref{eq:final level set} gives
\begin{eqnarray*}
\begin{aligned}
\mu(S_{C',A'}(\tau))&\leq \mu(B)-\big(\mu(A)-\mu(A')\big)-\big(\mu(C)-\mu(C')\big)+\tfrac{\eta}{100}\mu(B)\\
&\leq\mu(C)+\tfrac{\eta}{50}\mu(B)-2\cdot\tfrac{\eta}{10}\mu(B)+\tfrac{\eta}{100}\mu(B)\\
&\leq \mu(C')-\tfrac{\eta}{50}\mu(B)
\\
&\leq (1-\tfrac{\eta}{50})\mu(C').
\end{aligned}
\end{eqnarray*}
Note that the last of the three conclusions stated for $(A',C',S_{A',B'}(\tau))$
has been verified.

On the other hand,
\begin{eqnarray}
\begin{aligned}\label{eq:smaller by definite amount}
\mu(S_{C',A'}(\tau))&\geq \mu(B)-\big(\mu(A)-\mu(A')\big)-\big(\mu(C)-\mu(C')\big)-\tfrac{\eta}{100}\mu(B)\\
&\geq\mu(B)-\left(\tfrac{\eta}{10}\mu(B)+\tfrac{\eta}{100}\mu(B)\right)-\tfrac{\eta}{10}\mu(B)-\tfrac{\eta}{100}\mu(B)\\
&\geq \mu(B)-\tfrac{\eta}{4}\mu(B)
\\
&\geq \mu(C')-\tfrac{\eta}{4}\mu(B)\\
&>\left(1-\tfrac{\eta}{50}-\tfrac{\eta}{4}\right)\mu(B)\\
&>\tfrac{\eta}{2}\mu(B).
\end{aligned}
\end{eqnarray}
Since $\mu(A')=\mu(C')$ and $\mu(B)\ge \max(\mu(A'),\mu(C'),\mu(S_{A',C'}(\tau)))$,
the triple $(A',C',S_{A',C'}(\tau))$ is $\eta/2$--strictly admissible.




\medskip
We claim next that the intermediate triple $(S_{C',A}(\tau),C',A)$ 
is $\tfrac{\eta}{4}$-strictly admissible. 
Indeed, since $A'\subseteq A$ and $C'\subseteq C$, 
$$\mu(S_{C',A'}(\tau))\leq \mu(S_{C',A}(\tau))\leq \mu(S_{C,A}(\tau)),$$
whence, by \eqref{eq:B almost level set} and one of the inequalities in \eqref{eq:smaller by definite amount},
\begin{eqnarray*}
\begin{aligned}
\mu(B)-\tfrac{\eta}{4}\mu(B)\leq \mu(S_{C',A}(\tau))
\leq \mu(B)+2\mathcal{D}(A,B,C)^{1/2}
\leq\mu(B)+\tfrac{\eta}{400}\mu(B).
\end{aligned}
\end{eqnarray*}
Therefore, $\tfrac{\eta}{4}$-strict admissibility follows from the $\eta$--strict admissibility
of $(A,B,C)$ and the inequalities
$|\mu(C')-\mu(C)| \le \tfrac{\eta}{10}\mu(B)$
and $|\mu(A)-\mu(B)| \le \frac{\eta}{50}\mu(B)$.

\medskip
Finally, the $\eta/2$--boundedness  of $(A',C',S_{A',C'}(\tau))$
and
$\eta/4$--boundedness  of $(A,C',S_{A,C'}(\tau))$
follow from estimates shown above.
\end{proof}



\section{Relaxation} \label{section:relax}

For function $g_j: G\to[0,1]$,
define $g_j^\tarstar:\torus\to[0,\infty)$ 
to be the indicator function of the interval centered at $0$
whose Lebesgue measure is equal to $\int_G g_j\,d\mu$.
Define $\bg^\tarstar = (g_1^\tarstar,g_2^\tarstar,g_3^\tarstar)$.
Assuming that $g_j$ takes values in $[0,1]$ for each index $j$,
we say that $\bg$ is $\eta$--strictly admissible
if the triple $(\int_G g_j\,d\mu: 1\le j\le 3)$ is $\eta$--strictly admissible.

With these notations, Theorem~\ref{thm:relaxed} can be equivalently stated
as the inequality
\begin{equation}
\scriptt_G(\bg) \le \scriptt_\torus(\bg^\tarstar)
\ \text{ for all functions $g_j:G\to[0,1]$.}
\end{equation}

\begin{notation}
For any ordered triple $\bE$ of measurable subsets of $G$, define
\begin{equation} 
\overline{\scriptd}(\bE)=\scriptt_G(\bE^\star)-\scriptt_\torus(\bE).  
\end{equation}
More generally, for $g:G\to[0,1]$, define
\begin{equation}
\label{scriptdbardefn}
 \overline{\scriptd}(\bg) = \scriptt_\torus(\bg^\tarstar)-\scriptt_G(\bg),  
\end{equation}
and for $\bg = (g_j: j\in\{1,2,3\})$, define $\bg^\tarstar = (g_j^\star: j\in\{1,2,3\})$.
\end{notation}

Then 
\[ \scriptd(A,B,C)=\overline{\scriptd}(A,B,-C)\] 
for any ordered triple $(A,B,C)$ of measurable subsets of $G$.
That is,
\[
\langle \one_{\astar}*\one_{\bstar},\one_{\cstar}\rangle_\torus
- \langle \one_A*\one_B,\one_C\rangle_G
= \langle \one_{\astar}*\one_{\bstar},\one_{(-C)^\star}\rangle_\torus
- \langle \one_A*\one_B,\one_{-C}\rangle_G.
\]
Theorem~\ref{thm:relaxed} can again  be restated as $\overline{\scriptd}(\bE)\ge 0$
for every triple $\bE$.

The function $h:\torus\to[0,\infty)$ is said to be symmetric if $h(-x)=h(x)$ for all $x\in\torus$.
If $h$ is symmetric, $h$ is said to be nonincreasing if its restriction to $[0,\tfrac12]\subset\torus$
is nonincreasing, under the usual identification of $\torus$ with $[-\tfrac12,\tfrac12]$. 

\begin{lemma} \label{hardy-littlewood type} 
Let 
$f_1,f_2,f_3:\mathbb{T}\rightarrow \mathbb{R}$ be symmetric, nonincreasing functions satisfying $0\leq f_1,f_2,f_3\leq 1$. Let
$I\subset\torus$ be the interval centered at $0$ of length $|I|=\int_\torus f_1\, dm$.
Then
$$\mathcal{T}_\torus(f_1,f_2,f_3)\leq \mathcal{T}_\torus(\one_I,f_2,f_3). $$
\end{lemma}

\begin{proof} Defining $F$ by $f_1=\one_I+F$, one has
\begin{equation}\label{eq:properties of F}F\leq 0\text{ on }I\text{, }F\geq 0\text{ on }\mathbb{T}\setminus I\text{ and }\textstyle\int_\torus F\,dm=0.
\end{equation}
Since
$$\mathcal{T}_\torus(f_1,f_2,f_3)=\langle f_1, f_2*f_3\rangle_\torus =\langle \one_I,f_2*f_3\rangle_\torus +\langle F, f_2*f_3\rangle_\torus,
$$
it suffices to show that 
$\langle F, f_2*f_3\rangle_\torus\leq 0$.
Now, since $f_2, f_3$ are symmetric, non-increasing and non-negative, each can be approximated by a superposition of indicator functions of intervals centered at 0. Therefore, it suffices to show that
$\langle F, \one_J*\one_K\rangle_\torus\leq 0 $
for all intervals $J, K$ centered at 0. This is in fact trivially true, due to \eqref{eq:properties of F} and the fact that $\one_J*\one_K$ is symmetric, non-increasing and non-negative. Indeed,
\begin{eqnarray*}
\begin{aligned}\langle F, \one_J*\one_K\rangle_\torus&=\int_I \one_J*\one_K\cdot F\,dm + \int_{\mathbb{T}\setminus I} \one_J*\one_K\cdot  F\,dm\\
&\leq \int_I \left(\inf_I\one_J*\one_K\right) F\,dm+ \int_{\mathbb{T}\setminus I} \left(\sup_{\mathbb{T}\setminus I}\one_J*\one_K\right) F\,dm\\
&= c\int_I F\,dm+ c\int_{\mathbb{T}\setminus I} F\,dm=c\int F\,dm=0,
\end{aligned}
\end{eqnarray*}
where $c:=\one_J*\one_K\left(\tfrac{m(I)}{2}\right)$.
\end{proof}

\begin{proof}[Proof of Theorem~\ref{thm:relaxed}]
By expressing each of $f,g,h$ as a superposition of indicator functions
and invoking Theorem~\ref{thm:RS}, we deduce that
\begin{equation}
\langle f*g,h\rangle_G \le \langle f^\star*g^\star,h^\star\rangle_\torus.
\end{equation}
Express $h^\star$ as a superposition $\int_0^1 \one_{D(t)}\,dt$
where each $D(t)\subset\torus$ is an interval centered at $0$.
According to Lemma~\ref{hardy-littlewood type},
\begin{equation}
\langle f^\star,g^\star,\one_{D}\rangle_\torus
\le 
\langle \one_\astar*\one_\bstar,\one_D\rangle_\torus
\end{equation}
for any interval $D$ centered at $0$.
Integrating with respect to $t\in[0,1]$ yields 
\begin{equation}
\langle f^\star*g^\star,h^\star\rangle_\torus
\le
\langle \one_\astar*\one_\bstar,h^\star\rangle_\torus.
\end{equation}
A repetition of this reasoning gives
\begin{equation}
\langle \one_\astar*\one_\bstar,h^\star\rangle_\torus
\le
\langle \one_\astar*\one_\bstar,\one_\cstar\rangle_\torus.
\end{equation}
\end{proof}


\section{The perturbative Riesz-Sobolev regime} \label{section:RSperturbative}

In this section, we prove the following lemma,
in which the sets in question are assumed to be moderately well approximated
by appropriately related rank one Bohr sets,
and are proved to be better approximated if 
$\scriptd(\bE) = \scriptt_\torus(\bE^\star)-\scriptt_G(\bE)$ is sufficiently small.
The analysis is adapted from \cite{christRSult}.
 
\begin{lemma} \label{lemma:perturbative} For each $\eta,\eta'>0$ there exist $\delta_0>0$ 
and $\bC<\infty$ with the following property. 
Let $\bE = (E_1,E_2,E_3)$ be an $\eta$--strictly admissible triple 
of measurable subsets of $G$ satisfying
\begin{equation} \mu(E_1)+\mu(E_2)+\mu(E_3)\le 2-\eta'.  \end{equation}
Suppose that there exists a compatibly centered parallel ordered triple 
$\bB = (\scriptb_1,\scriptb_2,\sB_3)$
of  rank one Bohr sets $\scriptb_j\subset G$
satisfying $\mu(\scriptb_j)=\mu(E_j)$ and
\begin{equation}
\max_j \mu(E_j\symdif \scriptb_j) \le \delta_0 \max_k \mu(E_k).
\end{equation}
Then there exists $\by$ satisfying $y_1+y_2=y_3$ such that
\begin{equation}
\max_j \mu(E_j\symdif (\scriptb_j+y_j)) \le \bC\scriptd(\bE)^{1/2}.
\end{equation}
\end{lemma}

Since $0<\mu(\scriptb_j)<1=\mu(G)$, the homomorphism $\phi$ does not
vanish identically.


\begin{definition} An ordered triple $(\scriptb_1,\scriptb_2,\scriptb_3)$ of rank one Bohr sets is $\scriptt_G$-compatibly centered if $(\scriptb_1,\scriptb_2,-\scriptb_3)$ is compatibly centered.
\end{definition}

All of our discussion of the Riesz-Sobolev inequality
can be rephrased in terms of $\scriptt_G$ since
\begin{equation}
\scriptt_G(\bE) = \langle \one_{E_1}*\one_{E_2},\one_{-E_3}\rangle
\end{equation}
and $\mu(-E_3)=\mu(E_3)$.
Theorem~\ref{thm:RS} thus states that
\begin{equation} \scriptt_G(\bE) \le \scriptt_\torus(\bE^\star) \end{equation}
for all triples $\bE$ of measurable subsets of $G$.
Another equivalent formulation is
$\scriptt_G(\bE)\le \scriptt_G(\bB)$
for any $\scriptt_G$--compatibly centered ordered triple $\bB$
of parallel rank one Bohr sets satisfying $\mu(E_j)=\mu(\scriptb_j)$
for each $j\in\{1,2,3\}$;
the right-hand side equals $\scriptt_\torus(\bE^\star)$
for any such triple $\bB$.



Lemma~\ref{lemma:perturbative} can thus be equivalently formulated as follows.
\begin{lemma} \label{lemma:perturbative'} 
For each $\eta,\eta'>0$ there exist $\delta_0>0$ and $\bC<\infty$ 
with the following property. 
Let $\bE = (E_1,E_2,E_3)$ be an $\eta$--strictly admissible triple 
of measurable subsets of $G$ satisfying
\begin{equation} \mu(E_1)+\mu(E_2)+\mu(E_3)\le 2-\eta'.  \end{equation}
Suppose that there exists a $\scriptt_G$-compatibly centered parallel ordered triple $\bB = (\scriptb_1,\scriptb_2,\sB_3)$
of  rank one Bohr sets $\scriptb_j\subset G$
satisfying $\mu(\scriptb_j)=\mu(E_j)$ and
\begin{equation}
\max_j \mu(E_j\symdif \scriptb_j) \le \delta_0 \max_k \mu(E_k).
\end{equation}
Then there exists $\by$ satisfying $y_1+y_2+y_3=0$ such that
\begin{equation}
\max_j \mu(E_j\symdif (\scriptb_j+y_j)) \le \bC\ovscriptd(\bE)^{1/2}.
\end{equation}
\end{lemma}


\begin{remark}
One aspect of the conclusion may be unanticipated.
Suppose that $\bE,\bB$ satisfy the hypotheses,
and that $\overline{\scriptd}(\bE)$ vanishes.
Then the conclusion is not only  that $\bE$ is equivalent
to some ordered triple of Bohr sets,
but that it is equivalent to a translate of $\bB$. 
A consequence is that for any $\eta>0$,
there exists $\eps>0$ with this property: If $B,B'$ are rank one Bohr sets
satisfying $\eta \le \mu(B)=\mu(B')\le 1-\eta$,
and if $\mu(B\symdif B')<\eps$,
then $\mu(B\symdif B')=0$.
There is no surprise in this consequence,
but it is interesting that it is implicit in the lemma.
To deduce it, assume without loss of generality that $B,B'$ are centered at $0$,
that is, $B = \{x: \|\phi(x)\|_\mathbb{T}\le r\}$ for some homomorphism $\phi$ and $2r\in[\eta,1-\eta]$,
and likewise for $B'$ with respect to a homomorphism $\phi'$.
Set $\bB = (B,B,B)$ and $\bE = (B',B',B')$.
The hypotheses of Lemma~\ref{lemma:perturbative} are satisfied, if $\eps$ is sufficiently small.
Moreover,
$\overline{\scriptd}(\bE)=0$; any $\scriptt$--compatibly centered
parallel family of rank one Bohr sets saturates the Riesz-Sobolev inequality.
The conclusion of the lemma is that 
$B'$ differs from some translate of $B$ by a $\mu$--null set.
\qed
\end{remark}

We will prove Lemma~\ref{lemma:perturbative'} in the more general relaxed framework,
in which indicator functions of sets are replaced by functions
taking values in $[0,1]$.
In the remainder of \S\ref{section:RSperturbative},
we study triples $\bg = (g_j: j\in\{1,2,3\})$ with $g_j: G\to[0,1]$.

For functions $g:G\to[0,1]$,
define $g^\tarstar:\torus\to[0,\infty)$ 
to be the indicator function of the interval centered at $0\in\torus$
whose Lebesgue measure is equal to $\int_G g\,d\mu$.
For triples $\bg$, define $\bg^\tarstar = (g_1^\tarstar,g_2^\tarstar,g_3^\tarstar)$.
Recall the notation
$\overline{\scriptd}(\bg) = \scriptt_\torus(\bg^\tarstar)-\scriptt_G(\bg)$  
introduced in \eqref{scriptdbardefn}.
Assuming that $g_j$ takes values in $[0,1]$ for each index $j$,
we say that $\bg$ is $\eta$--strictly admissible
if the triple $(\int_G g_j\,d\mu: 1\le j\le 3)$ of positive scalars
is $\eta$--strictly admissible.

The next lemma generalizes Lemma~\ref{lemma:perturbative'} to the relaxed
framework. The remainder of this section will be devoted to its proof.

\begin{lemma} \label{lemma:perturbative_relaxed} 
For each $\eta,\eta'>0$ there exist $\delta_0>0$ and $\bC<\infty$ 
with the following property. 
Let $\bg$ be an $\eta$--strictly admissible triple 
of measurable functions $g_j:G\to[0,1]$ satisfying
\begin{equation} \sum_{j=1}^3 \int g_j\,d\mu \le 2-\eta'.  \end{equation}
Suppose that there exists a $\scriptt_G$-compatibly centered parallel ordered triple 
$\bB = (\scriptb_1,\scriptb_2,\sB_3)$ of  rank one Bohr sets $\scriptb_j\subset G$
satisfying $\mu(\scriptb_j)= \int g_j\,d\mu$ and
\begin{equation}
\max_j \norm{g_j -\one_{\scriptb_j}}_{L^1(G)}  \le \delta_0 \max_k \int g_k\,d\mu.
\end{equation}
Then there exists $\by\in G^3$ satisfying $y_1+y_2+y_3=0$ such that
\begin{equation}
\max_j \norm{g_j - \one_{\scriptb_j+y_j}}_{L^1(G)} \le \bC\ovscriptd(\bg)^{1/2}.
\end{equation}
\end{lemma}


Define the orbit $\orbit(\bA)$ of the triple $\bA$ of subsets of $G$ 
to be the set of all triples
$\bA+\by=(A_j+y_j: j\in\{1,2,3\})$ with 
$\by\in G^3$ satisfying $y_1+y_2+y_3=0$.
For $g_j:G\to[0,1]$ and $\bB = (\scriptb_j: 1\le j\le 3)$
satisfying $\mu(\scriptb_j) = \int g_j\,d\mu$,
define
\begin{equation}
\distance(\bg,\orbit(\bB)) 
= \inf_\by \max_{j\in\{1,2,3\}} \norm{ g_j - \one_{\scriptb_j+y_j}}_{L^1(G)}, 
\end{equation}
with the infimum taken over all
$\by\in G^3$ satisfying $y_1+y_2+y_3=0$.
With these definitions, Lemma~\ref{lemma:perturbative_relaxed} states that
if $\bB,\bg$ satisfy its hypotheses then
\begin{equation} \label{distanceformulation}
\distance(\bg,\orbit(\bB)) \le \bC\overline{\scriptd}(\bg)^{1/2}.
\end{equation}

We use $c$ to denote a strictly positive constant
that depends only on $\eta$, but whose value is permitted to change
from one occurrence to the next.
We write $\langle f,g\rangle = \int_G fg\,d\mu$
for functions $f,g:G\to\reals$.

\begin{proof}[Proof of Lemma~\ref{lemma:perturbative_relaxed}]
Set
\begin{equation} \delta = \distance(\bg,\orbit(\bB)). \end{equation}
Choose $\bz$ satisfying $z_1+z_2+z_3=0$ so that 
\begin{equation} \label{maximizerchoice}  
\max_j \norm{g_j - \one_{\scriptb_j+z_j}}_{L^1(G)} = \delta. \end{equation}
Such a minimizing $\bz$ must exist,
since $\norm{g_j - \one_{\scriptb_j+z_j}}_{L^1(G)}$ is a continuous
function of $\bz$ with compact domain. 
If $\delta=0$ then the conclusion of the lemma certainly holds,
so we may assume for the remainder of the proof that $\delta>0$.

The hypotheses and conclusion of the lemma are invariant under translation of each $g_j$ by $u_j\in G$,
with $\sum_j u_j=0$. 
By means of such a transformation, we may assume without loss of generality that
$\scriptb_j=\{x\in G: \norm{\phi(x)}_{\torus}\le r_j\}$, 
with $\phi:G\rightarrow \torus$ a continuous homomorphism independent of $j$,
and each $z_j=0$.
Here, $0<r_j = \tfrac12\mu(\scriptb_j) \le \tfrac12 (1-\tilde\eta)$
with $\tilde\eta=\tilde\eta(\eta,\eta')>0$.

Define functions $f_j$ by
\begin{equation}
g_j = \one_{\scriptb_j} + f_j.
\end{equation}
These functions take values in $[-1,1]$, and satisfy $\int_G f_j\,d\mu=0$.
Moreover,
$\max_{k\in\{1,2,3\}} \norm{f_k}_{L^1} =\delta$ by \eqref{maximizerchoice},
$f_k\le 0$ in $\scriptb_k$, and $f_k\ge 0$ in $G\setminus \scriptb_k$.

Regard $\phi$ as a (discontinuous) mapping from $G$ to $(-\tfrac12,\tfrac12]$
by identifying $\torus$ with $(-\tfrac12,\tfrac12]$ in the usual way.
For each $k\in\{1,2,3\}$,
write $\{1,2,3\} = \{i,j,k\}$ and define
\[ K_k(x) = \one_{\scriptb_i}*\one_{\scriptb_j}(x) \ \text{ for $x\in G$.}\]
$K_k$ is continuous and nonnegative.
There exists $\gamma_k>0$ such that 
$K_k(x)>\gamma_k$ if $|\phi(x)|<\tfrac12\mu(\scriptb_k)$,
$K_k(x)<\gamma_k$ if $|\phi(x)|>\tfrac12\mu(\scriptb_k)$,
and
$K_k(x)=\gamma_k$ when $|\phi(x)|=\tfrac12\mu(\scriptb_k)$.
The $\eta$--strict admissibility hypothesis 
implies that there exists a small positive constant $c>0$,
depending only on $\eta$, such that
\begin{equation} \label{eq:8.17} 
\begin{cases}
& |K_k(x)-\gamma_k|  
= \big|\,|\phi(x)|- \tfrac12 \mu(B_k)\,\big|
\ \ \text{whenever $\big|\,|\phi(x)|- \tfrac12 \mu(B_k)\,\big| \le c\mu(B_k)$,}
\\
& |K_k(x)-\gamma_k|  
\ge c\mu(B_k) \text{ otherwise.}
\end{cases} \end{equation}

Let $\lambda$ be a large positive constant, to be chosen below.
There exist a decomposition
\begin{equation} f_j =  f_j^\dagger + \tilde f_j \end{equation}
and consequently an expansion $g_j = \one_{\scriptb_j} + f_j^\dagger + \tilde f_j$,
with the following properties:
\begin{gather} 
\int f_j^\dagger\,d\mu = \int  \tilde f_j\,d\mu=0
\\
\text{$\tilde f_j,f_j^\dagger \ge 0$ on $G\setminus \scriptb_j$}
\\
\text{$\tilde f_j,f_j^\dagger \le 0$ on $\scriptb_j$}
\\
\text{If
$\big|\,|\phi(x)| -\tfrac12\mu(\scriptb_j)\,\big|\ge\lambda\delta$
then $f_j^\dagger(x)=0$.}
\\
\norm{\tilde f_j}_{L^1}
\le 2 \int_{ \big|\,|\phi(x)| -\tfrac12\mu(\scriptb_j)\,\big|\ge\lambda\delta }
|f_j(x)|\,d\mu(x). \label{l1}
\end{gather}

To achieve this, set $\tilde f_j(x)=f_j(x)$ whenever
$\big|\,|\phi(x)| -\tfrac12\mu(\scriptb_j)\,\big|\ge\lambda\delta$.
We do not simply set $\tilde f_j(x)\equiv 0$
otherwise (even though such an $\tilde f_j$ clearly satisfies the desired
condition \eqref{l1} above), 
because the vanishing condition
$\int \tilde f_j\,d\mu=0$ will be essential below.
Instead, for $x\in G$ satisfying
$\big|\,|\phi(x)| -\tfrac12\mu(\scriptb_j)\,\big|<\lambda\delta$,
we define $\tilde f_j(x) = f_j(x)\one_S(x)$ with the set $S$ chosen as follows.

If $\int_{\big|\,|\phi(x)| -\tfrac12\mu(\scriptb_j)\,\big|\geq\lambda\delta} f_j\,d\mu \ge 0$, then
$S\subset\scriptb_j$, and $S$ is chosen so that $\int \tilde f_j\,d\mu=0$.
Such a subset exists because
$\int f_j\,d\mu=0$,
$f_j \ge 0$ on $G\setminus\scriptb_j$ and $\le 0$ on $\scriptb_j$,
and $\mu$ is nonatomic.
For our purpose, any such set $S$ suffices.

If $\int_{\big|\,|\phi(x)| -\tfrac12\mu(\scriptb_j)\,\big|\geq\lambda\delta} f_j\,d\mu < 0$,
then instead choose $S\subset G\setminus \scriptb_j$ 
to ensure that $\int \tilde f_j\,d\mu=0$.
In both cases, define $f_j^\dagger = f_j-\tilde f_j$.
The resulting functions $\tilde f_j,f_j^\dagger$ enjoy
all of the required properties.

Set $g_j^\dagger = \one_{\scriptb_j}+f_j^\dagger$.
These functions satisfy 
$g_j= g_j^\dagger + \tilde f_j$,
$0\le g_j^\dagger\le 1$,
$-1\le \tilde f_j,f_j^\dagger\le 1$,
and (since $\int \tilde f_j=0$) $\int g_j^\dagger = \int g_j$.

Define
\begin{equation}
\tilde\delta = \max_{j} \norm{\tilde f_j}_{L^1(G)} \le\delta.
\end{equation}
$\scriptt=\scriptt_G$ satisfies 
\begin{equation}
\label{eq:trivialbound}
|\scriptt(h_1,h_2,h_3)|\le \norm{h_1}_{L^1}\norm{h_2}_{L^1} \norm{h_3}_{L^\infty}
\end{equation}
for arbitrary functions, and is invariant
under permutation of $(h_1,h_2,h_3)$.
Using the assumption that $\norm{g_j}_{L^\infty}\le 1$,
and for each $k$ writing $\{1,2,3\} = \{i,j,k\}$ in some arbitrary manner,
it follows that
\begin{align*} 
\scriptt(\bg) 
&= \scriptt(g_1^\dagger + \tilde f_1, g_2^\dagger + \tilde f_2, g_3^\dagger + \tilde f_3)
\\&
= \scriptt(\bg^\dagger) 
+ \sum_{k=1}^3 \scriptt(g_i^\dagger,g_j^\dagger,\tilde f_k)
+ O(\tilde\delta^2)
\\&
= \scriptt(\bg^\dagger) 
+ \sum_{k=1}^3 \scriptt(\one_{\scriptb_i},\one_{\scriptb_j},\tilde f_k)
+ O(\tilde\delta \cdot \delta)
\\&
= \scriptt(\bg^\dagger) 
+ \sum_{k=1}^3 \langle \scriptk_k,\tilde f_k\rangle
+ O(\tilde\delta \cdot \delta).
\end{align*}
The constant implicit in the $O(\tilde\delta\cdot\delta)$
term is independent of the parameter $\lambda$.


Since $\int \tilde f_k\,d\mu=0$,
$\langle \scriptk_k,\tilde f_k\rangle = \langle \scriptk_k-\gamma_k,\tilde f_k\rangle$.
On the complement of $\scriptb_k$, $\tilde f_k\ge 0$ and $K_k-\gamma_k\le 0$;
on $\scriptb_k$, both signs are reversed.
Therefore
\begin{equation*} 
\langle K_k, \tilde f_k\rangle 
= \int (K_k-\gamma_k) \tilde f_k\,d\mu 
= -\int |K_k-\gamma_k|\cdot|\tilde f_k|\,d\mu
\le -c\lambda \delta \norm{\tilde f_k}_{L^1}
\end{equation*}
according to the properties \eqref{eq:8.17} of $\scriptk_k$
and the relation $\lambda\delta\le c\mu(\scriptb_k)$,
which holds, for any particular choice of large constant $\lambda$, 
by the smallness hypothesis on $\delta/\mu(\scriptb_k)$.
Therefore in all,
\[\scriptt(\bg)
\le \scriptt(\bg^\dagger) - c\lambda \delta\cdot \tilde\delta
+ O(\tilde\delta\cdot\delta)\]
with both $c$ and the implicit constant in the remainder term $O(\delta^2)$
independent of the parameter $\lambda$, but with $\tilde\delta$ dependent on $\lambda$.
Choosing $\lambda$ sufficiently large gives
\begin{equation}\label{pathsdiverge}
\scriptt(\bg) 
\le \scriptt(\bg^\dagger) - c\lambda \delta\cdot \tilde\delta
\le \min\big(
\scriptt(\bg^\dagger),
\scriptt_\torus(\bg^\tarstar) - c\lambda \delta\cdot \tilde\delta\big),
\end{equation}
with $c>0$ independent of $\lambda$,
and $\lambda$ independent of $\bg$.
We have used the bound $\scriptt(\bg^\dagger)\le \scriptt_\torus((\bg^\dagger)^\tarstar)$
of Theorem~\ref{thm:relaxed}, and the identity
$(\bg^\dagger)^\tarstar = \bg^\tarstar$.



There are now two cases, depending on the magnitude of $\tilde\delta/\delta$.
If $\tilde\delta\ge \tfrac12\delta$ then
$\scriptt(\bg) \le \scriptt_{\mathbb{T}}(\bg^\tarstar) - \tfrac12 c  \delta^2$.
This is the desired conclusion of Lemma~\ref{lemma:perturbative_relaxed}.

In the second case, $\tilde\delta\le \tfrac12\delta$.
From the triangle inequality in the form
\[ \max_j \norm{f_j^\dagger}_{L^1}
= \max_j \big( \norm{f_j}_{L^1} - \norm{\tilde f_j}_{L^1} \big)
\ge \delta-\tilde\delta \ge \tfrac12\delta,\]
it follows that
\[ \max_j \norm{g_j^\dagger - \one_{\scriptb_j}}_{L^1}
= \max_j \norm{f_j^\dagger}_{L^1} \ge \tfrac12\delta.\]
In this case, we use the alternative bound $\scriptt(\bg) \le \scriptt(\bg^\dagger)$
from \eqref{pathsdiverge}. 
Thus it suffices to prove that
\[ \scriptt(\bg^\dagger) 
\le \scriptt_{\mathbb{T}}(\bg^\tarstar) - c\max_j \norm{g_j^\dagger-\one_{\scriptb_j}}_{L^1}^2,\]
that is, to establish
the conclusion of Lemma~\ref{lemma:perturbative_relaxed} for $\bg^\dagger$.

The modified triple $\bg^\dagger$ satisfies all hypotheses of the lemma,
and enjoys the supplementary property that 
$g_j^\dagger - \one_{\scriptb_j} \equiv 0$ whenever 
$\big|\,|\phi(x)|-\tfrac12 \mu(\scriptb_j) \,\big| \ge \lambda \delta$. 

Moreover,
\begin{equation}
\tfrac12 \distance(\bg,\orbit(\bB))
\le \distance(\bg^\dagger,\orbit(\bB))
\le \tfrac32 \distance(\bg,\orbit(\bB))
\end{equation}
by the triangle inequality for $L^1(G)$ norms,
since $\tilde\delta\le \tfrac12 \delta$.
Therefore we have reduced matters to proving Lemma~\ref{lemma:perturbative_relaxed}
under the supplementary hypothesis 
that for every $j\in\{1,2,3\}$,
\begin{equation} \label{supplementary}
g_j- \one_{\scriptb_j} \equiv 0 \text{ whenever } 
\big|\,|\phi(x)|-\tfrac12 \mu(\scriptb_j) \,\big| \ge C_0 \delta.
\end{equation}
Here $C_0$ is some universal constant that is not at our disposal, but is dictated by
our choice of $\lambda$.
For the remainder of the proof of Lemma~\ref{lemma:perturbative_relaxed}
we drop the superscripts $\dagger$,
denoting by $\bg$ an ordered triple of functions that satisfies
the hypotheses of the lemma, as well as \eqref{supplementary} for $\delta$ and $\bB$ such that $\max_j\|g_j-\one_{\scriptb_j}\|_1\sim\delta$. Redefine $f_j = g_j-\one_{\scriptb_j}$.

The perturbative term $f_j$ satisfies \eqref{supplementary}, that is,
is supported where $\big|\,|\phi(x)|-\tfrac12 \mu(\scriptb_j)\,\big| \le C_0\delta$.
We claim that 
if $\eps_0$ is a sufficiently small constant multiple of $\eta \max_k \mu(\scriptb_k)$,
and if $0 < \delta\le\eps_0$, then this restriction on the support of $f_j$ ensures that
\begin{equation} \label{vanishing} \scriptt(f_1,f_2,f_3)=0.  \end{equation}
Indeed, $f_1*f_2$ is supported 
where $\phi$ differs by at most $2C_0 \delta$
from some quantity $(\pm \tfrac12 \mu(\scriptb_1)\pm \tfrac12 \mu(\scriptb_2))$,
while $f_3$ is supported where $\phi$ differs by at most $C_0 \delta$
from $\pm \tfrac12 \mu(\scriptb_3)$. 
The upper bound on $\mu(\scriptb_1)+\mu(\scriptb_2)+\mu(\scriptb_3)$ 
and the $\eta$--strict admissibility of $\bB$ ensure that
\[ \eta\max_j\mu(\scriptb_j)\le\big| \pm \mu(\scriptb_1)
\pm \mu(\scriptb_2) \pm \mu(\scriptb_3)\big|\leq 2-\eta'\] 
for all eight choices of signs,
yielding \eqref{vanishing} by the triangle inequality
since $\delta\le\eps_0$ is assumed to be small relative to $\eta\max_k\mu(\scriptb_k)$.

For any $\by=(y_1,y_2,y_3)\in\torus^3$ satisfying $y_1+y_2+y_3=0$,
these constructions can be applied to 
the triple $\bg^{\by}$ defined by replacing $g_j(x)$ by the translated
function $g_j^{y_j}(x) = g_j(x-y_j)$.  
Then $\scriptt(\bg)=\scriptt(\bg^\by)$, and $\int g_j^{y_j}\,d\mu = \int g_j\,d\mu$.
Assume that $|\phi(y_j)| =O(\delta)$ for all three indices $j$. Then 
\[ \max_j \norm{g_j^{y_j}-\one_{\scriptb_j}}_{L^1}
\le 
\max_j \norm{g_j^{y_j}-\one_{\scriptb_j^{y_j}}}_{L^1}
+ \max_j \norm{\one_{\scriptb_j^{y_j}}- \one_{\scriptb_j}}_{L^1}
= O(\delta).\]
On the other hand, 
\[ \max_j \norm{g_j^{y_j}-\one_{\scriptb_j}}_{L^1} \ge 
\distance(\bg^{\by},\orbit(\bB))
= \distance(\bg,\orbit(\bB))
\ge c\delta\]
by $\by$--translation invariance of the orbit
and translation invariance of $\mu$.
Each translated function $g_j^{y_j}-\one_{\scriptb_j}$ remains supported in
$\{x: \big|\,|\phi(x)|-\mu(E_j)/2\,\big| )\le O(\delta) \}$. 

Each $f_j=g_j-\one_{\scriptb_j}$ has a unique additive decomposition $f_j = f_j^+ + f_j^-$, with 
$f_j^\pm$ supported where $\big|\,\phi(x)  \mp \tfrac12 \mu(B_j)\,\big|=O(\delta)$, 
respectively.
It will be advantageous to work instead with $\bg^\by$, with $\by$ chosen so that the summands corresponding to $g_j^{y_j}-\one_{\scriptb_j}$ satisfy certain vanishing properties which the summands $f_j^\pm$ potentially lack. In particular, define functions $f_{j,y_j}^\pm$ 
by first setting $f_{j,y_j} = g_j^{y_j}-\one_{\scriptb_j}$,
and then expressing $f_{j,y_j} = f_{j,y_j}^+ + f_{j,y_j}^-$,
with $f_{j,y_j}^\pm$ supported where $|\phi(x)\mp \tfrac12\mu(\scriptb_j)| = O(\delta)$.

\begin{lemma}\label{lemma:paininneck}
For each index $j$, there exists $y_j\in G$ satisfying 
$|\phi(y_j)| \le C_0\delta$ and
\begin{equation} \label{fjplusvanishing} 
\int f_{j,y_j}^+\,d\mu = \int f_{j,y_j}^-\,d\mu =0.  \end{equation}
\end{lemma}


\begin{proof}
$f_{j,y}^+$ is that portion of $g_j^{y}-\one_{\scriptb_j}$
that is supported where $|\phi(x)-\tfrac12\mu(\scriptb_j)|$ is small.
Since $|\phi(y)| \le C_0\delta$
and $g_j(x)=\one_{\scriptb_j}(x)$ wherever 
$|\phi(x)-\tfrac12\mu(\scriptb_j)|>C_0\delta$,
$f_{j,y}^+$ is supported where
$|\phi(x)-\tfrac12\mu(\scriptb_j)|\le 2C_0\delta$.

Consider the function that maps $z\in [-C_0\delta,C_0\delta]$ to 
\[ \int f_{j,y}^+(x)\,d\mu(x) 
= \int_{|\phi(x)-\tfrac12\mu(\scriptb_j)| \le 2C_0\delta} 
\big(g_{j}^y - \one_{\scriptb_j}\big)(x)\,d\mu(x),\]
with $y= y(z)$ satisfying $\phi(y)=z$.
While $y$ is not uniquely determined by $z$ via this equation,
the integral nonetheless depends only on $z$.
Indeed, 
the contribution of the term $\one_{\scriptb_j}$ to the
integral does not involve $y$.
Substituting $x=u+y$ allows us to rewrite
the contribution of $g_{j}^y(x)=g_j(x-y)$ as
\begin{align*}  \int_{|\phi(x)- \tfrac12\mu(\scriptb_j)|\le 2C_0\delta}  g_{j}(x-y)\,d\mu(x)
=  \int_{|\phi(u)+z  - \tfrac12\mu(\scriptb_j)| \le 2C_0\delta}  g_{j}(u)\,d\mu(u)
\end{align*}
which likewise depends on $z$ alone.

This function of $z$ is nonnegative when $z=C_0\delta$.
Indeed, if $\phi(x) \in [\tfrac12\mu(\scriptb_j)-2C_0\delta,\tfrac12\mu(\scriptb_j)]$
then $g_j(x-y)=1$,
since $\phi_j(x-y) = \phi_j(x)-C_0\delta\le \tfrac12\mu(\scriptb_j)-C_0\delta$
and (by virtue of the reduction to the case $g_j=g_j^\dagger$ made above)
$g_j(u)\equiv 1$ when $\tfrac12\mu(\scriptb_j)-O(\delta) \le \phi(u)
\le \tfrac12\mu(\scriptb_j)-C_0\delta$. Thus $g_{j}^y(x)-\one_{\scriptb_j}(x)=1-1=0$ for these values of $x$.
On the other hand,
if $\phi(x) \in [\tfrac12\mu(\scriptb_j),\tfrac12\mu(\scriptb_j)+2C_0\delta]$
then $\one_{\scriptb_j}(x)=0$, so $g_j^y(x)-\one_{\scriptb_j}(x)\ge 0$.

The same reasoning shows that this function of $z$ is nonpositive when $z=-C_0\delta$. 
Therefore we may apply the Intermediate Value Theorem on $[-C_0\delta,C_0\delta]$
to conclude that there exists $y_j$ with $\phi(y_j)=z\in [-C_0\delta,C_0\delta]$ 
satisfying $\int f_{j,y_j}^+\,d\mu=0$.

It follows at once that
$\int f_{j,y_j}^-\,d\mu = \int f_{j,y_j}\,d\mu -\int f_{j,y_j}^+\,d\mu=0$. 
\end{proof}


Choose $y_1,y_2$ to ensure \eqref{fjplusvanishing} for $j=1,2$, 
but then define $y_3$ by $y_1+y_2+y_3=0$. 
With such a choice of $\by$ fixed henceforth,
simplify notation by suppressing $y_j$ and writing again $g_j,f_j,f_j^\pm$,
continuing to use the notation $\bg$
for this modified triple. The quantities $\scriptt(\bg)$
and $\distance(\bg,\orbit(\bB))$ are unchanged.

The functions $f_3^\pm$ need not have vanishing integrals. 
Nonetheless, 
\begin{equation} \label{2of3works} \int f_i^\pm*f_j^\pm\,d\mu=0 
\ \text{ for any distinct indices $i\ne j\in\{1,2,3\}$,}
\end{equation}
for all four possible choices of $\pm$ signs,
since $\int (f_i^\pm*f_j^\pm)\,d\mu = \int f_i^\pm\,d\mu\,\cdot\,\int f_j^\pm\,d\mu$ and
at least one of the two indices $i,j$ must belong to $\{1,2\}$.

Expand $\scriptt(\bg) = \scriptt(\one_{\scriptb_j}+f_j: j\in\{1,2,3\})$
into eight terms, using the multilinearity of $\scriptt$.
The simplest term is $\scriptt(f_1,f_2,f_3)$. 
Provided that $\delta$ is sufficiently small relative to $\max_j \mu(\scriptb_j)$,
with constant of proportionality depending on $\eta,\eta'$,
this term vanishes 
for the modified triple $\bg$, just as it was shown 
in \eqref{vanishing} to vanish for the original triple.

The vanishing of $\scriptt(f_1,f_2,f_3)$ 
simplifies the expansion of $\scriptt(\bg)$ to
\begin{equation} \label{eq:finalsum}
\scriptt(\bg) 
=\scriptt(\bB)
+ \sum_{k=1}^3 \langle K_k,f_k\rangle
+ \sum_{i<j} \langle \one_{\scriptb_l},\,f_i*f_j\rangle.
\end{equation}
In the final sum, $i<j\in\{1,2,3\}$
and $l$ is defined by $\{1,2,3\} = \{i,j,l\}$.

We next discuss the terms 
\begin{equation}
\langle K_k,f_k\rangle 
= -\int |f_k(x)|\,\big|\,|\phi(x)|- \tfrac12 \mu(\scriptb_k)\,\big|\,d\mu(x) \le 0. 
\end{equation}

There exist an absolute constant $c_0>0$ and $n\in\{1,2,3\}$ 
such that $\norm{f_n}_{L^1}\ge c_0\delta$.
Let $c_1 = \tfrac18 c_0$. 
Because $\norm{f_n}_{L^\infty} \le 1$ and 
\[ \mu(\{x\in G: \big|\,|\phi(x)|-\tfrac12 \mu(\scriptb_n)\,\big|\le c_1\delta\})
= 4c_1\delta,\]
necessarily \[\int_{ |\,|\phi(x)|-\tfrac12 \mu(\scriptb_n)\,|\ge c_1\delta}
|f_n|\,d\mu \ge  \norm{f_n}_{L^1}-4c_1\delta \ge \tfrac12 c_0\delta.\]
Therefore
\begin{equation}
\langle K_n,f_n\rangle 
\le 
- \int_{ |\,|\phi(x)|-\tfrac12 \mu(\scriptb_n)\,|\ge c_1\delta}
|f_n(x)| \cdot \big|\,|\phi(x)|- \tfrac12 \mu(\scriptb_n)\,\big|\,d\mu(x) 
\le -c_1\delta \cdot \tfrac12 c_0\delta,
\end{equation}
which is comparable to $\max_j \norm{f_j}_{L^1}^2$
and therefore to $\distance(\bg,\orbit(\bB))^2$.
Thus
\begin{equation} \sum_k \langle K_k,f_k\rangle \le -c' \delta^2.  \end{equation}

To complete the proof, we next show that 
\begin{equation}\label{quadtermsallvanish} \langle \one_{\scriptb_l},f_i*f_j\rangle=0
\ \text{ for any three distinct indices $i,j,l$.} \end{equation}
For any of the four possible choices of $\pm$ signs, 
the support of the convolution $f_i^\pm*f_j^\pm$ is 
contained in the sum of the supports of the two factors,
hence consists of points $x$ at which $\phi(x)
= \tfrac12 (\pm\mu(B_i) \pm \mu(B_j)) + O(\delta)$.
On the other hand,
$\scriptb_l$ is the set of $x$ satisfying $|\phi(x)|\le \tfrac12\mu(\scriptb_l)$,
and the $\eta$--strict admissibility hypothesis says that 
\[|\pm \mu(\scriptb_l)\pm \mu(\scriptb_i)\pm \mu(\scriptb_j)| \ge c\eta \max_k \mu(\scriptb_k).\]
A hypothesis of Lemma~\ref{lemma:perturbative_relaxed}
is that $\delta$ is small relative to $\eta \max_k \mu(\scriptb_k)$.
Therefore for any choice of $\pm$ signs, the support of  $f_i^\pm*f_j^\pm$
is either entirely contained in $\scriptb_l$, or entirely contained in its complement.
Therefore 
in the integral
\[ \langle \one_{\scriptb_l},f_i^\pm*f_j^\pm\rangle
= \int \one_{\scriptb_l}\cdot (f_i^\pm*f_j^\pm)\,d\mu,\]
the factor $\one_{\scriptb_1}$ is constant. Since $\int f_i^\pm*f_j^\pm\,d\mu=0$
by \eqref{2of3works},
this integral vanishes.
Summing over all four possible choices of signs gives \eqref{quadtermsallvanish}.

Inserting these results into the expansion \eqref{eq:finalsum}, we conclude that
when the supplementary hypothesis \eqref{supplementary} is satisfied,
$\scriptt(\bg) \le \scriptt(\bg^\tarstar)-c\delta^2$,
that is,
\begin{equation} \scriptt(\bg) \le \scriptt(\bg^\tarstar)-c\distance(\bE,\orbit(\bB))^2, \end{equation}
as was to be shown.
\end{proof}

\section{The perturbative regime for sumsets} \label{section:sumsetperturbative}

In this section we prove Theorem~\ref{thm:kneserstability},
the quantitative stability result for the inequality
$\mu_*(A+B)\ge \min(\mu(A)+\mu(B),\mu(G))$.

We begin with a small lemma needed in the analysis. 
\begin{lemma} \label{lemma:>onehalf}
Let $K$ be a compact Abelian group with Haar measure $\nu$.
Let $A,B\subset K$ be compact.
Suppose that $B\ne\emptyset$ and that $\nu(A)>\tfrac12\nu(K)$.
Then
\begin{equation}
\nu(A+B) \ge \min(\nu(B) + \tfrac12\nu(A),\nu(K)).
\end{equation}
\end{lemma}

$K$ is not assumed to be connected.
The conclusion is false in general, without the hypothesis
that $\nu(A)>\tfrac12\mu(K)$. It fails, for instance, if there exists
a subgroup $H$ of $K$ satisfying $\nu(H)=\tfrac12\nu(K)$ and $A=B=H$.

\begin{proof}
According to a theorem of Kneser \cite{kneser},
either $\nu(A+B)\ge \nu(A)+\nu(B)$
or there exists a subgroup $H$ of $K$
of positive Haar measure
satisfying 
$A+B+H=A+B$ and
$\nu(A+B) = \nu(A+H)+\nu(B+H)-\nu(H)$.
In the first case, the conclusion of the lemma holds.
In the second case,
if $\nu(H)=\nu(K)$ then 
$\nu(A+H)+\nu(B+H)-\nu(H) = \nu(K)+\nu(K)-\nu(K)=\nu(K)$
and again the conclusion holds.

Now, suppose that $\nu(H)<\nu(K)$. $A+H$ is a union of cosets of $H$. It cannot be a single coset,
for $\nu(H)<\nu(K)$ implies $\nu(H)\le\tfrac12\nu(K)<\nu(A)$.
Therefore $A+H$ is a union of at least two cosets of $H$,
so $\nu(A+H) \ge 2\nu(H)$, so 
\[ \nu(A+H)-\nu(H)\ge \tfrac12\nu(A+H) \ge\tfrac12 \nu(A).\]
Thus
\[\nu(A+B) = \nu(A+H) -\nu(H) + \nu(B+H)
\ge \tfrac12\nu(A) + \nu(B),\]
as was to be shown.
\end{proof}

Let $G$ be a compact Abelian group with Haar measure
$\mu$, satisfying $\mu(G)=1$. 
Let $|\scripta|$ denote the Lebesgue measure
of any set $\scripta\subset\torus$.

\begin{proposition} \label{prop:refineTao}
There exists  $\delta_0>0$ with the following property.
Let $A,B\subset G$ be compact sets of positive measures satisfying
\[\mu(A)+\mu(B)\le 1-200\delta_0\min(\mu(A),\mu(B)).\]
Suppose that
\begin{align*}
\norm{\phi(x)}_\torus &\le \tfrac12\mu(A)+\delta_0\min(\mu(A),\mu(B))\ \text{ for all $x\in A$},
\\
\norm{\phi(x)}_\torus &\le \tfrac12\mu(B)+\delta_0\min(\mu(A),\mu(B))\ \text{ for all $x\in B$},
\\
\mu(A+B) &\le \mu(A)+\mu(B) +\delta \min(\mu(A),\mu(B))\text{ for some }0<\delta\leq \delta_0.
\end{align*}
Then $\phi(A)$ is contained in some interval in $\torus$
of length $\mu(A)+100\delta\min(\mu(A),\mu(B))$.
Likewise for $B$.
\end{proposition}

Define $\scripta=\phi(A)$ and $\scriptb = \phi(B)$ in $\torus$.
For $t\in\torus$ define
\[ A_t = \{x\in A: \phi(x)=t\} \subset A\subset G.\]
$A_t$ will be regarded sometimes as a subset of a coset of $K=\kernel(\phi)$,
and sometimes as a subset of $K$ itself (by translating by any appropriate element of $G$).
Likewise define $B_t\subset B$.

Let $\nu$ be Haar measure on $H=\kernel(\phi)$, normalized to satisfy $\nu(H)=1$.

Each slice $\phi^{-1}(\{t\})\subset G$ is a coset of $H$.  
By translation, $\nu$ also defines a measure on each such coset,
which will also be denoted by $\nu$. Thus we may write
$\nu(A_t)$, even though there is no canonical identification
of $A_t$ with a subset of $H$. 

The hypotheses allow us to regard $\phi$ as a mapping from $A+B$ to $\reals$, 
rather than to $\torus$. Indeed, denoting $\eta:=\delta_0\min(\mu(A),\mu(B))$, each element of $\phi(a)\in \phi(A)$ is represented
by some element $\tilde\phi(a)\in [-\tfrac12\mu(A)-\eta,\tfrac12\mu(A)+\eta]$,
and correspondingly for $\phi(B)$. Therefore, for any $a\in A$ and $b\in B$,
$\phi(a+b)$ is represented by some element $\tilde\phi(a+b)\in (-\tfrac12,\tfrac12)$.
These satisfy $\tilde\phi(a+b)=\tilde\phi(a)+\tilde\phi(b)$,
where addition on the right-hand side is performed in $\reals$ rather than in $\torus$.
These three mappings, all denoted by the common symbol $\tilde\phi$,
are measure-preserving bijections.

\begin{proof}[Proof of Proposition~\ref{prop:refineTao}]
Define $\rho_A,\rho_B\in[0,1)$ by
\begin{equation} \label{rhoAB} 
1- \rho_A = \sup_t \nu(A_t) \text{ and } 1- \rho_B = \sup_s \nu(B_s).
\end{equation}
The hypothesis $\norm{\phi(x)}_\torus\le \tfrac12\mu(A)+\delta_0\min(\mu(A),\mu(B))$ for $x\in A$
implies that $|\scripta| \le(1+2\delta_0)\mu(A)$.
On the other hand, $\mu(A)\le (1-\rho_A)|\scripta|$.
Therefore $\rho_A \le 1-(1+2\delta_0)^{-1}$.
Thus if $\delta_0$ is sufficiently small then $\rho_A < \tfrac14$.
Likewise $\rho_B < \tfrac14$.
Therefore $\min(1-2\rho_A,\rho_B)=\rho_B$.
This relation will be used momentarily.

Let $\eps\in(0,\rho_A)$ be sufficiently small so that $1-\rho_A-\epsilon>\tfrac{3}{4}$, and choose $\tau\in\scripta$ satisfying 
\begin{equation} \nu(A_\tau)>1-\rho_A-\eps.\end{equation}
Set \[A_-=\{a\in A: \phi(a)<\tau\} \text{ and } A_+=\{a\in A: \phi(a)>\tau\}\]
(where $\mathcal{A}$ is seen as a subset of $\reals$ rather than of $\torus$).

Regarding $\scriptb$ as a subset of $\reals$,
let $b_-,b_+\in\reals$ be its minimum and maximum elements, respectively. 

Now \begin{equation*} 
A+B \supset (A_\tau+B) + (A_- +B_{b_-}) + (A_+ + B_{b_+})\end{equation*}
and these three sets are pairwise disjoint.
Therefore
\begin{equation*} 
\mu(A+B) \ge \mu(A_\tau+B) + \mu(A_- +B_{b_-}) + \mu(A_+ + B_{b_+}).\end{equation*}
$A_- + B_{b_-}$ contains a translate of $A_-$,
so $\mu(A_-+B_{b_-}) \ge \mu(A_-)$.
Likewise $\mu(A_+ + B_{b_+})\ge \mu(A_+)$.
Therefore
\begin{equation} \label{from3disjointness}
\mu(A+B) \ge \mu(A_\tau+B) + \mu(A). \end{equation}


One application of \eqref{from3disjointness} is the relation
\begin{equation} \max(\rho_A,\rho_B)\le\delta. \end{equation}
To prepare for its proof recall that according to Lemma~\ref{lemma:>onehalf},
\[ \nu(A_\tau+B_t) 
\ge \min\big(\tfrac12\nu(A_\tau) + \nu(B_t),1\big)
\ge \min\big(\nu(B_t) + \tfrac38,1\big)
\]
for any $t\in\phi(B)$, since $\nu(A_\tau) > \tfrac34$.
Therefore
\begin{multline*}
\mu(A_\tau+B)
= \int_\scriptb \nu(A_\tau+B_t)\,dt
\\
\ge \int_\scriptb \min(\nu(B_t)+\tfrac38,1)\,dt
= \mu(B) +  \int_\scriptb \big[ \min(\tfrac38,1-\nu(B_t))\big]\,dt
\\
\ge \mu(B) +  \int_\scriptb \big[ \min(\tfrac38,\rho_B)\big]\,dt
=  \mu(B) +   \rho_B|\scriptb|
\end{multline*}
since $\rho_B<\tfrac14$.
Since $|\scriptb|\ge \mu(B)$, 
inserting this bound into \eqref{from3disjointness} gives
\[ \mu(A+B) \ge \mu(A)+\mu(B) + \rho_B\mu(B).\]
Since $\mu(A+B)\le\mu(A)+\mu(B)+\delta\mu(B)$, 
we may conclude that $\rho_B\le\delta$.
The roles of $A,B$ can be interchanged, so $\rho_A\le\delta$ also.

Let $\scripta'=\{t\in\scripta: \nu(A_t)>\tfrac12\}$.
Likewise define $\scriptb'\subset\scriptb$.

We claim that
\begin{equation} \label{+14claim} 
\mu(A+B) \ge \mu(A)+\mu(B) + (\tfrac12-\rho_A) |\scriptb\setminus\scriptb'|.\end{equation}
The proof will use the fact that
for any subsets $S,T$ of a compact group $H$ satisfying $\mu(S)+\mu(T)>\mu(H)$, 
the associated sumset $S+T$ is all of $H$. 
Connectivity of $H$ is not required for this conclusion; it is valid
for the kernel $H$ of $\phi$.
Indeed, for any $z\in H$ it holds that $\{z-x: x\in S\}\cap T\ne\emptyset$,
since the intersection of these sets has measure equal to $\mu(S)+\mu(T)-\mu(H)>0$.
To prove the claim, majorize
\begin{equation} \label{again} \mu(A_\tau+B) 
\ge \int_\scriptb \nu(A_\tau+B_t) \,dt.
\end{equation}
One has $\nu(A_\tau+B_t)\ge \nu(B_t)$ for all $t$.
Moreover,
if $\nu(B_t) \le \tfrac12$ then \[\nu(A_\tau+B_t) \ge \nu(A_\tau)
\ge 1-\rho_A-\eps
\ge  \nu(B_t) + \tfrac12-\rho_A-\eps.  \]
Therefore
\begin{equation} \label{AtauB} 
\mu(A_\tau+B) 
\ge \int_\scriptb \nu(B_t)\,dt
 + \int_{\scriptb\setminus\scriptb'} (\tfrac12-\rho_A-\eps)\,dt
= \mu(B) +  (\tfrac12-\rho_A-\eps) |\scriptb\setminus \scriptb'|.  \end{equation} 
Letting $\eps\to 0$ and combining this with
\eqref{from3disjointness} gives \eqref{+14claim}. 
\qed

From \eqref{+14claim} together with the hypothesis 
$\mu(A+B)\le \mu(A)+\mu(B)+\delta\min(\mu(A),\mu(B))$
and the bound $\max(\rho_A,\rho_B)\le\delta$
we deduce that
\begin{equation}
\label{scriptbminusupper}
|\scriptb\setminus\scriptb'|\le (2+O(\delta))\delta\min(\mu(A),\mu(B)).\end{equation}
Since the roles of $A,B$ can be freely interchanged in this reasoning, 
$|\scripta\setminus\scripta'|$ satisfies the same inequality.

For every $s\in\scripta'$ and $t\in \scriptb'$,
$\nu(A_s+B_t) \ge \min(\nu(A_s)+\nu(B_t),1) =1$ since $\nu(A_s)>\tfrac12$
and likewise $\nu(B_t)> \tfrac12$.
Therefore 
$\nu((A+B)_x)=1$ for every $x\in\scripta'+\scriptb'$. Therefore
$|\scripta'+\scriptb'| \le \mu(A+B)$, and consequently
\begin{align*} |\scripta'+\scriptb'| 
&\le \mu(A)+\mu(B)+\delta\min(\mu(A),\mu(B))
\\
&\le |\scripta|+|\scriptb|+\delta\min(\mu(A),\mu(B))
\\
&< |\scripta'|+|\scriptb'|
+ 6\delta\min(\mu(A),\mu(B)).
\end{align*}
On the other hand,
\begin{align*}
|\scripta'|
&\ge|\scripta| -(2+O(\delta))\delta\min(\mu(A),\mu(B))
\\
&\ge\mu(A) -(2+O(\delta))\delta\min(\mu(A),\mu(B))
\\
&> \mu(A)-3\delta\min(\mu(A),\mu(B))\end{align*}
and likewise for $|\scriptb'|$.

A straightforward adaptation to $\reals$ (see \cite{christRS}) of a theorem of \freiman{}
states that if $S,S'\subset\reals$ are nonempty Lebesgue measurable
sets satisfying $|S+S'|_* < |S|+|S'|+ \min(|S|,|S'|)$,
then $S$ is contained in an interval of length $\le |S+S'|-|S'|$.
Regarding $\scripta',\scriptb'$ as subsets of $\reals$, as we may, 
this result allows us to conclude that if $\delta_0$ is less than some absolute constant, then $\scripta'$
is contained in an interval $I$ of length $\le |\scripta'|+(6+O(\delta))\delta\min(\mu(A),\mu(B))$.
Similarly, $\scriptb'$ is contained in an interval $J$
of length $\le |\scriptb'|+(6+O(\delta))\delta\min(\mu(A),\mu(B))$.

The following claim completes the proof of Proposition~\ref{prop:refineTao}.

\begin{claim} \label{claim:fullsets}
The full sets $\scripta=\phi(A)$ and $\scriptb=\phi(B)$
are contained in intervals of lengths
$\mu(A)+100\delta\min(\mu(A),\mu(B))$
and
$\mu(B)+100\delta\min(\mu(A),\mu(B))$, respectively.
\end{claim}

The reasoning in the following proof of this claim will be used again below.

\begin{proof}[Proof of Claim~\ref{claim:fullsets}]
Suppose that some point $z\in\scripta$ 
were to lie to the left of the left endpoint of $I$
by a distance $\ge C_1\delta\min(\mu(A),\mu(B))$.
If $y\in\scriptb'$ lies within distance $C_1\delta\min(\mu(A),\mu(B))$
of the left endpoint of $J$, then $A_z + B_y$ lies outside $I+J$.
The set of all $y\in\scriptb'$ with this property 
has Lebesgue measure 
\[ \ge |\scriptb'| - \big(|J|-C_1\delta\min(\mu(A),\mu(B))\big) \ge (C_1-6)\delta\min(\mu(A),\mu(B)).\]
The sum of $A_z$ with the union of all such $B_y$ 
therefore has Haar  measure 
$\ge \tfrac12 (C_1-6)\delta\min(\mu(A),\mu(B))$.
This sumset is disjoint from $\phi^{-1}(\scripta'+\scriptb')
= \phi^{-1}(\scripta') + \phi^{-1}(\scriptb')$.
Therefore
\begin{align*} \mu(A+B) 
&\ge \mu(\phi^{-1}(\scripta'))
+ \mu(\phi^{-1}(\scriptb'))
+ \tfrac12 (C_1-6)\delta\min(\mu(A),\mu(B))
\\&
\ge(\mu(A)-|\scripta\setminus\scripta'|) 
+ (\mu(B)-|\scriptb\setminus\scriptb'|) 
+ \tfrac12 (C_1-6)\delta\min(\mu(A),\mu(B))
\\&
\ge \mu(A)+\mu(B) + \tfrac12 (C_1-18)\delta\min(\mu(A),\mu(B)).
\end{align*}
Choosing $C_1=21$ yields a contradiction for all sufficiently small $\delta$.


Thus $\scripta=\phi(A)$ is contained in an interval of length less than
\[|I|+ 40\delta\min(\mu(A),\mu(B))
\le |\scripta'|+46\delta\min(\mu(A),\mu(B))
\le \mu(A) + 100\delta\min(\mu(A),\mu(B)).\]
Likewise for $\scriptb$.
\end{proof}

The conclusions of Proposition~\ref{prop:refineTao} hold if $A,B$ satisfy the same
hypotheses but are merely assumed to be measurable, rather
than compact, except that the constant $200$ is replaced by a sufficiently large finite constant $\boldC$.
To prove this, choose compact subsets $A',B'$ of $A,B$ whose Haar measures are nearly
those of $A,B$ respectively, and invoke Proposition~\ref{prop:refineTao} to obtain
parallel rank one Bohr sets $\scriptb_{A'}\supset A'$ and $\scriptb_{B'}\supset B'$
satisfying $\mu(\scriptb_{A'})\le \mu(A)+\boldC\delta\min(\mu(A),\mu(B))$
with the corresponding bound for $\mu(\scriptb_{B'})$.
Then repeat the reasoning in the proof of the claim above
to deduce that there exist slightly larger parallel rank one Bohr sets,
associated to the same homomorphism $\phi$, which contain all of $A,B$ respectively,
and whose measures satisfy the required upper bounds with a larger constant factor $\boldC$. 
\end{proof}

\medskip
With a small modifications, the proof of Proposition~\ref{prop:refineTao}
establishes an extension: Setting $M = \min(\mu(A),\mu(B))$ to simplify notation, the hypotheses that 
$\norm{\phi(x)}_\torus \le \tfrac12\mu(A)+\delta_0M$ for all $x\in A$
and analogously for $B$
can be relaxed to
\begin{equation}
\left\{ \begin{aligned}
& \norm{\phi(x)}_\torus \le \tfrac12\mu(A)+\delta_0M 
\ \text{ $\forall\, x\in A$ outside a set of Haar measure $\le\delta_0M$}
\\
&\norm{\phi(x)}_\torus \le \tfrac12\mu(B)+\delta_0M
\ \text{ $\forall\,x\in B$ outside a set of Haar measure $\le\delta_0M$,}
\end{aligned} \right. \end{equation}
to conclude that, provided that $\delta_0$ is sufficiently small, $\phi(A)$ is contained in some interval in $\torus$
of length $\mu(A)+\bC\delta_0\min(\mu(A),\mu(B))$, and likewise for $\phi(B)$.

To prove this, define $A',B'$ to be the subsets of $A,B$, respectively, specified by these inequalities. 
We will use the proof of Claim~\ref{claim:fullsets} to control $A,B$ in terms of $A',B'$,
demonstrating that $A,B$ satisfy the hypotheses of Proposition~\ref{prop:refineTao}
with $\delta_0$ replaced by $\eps_0$, where $\eps_0$ depends only on $\delta_0$
and tends to zero as $\delta_0\to 0$. Thus the extension will be proved.

The only change to the reasoning in the proof of the claim  is that 
we may no longer conclude that, in the notation of that discussion, 
$A_z+B_y$ is disjoint from
\[I+J=\{x\in\torus: \norm{x}_\torus \le \tfrac12(\mu(A)+\mu(B))+O(\delta_0)M\}\]
whenever $\phi(z)\in[\tfrac12\mu(A)+2\bC\delta_0M,\tfrac12]$
and $\phi(y)\in[\tfrac12\mu(B)-\bC\delta_0M,\tfrac12\mu(B)]$.
Under the hypotheses of this extension, 
it is not permissible to regard $\phi(A),\phi(B)$ as subsets of $\reals$,
and the desired disjointness could fail due to periodicity.

Instead, we claim that if  $C_1$ is a sufficiently large constant 
and $\delta_0$ is sufficiently small then
for any $x\in A$ satisfying
\[ \tfrac12 \mu(A) + C_1\delta_0 M
\le |\phi(x) |\le\tfrac12,\]
the set of all $y\in B'$ satisfying $\phi(x+y)\notin I+J$
has Haar measure $\ge C_2\delta_0 M$,
where $C_2$ depends on $C_1$ but not on $\delta_0$, 
and $C_2\to\infty$ as $C_1\to\infty$.
We may assume without loss of generality that $\phi(x)\in[0,\tfrac12]$ 
by replacing $(A,B)$ by $(-A,-B)$ if necessary. 
The two desired conditions for $y$ are that $w=\phi(y)$ should satisfy
\[ w\ge \tfrac12 \mu(B)+\tfrac12\mu(A)-\phi(x) +O(\delta_0 M)\]
and 
\[ w \le 1-\tfrac12\mu(A)-\tfrac12\mu(B) -\phi(x) -O(\delta_0 M).\]
The  set of all $w\in [-\tfrac12\mu(B),\tfrac12\mu(B)]$
that satisfy both inequalities has Lebesgue measure
$\ge \bC_1\delta_0 M$
provided that $\delta_0$ is sufficiently small. 
The inverse image under $\phi$ of this set of elements $w$ has Haar measure
$\ge \bC_1\delta_0 M$.
The complement of the intersection with $B$ of 
this inverse image 
has Haar measure $O(\delta_0 M)$, with the constant in the $O(\cdot)$ notation
independent of the choice of $\bC_1$. The result therefore follows.

This completes the proof of the extension of Proposition~\ref{prop:refineTao}.
\qed

\medskip

\begin{proof}[Proof of Theorem~\ref{thm:kneserstability}]
Let $\eta>0$.
Let $A,B\subset G$ 
satisfy $\min(\mu(A),\mu(B))\ge\eta$ and $\mu(A)+\mu(B)\le 1-\eta$.
Suppose that $\mu_*(A+B)\le \mu(A)+\mu(B)+\delta\min(\mu(A),\mu(B))$.
By the same reasoning as the one used to extend the statement of Proposition \ref{prop:refineTao} to measurable sets,
it suffices to treat the case in which $A,B$ are compact.

If $\delta$ is sufficiently small, as a function of $\eta$ alone,
then the theorems of Tao \cite{taokneser} and/or Griesmer \cite{griesmer}
can be applied. The conclusion is that there exist parallel rank one Bohr sets
$\scriptb_A,\scriptb_B$ such that 
\[\mu(A\symdif \scriptb_A)\le \eps(\delta)\min(\mu(A),\mu(B))\]
and likewise for $\mu(B\symdif \scriptb_B)$. The quantity $\eps(\delta)$
tends to $0$ as $\delta\to 0$, provided that $\eta$ remains fixed.

The reasoning in the proof of the claim above now shows that the full sets $A,B$
are contained in parallel rank one Bohr sets $\scriptb^\sharp_A,\scriptb^\sharp_B$,
respectively,
satisfying 
\[\mu(\scriptb^\sharp_A)\le\mu(A)+\eps^\sharp\min(\mu(A),\mu(B))\]
where $\eps^\sharp\to 0$ as $\delta\to 0$.
Likewise for $B,\scriptb^\sharp_B$.

This is not the desired conclusion, since it includes no quantitative bound
for the dependence of $\eps^\sharp$ on $\delta$.
However, since $\eps^\sharp\to 0$ as $\delta\to 0$, 
it follows that if $\delta$ is sufficiently small then the pair $(A,B)$ 
satisfies the hypotheses of Proposition~\ref{prop:refineTao}.
Invoking that proposition completes the proof of the theorem.
\end{proof}

\section{A special case on $\torus$} \label{section:specialcase}

In this section, we discuss our functionals for $G=\torus$,
in the special situation in which one of the sets is an interval. 
In particular, our next result
ensures that, if $(A,B,C)$ satisfies near equality in the Riesz-Sobolev inequality on $\mathbb{T}$, and $C$ is an interval, then $A$ and $B$ are nearly intervals.

When discussing the special case $G=\torus$, we will often use
$|E|$, rather than $m(E)$, to denote the Lebesgue measure of $E$.

\begin{proposition} \label{prop:longname}
Let $\eta>0$.
There exists a constant $\bC<\infty$, depending only on $\eta$, 
with the following property.
Let $(A,B,C)$ be an 
$\eta$--strictly admissible and $\eta$--bounded
triple of measurable subsets of $\torus$.
Suppose that $C$ is an interval with center $x_C$.
Then
\begin{equation}
\label{eq:special case characterisation}
\inf_{x+y=x_C}\big(|A\symdif (A^\star+x)|+|B\symdif (B^\star+y)|\big)\le \bC \scriptd(A,B,C)^{1/2}.
\end{equation}
\end{proposition}

We outline here a proof based on a method relying on reflection
symmetry and a two-point inequality of Baernstein and Taylor \cite{baernsteintaylor}.
This technique does not otherwise appear in this paper.
It is also used by O'Neill \cite{oneill_sphere}
to analyze the corresponding issue for the sphere $S^d$, $d\ge 2$.

\begin{proof} 
The proof will consist of three steps.

\medskip
\noindent{Step 1.}\   
\textit{If $\scriptd(A,B,C)=0$ and $C$ is an interval, 
then $A,B$ differ from  intervals by Lebesgue null sets,
and these three intervals are compatibly centered.}

Assume without loss of generality that $C$ is centered at 0. Thus  $C=C^*$. By the complementation principle described in \S\ref{section:twokey}, it may also be assumed that $m(A)\leq \tfrac{1}{2}$, $m(B)\leq \tfrac{1}{2}$.

Identify $\torus$ with the unit circle in $\complex\leftrightarrow \reals^2$ 
via the mapping $x\mapsto e^{2\pi i (x+ \tfrac\pi{2})}$.
For each $x=(x_1,x_2)\in\torus$  let $R(x)=(x_1,-x_2)$ be the reflection
of $x$ about the horizontal axis.
To any $E\subset\torus$ associate $E^\sharp\subset\torus$,
defined as follows. For each pair of points $\{x,R(x)\}$
with $x=(x_1,x_2)$ with $x_2\ne 0$, let $x_+=(x_1,|x_2|)$ 
and $x_-=(x_1,-|x_2|)$.
If both $x_+,x_-\in E$ then both $x_+,x_-\in E^\sharp$;
if neither belongs to $E$ then neither belongs to $E^\sharp$;
and if exactly one belongs to $E$ then $x_+\in E^\sharp$
and $x_-\notin E^\sharp$.
If $x_2=0$ then $x\in E^\sharp$ if and only if $x\in E$.

Define $\torus_+=\{x=(x_1,x_2)\in\torus: x_2>0\}$.
For $y\in\torus$ define $R_y E = (E+y)^\sharp$,
where addition is in the additive group $\torus=\reals/\integers$.

Assume without loss of generality that the interval $C$ is centered at $0$.
The following hold: for any measurable sets $A,B\subset\torus$,
\begin{gather}
m(A^\sharp)=m(A),\; m(B^\sharp)=m(B),\tag{a}\\
m(A^\sharp \symdif B^\sharp) \le m(A \symdif B), \tag{b} 
\\
\langle \one_{A^\sharp} * \one_C,\one_{B^\sharp}\rangle
\ge \langle \one_{A} * \one_C,\one_{B}\rangle.
\tag{c} \end{gather}
Consequently the above conclusions hold with $A^\sharp,B^\sharp$
replaced by $R_yA, R_yB$, respectively,
for any $y\in\torus$.
(a) and (b) are direct consequences of the definition of the $\sharp$ operation,
while (c) is an almost equally direct consequence \cite{baernsteintaylor}.

Observe that
\begin{equation}\langle \one_{A^\sharp} * \one_C,\one_{B^\sharp}\rangle
\ > \  \langle \one_{A} * \one_C,\one_{B}\rangle
\tag{d} \end{equation}
if the set of all points $(x_+,y_+)\in \torus_+^2$ satisfying
$x_+\in A$, $x_-\notin A$, $y_+\notin B$, $y_-\in B$
and
$\norm{x_+-y_+}_\torus < \tfrac12m(C)$
and
$\norm{x_+-y_-}_\torus > \tfrac12m(C)$
has positive Lebesgue measure in $\torus^2$. The same holds if the set of all points $(x_+,y_+)\in \torus_+^2$ satisfying
$x_+\notin A$, $x_-\in A$, $y_+\in B$, $y_-\notin B$
and the above two inequalities has positive Lebesgue measure in $\torus^2$.

Moreover,
if $A\subset\torus$ is a finite union of closed intervals then there exists
a finite sequence $y_1,\dots,y_N$ of elements of $\torus$ such that
\begin{equation}
R_{y_N}R_{y_{N-1}}\cdots R_1A = A^\star.
\tag{e} \end{equation}
This is elementary, and its proof is left to the reader.

If $A\subset\torus$ is Lebesgue measurable then there exists
an infinite sequence $y_n\in\torus$ such that
\begin{equation}
\lim_{N\to\infty} m\big(R_{y_N}R_{y_{N_1}}\cdots R_1 A\symdif A^\star \big) = 0;
\tag{f} \end{equation}
(f) follows by combining (e) with the contraction property (b).

Consider any pair of measurable sets $A,B\subset\torus$ that satisfy
$\langle \one_A*\one_C,\one_B\rangle = \langle \one_\astar*\one_C,\one_\bstar\rangle$. 
Choose a sequence $(y_n)$ such that the sets defined recursively by
$A_0=A$ and $A_{n} = R_{y_n}A_{n-1}$ for $n\ge 1$
satisfy \[m(A_n\symdif \astar)\to 0.\]
Define $B_n$ recursively by 
$B_0=B$ and $B_{n} = R_{y_n}B_{n-1}$ for $n\ge 1$.
Then 
$\langle \one_{A_n}*\one_C,\one_{B_n}\rangle = \langle \one_\astar*\one_C,\one_\bstar\rangle$ 
for every $n$.
Choose $n_\nu$ so that the sequence $\one_{B_{n_\nu}}$ converges
weakly in $\lt(\torus)$ to some $h\in\lt(\torus)$, with $0\le h\le 1$, $\int h\,dm=m(B)$.
Denoting $(A_{\nu_n},B_{\nu_n})$ by $(A_n,B_n)$ for simplicity, the above implies
\[\langle \one_{A_n}*\one_C,\one_{B_n}\rangle \to \langle \one_{\astar}*\one_C,h\rangle.\]
From this and the admissibility hypothesis it follows that $h=\one_{\bstar}$.
Thus $\one_{B_n}\to\one_\bstar$ weakly.
Since $m(\torus)$ is finite, this forces
$$m(B_n\symdif\bstar)\to 0. 
$$
By (a), $A_n^\star=A^\star$ and $B_n^\star=B^*$ for all $n\in\mathbb{N}$; therefore, for all $\epsilon>0$ there exists $N=N(\epsilon)\in\mathbb{N}$ for which
$$m(A_n\symdif A_n^\star)<\epsilon\text{ and }m(B_n\symdif B_n^\star)<\epsilon,
$$
while also
$$\langle \one_{A_N}*\one_C,\one_{B_N}\rangle =\langle \one_{A_N^*}*\one_{C^*},\one_{B_N^*} \rangle
$$
by (c). Therefore, fixing $\eps$ to be sufficiently small
as a function of $\eta$ alone,
then the perturbative theory of Lemma \ref{lemma:perturbative} can be applied, implying that
\begin{equation}\label{eq:good N} \text{there exists $y_N$ such that $A_N=\astar+y_N$ and $B_N=\bstar+y_N$.}
\end{equation}

Denote by $\scriptr:\torus\to\torus$ the reflection
$\scriptr(x_1,x_2)=(x_1,-x_2)$. Consider any measurable $A,B\subset \torus$ such that the triple $(A,B,C)$ (for our fixed $C$) satisifes the hypotheses of the proposition. 

\begin{claim}  If $A^\sharp=\astar$ and $B^\sharp=\bstar$,
and if $\langle \one_A*\one_C,\one_B\rangle
=\langle \one_\astar*\one_C,\one_\bstar\rangle$,
then either $(A,B) = (\astar,\bstar)$
or $(A,B) = (\scriptr\astar,\scriptr\bstar)$.
\end{claim}

The strict admissibility hypothesis guarantees that
there exists $\eps>0$ such that for any $x\in \astar$
there exists $y=b(x)\in\bstar$ such that whenever
$x',y'\in\torus$ satisfy
$\norm{x-x'}_\torus\le\eps$ and $\norm{y-y'}_\torus\le\eps$,
one has
\begin{equation}
\norm{x_+'-y_+'}_\torus<\tfrac12m(C) \ \text{ but } \ \norm{x_+'-y_-'}_\torus>\tfrac12m(C).
\end{equation}
Indeed, identifying $\torus$ with $[-\frac{1}{2},\frac{1}{2})$, it suffices to prove the above for $x\in A^\star$ with $x\geq 0$. The $\eta$-strict admissibility and $\eta$-boundedness of $(A,B,C)$ ensure that $\bar{p}:=\tfrac{1}{2}m(A)-\tfrac{1}{2}m(C)$ has the property
$$-\tfrac{1}{2}m(B)+\tfrac{\eta^2}{2}\leq \bar{p} \leq \tfrac{1}{2}m(B)-\tfrac{\eta^2}{2}
$$
(in particular $\bar{p} \in B^*$), while the right endpoint $\tfrac{1}{2}m(A)$ of $A^*$ satisfies
$$\big|\tfrac{1}{2}m(A)-\bar{p}_-\big|_{\torus}=\big|\tfrac{1}{2}m(A)-\big(\tfrac{1}{2}-\tfrac{1}{2}m(A)+\tfrac{1}{2}m(C)\big)\big|=\tfrac{1}{2}-m(A)\geq\tfrac{1}{2}m(C),
$$
as $m(A)\leq \tfrac{1}{2}$.

Therefore, if $\bar{p}\leq 0$, define $b(x):=\bar{p}-\tfrac{\eta^2}{4}$ for all $x\in A^\star$. If $\bar{x}>0$, define $b\left(t\tfrac{m(A)}{2}\right):=\bar{p}-\tfrac{\eta^2}{4}$ for every element $t\tfrac{m(A)}{2}$ of $A^\star$, for all $0<t\leq1$.

Denoting by $N_x$ the $\epsilon$-neighbourhood on $\torus$ of any $x\in \torus$, it follows by the above and (d) that, for any $x\in A^\star$,
\begin{eqnarray}
\begin{aligned}
\label{eq:one side}
\text{either }&m\left(\{x'\in  A^\star\cap N_x: x_+'\notin A, x_-'\in A\}\right)=0\\
\text{or }&m\left(\{y'\in  B^\star\cap N_{b(x)}: y_+'\in B, y_-'\notin B\}\right)=0
\end{aligned}
\end{eqnarray}
and
\begin{eqnarray}
\begin{aligned}
\label{eq:other side}
\text{either }&m\left(\{x'\in  A^\star\cap N_x: x_+'\in A, x_-'\notin A\}\right)=0\\
\text{or }&m\left(\{y'\in  B^\star\cap N_{b(x)}: y_+'\notin B, y_-'\in B\}\right)=0.
\end{aligned}
\end{eqnarray}
The second conclusion of \eqref{eq:one side} and the second conclusion of \eqref{eq:other side} cannot simultaneously hold, therefore the first conclusion of either \eqref{eq:one side} or \eqref{eq:other side} holds. That is,
$$\text{for every }x\in A^\star\text{, either }m(\scriptr A\cap N_x)=0\text{ or }m(A\cap N_x)=0.
$$
Now assume that, for some $x\in A^\star$ with $N_x\subset A^\star$, it holds that $m(\scriptr A\cap N_x)=0$. It will be shown that 
$$m(\scriptr A\cap N_y)=0\text{ for all }y\in A^\star\text{ with }\|x-y\|_\torus<\epsilon\text{ and }N_y\subset A^\star
$$
(and therefore, by the connectivity of $A^\star$, $m(\scriptr A\cap A^\star)=0$, i.e. $A=A^\star$ up to a Lebesgue null set). 

Indeed, let $y\in A^\star$ as above, and suppose that $m(\scriptr A\cap N_y)>0$. Due to the fact that $A^\sharp=A^\star$, it holds that $m(\scriptr A\cap N_z)+m(A\cap N_z)=m(A^\star\cap N_z)$ for every $z\in A^\star$. Therefore, $m(\scriptr A\cap N_y)=m(A\cap N_x)=\epsilon$. Since the sets $\scriptr A$ and $A$ share at most two points (as $m(A)\leq \tfrac{1}{2}$), it follows that
$$m\big((\scriptr A\cap N_y)\cup (A\cap N_x)\big)=2\epsilon.
$$
This is a contradiction, as the set $(\scriptr A\cap N_y)\cup (A\cap N_x)$ is contained in the arc $N:=N_y\cup N_x$ of $\torus$, of length $<\tfrac{\epsilon}{2}+\epsilon+\tfrac{\epsilon}{2}<2\epsilon$. 

Therefore, if $x$ as above exists, then $A=A^\star$ up to a Lebesgue null set. In a similar manner it can be shown that if there exists $x\in A^\star$ with $m(A\cap N_x)=0$ and $N_x\subset A^\star$, then $\scriptr(A)=A^\star$ up to a Lebesgue null set.

Thus, either $A=A^*$ or $A=\scriptr A^\star$ up to a Lebesgue null set. Without loss of generality, it is assumed that the former holds (the functional $\langle \one_A*\one_C,\one_B\rangle$ is invariant under simultaneous translations of $A$ and $B$). Then, the fact that $\scriptd(A,B,C)=0$ means that
$$\langle\one_{A^\star}*\one_{C},\one_{-B}\rangle =\langle\one_{A^\star}*\one_{C},\one_{-B^\star}\rangle;
$$
since $A^\star,C$ are intervals, the above implies that $|B\symdif B^\star|=0$.

It has thus been shown that either $(A,B)=(A^\star,B^\star)$ or $(A,B)=(\scriptr A^\star,\scriptr B^\star)$ (up to Lebesgue null sets). This completes the proof of the claim. \qed

We are now in a position to complete the proof for Step 1.
Return to the sequence of pairs $(A_n,B_n)$, for $n=1,\ldots,N$.
By \eqref{eq:good N}, $(A_N,B_N)=(\astar+y_N,\bstar+y_N)$
up to Lebesgue null sets.
Now, $(A_N,B_N) =((A_{N-1}-y)^\sharp,(B_{N-1}-y)^\sharp)$ for some $y\in \mathbb{T}$.
Therefore,
either $(A_{N-1},B_{N-1})$ equals either $(\astar+y_N+y,\bstar+y_N+y)$
or $(\scriptr\astar+y_N+y,\scriptr\bstar+y_N+y)$, up to Lebesgue null sets.
Repeating this reasoning recursively for $n=N-2,N-3,\dots$,
we get a similar conclusion for $(A,B)$.

This completes the discussion of Step 1.


\medskip
\noindent{\bf Step 2.}\   
\textit{For any $\eps>0$, there exists $\delta=\delta(\eta)>0$
such that if $\scriptd(A,B,C)^{1/2}\le\delta$ then}
\begin{equation} \label{notnegation}
\inf_{x+y=x_C}\big(|A\symdif (A^\star+x)| +|B\symdif (B^\star+y)|\big) \le \eps.  
\end{equation}

We argue by contradiction. 
If the conclusion fails to hold then there exists
an $\eta$--strictly admissible $\eta$--bounded sequence $(A_k,B_k,C_k)$ 
such that $\lim_{k\to\infty} \scriptd(A_k,B_k,C_k) = 0$, but
\begin{equation} \label{negation}
\inf_{x+y=x_{C_k}}\big(|A_k\symdif (A_k^\star+x)| +|B_k\symdif (B_k^\star+y)|\big) \ge \eps.
\end{equation}
It may be assumed without loss of generality that each $x_{C_k}=0$.

By passing to subsequences 
we may assume that $\one_{A_k}, \one_{B_k}$ converge 
weakly in $\lt(\torus)$ to $f,g\in L^2(\torus)$, respectively.
Then $0\le f,g\le 1$, $\lim_{k\to\infty} |A_k|$ exists and is equal to $\int_\torus f$,
and likewise $|B_k|\to \int_\torus g$.
Moreover, after a further diagonal argument,
$\one_{C_k}\to \one_C$ weakly for some interval $C$ centered at 0.

Because the $C_k$ are intervals, a simple compactness argument shows that
$\one_{A_k}*\one_{C_k}$ converges strongly in $\lt(\torus)$. Therefore
$\lim_{k\to\infty} \scriptt(A_k,B_k,C_k) =  \scriptt(\astar,\bstar,C)$
where $\astar,\bstar$ denote here the intervals centered at $0$
of lengths $\int_\torus f,\int_\torus g$, respectively.
By continuity, the limiting triple $(\astar,\bstar,C)$ satisfies $\scriptd(\astar,\bstar,C)=0$. 

By Lemma~\ref{hardy-littlewood type},
\[\int_C f*g \le \int_C f*\one_\bstar
\le \langle f^\star,\one_\bstar*\one_C\rangle.\]
The triple $(\int_\torus f^\star,|\bstar|,|C|)$ is $\eta$--strictly admissible.
Because $0\le f^\star\le 1$ and $\int_\torus f^\star=|\astar|$,
$\eta$--strict admissibility ensures that 
$\one_{\bstar}*\one_C$, which is is symmetric  and nonincreasing, 
is also  strictly decreasing with derivative identically
equal to $-1$ in $\{x: |\,x-|\astar|/2\,|\le r\}$ for some $r>0$ which depends only on $\eta$.
Therefore
$\langle f^\star,\one_\bstar*\one_C\rangle =\langle \one_\astar,\one_\bstar*\one_C\rangle$
if and only if $f^\star=\one_\astar$ almost everywhere.
Thus $f^\star$ is the indicator function of a set.
Since $f$ has the same distribution function as $f^\star$,
we conclude that $f=\one_A$ for some $A\subset\torus$.
Likewise, $g=\one_B$ for some set $B$.

Thus $\one_{A_k}\to\one_A$ and $\one_{B_k}\to\one_B$ weakly in $\lt(\torus)$.
Therefore $|A_k\symdif A| + |B_k\symdif B|\to 0$,
and $\scriptd(A,B,C)=0$. Step 1 now applies, allowing us to conclude that
$A$ and $B$ differ from intervals
by Lebesgue null sets, and that the centers of $A,$ satisfy $x_A+x_B=0$. 
This contradicts \eqref{negation}, completing Step 2.
\qed

\medskip
\noindent{\bf Step 3.} 
Let $(A,B,C)$ be a triple satisfying the hypotheses of the proposition. Let $\delta_0$ be the constant appearing in the statement of the perturbative Lemma \ref{lemma:perturbative}. By Step 2, there exists $\delta=\delta(\eta)>0$,
such that if $\mathcal{D}(A,B,C)^{1/2}\leq\delta\max(|A|,|B|,|C|)$ then 
$$\inf_{x+y=x_C} \big(|A\symdif (A^\star+x)|+ |B\symdif (B^\star+y)| \leq \delta_0\eta\leq \delta_0 \max(|A|,|B|,|C|),
$$
where the last inequality is due to the $\eta$-boundedness of $(A,B,C)$. 
Therefore, by Lemma~\ref{lemma:perturbative}, there exist $x',y',z'\in\mathbb{T}$ with $x'+y'=z'$ such that
\begin{equation}\label{eq:almost characterisation of sp case}|A\symdif (A^\star+x')|, |B\symdif (B^\star+y')|, |C\symdif (C^\star +z')|\leq \bC \mathcal{D}(A,B,C)^{1/2},
\end{equation}
for some $\bC>0$ depending only on $\eta$. This would be the desired result if $z'=x_C$, something which does not necessarily follow from Lemma \ref{lemma:perturbative}. However, it can be proved that $z'$ is very close to $x_C$; so close that, perturbing $z'$ to become $x_C$ and perturbing $x'$ by the same amount, the truth of \eqref{eq:almost characterisation of sp case} is not violated, up to multiplication by constant factors.

More precisely, it holds that $\|z'-x_C\|_\mathbb{T}\leq \bC\mathcal{D}(A,B,C)^{1/2}$. Indeed, first observe that $C\cap (C^\star+z')\neq \emptyset$, as otherwise \eqref{eq:almost characterisation of sp case} would imply
$$\mathcal{D}(A,B,C)^{1/2}\geq\tfrac{1}{\bC}|C\symdif (C^\star+z')|=\tfrac{2}{\bC}|C|\geq \tfrac{2\eta}{\bC}\max(|A|,|B|,|C|)
$$
by the $\eta$-strict admissibility of $(A,B,C)$, a contradiction for $\delta$ sufficiently small. Thus, since $C,C^\star+z'$ are intervals centered at $x_C,z'$, respectively, it holds that $\|z'-x_C\|_\mathbb{T}=\tfrac{1}{2}|C\symdif (C^\star+z')|\leq \bC\mathcal{D}(A,B,C)^{1/2}$.

Therefore, $\bar{x}':=x'+(x_C-z')$ satisfies $\bar{x}'+y'=x_C$ and
\begin{eqnarray*}
\begin{aligned}\big|A\symdif (A^\star+\bar{x}')\big|&\leq \big|A\symdif (A^\star+x')\big|+\big| (A^\star+x')\symdif (A^\star+\bar{x}')\big|\\
& \leq\bC\mathcal{D}(A,B,C)^{1/2}+\|x'-\bar{x}'\|_{\mathbb{T}}\\
&\leq 2\bC\mathcal{D}(A,B,C)^{1/2};
\end{aligned}
\end{eqnarray*}
likewise for $B$. Therefore, the triple $(A,B,C)$ satisfies \eqref{eq:special case characterisation} with constant depending only on $\eta$.

As long as the quantity $\delta$ in the argument above
is chosen sufficiently small, the complementary situation in which
$\mathcal{D}(A,B,C)^{1/2}>\delta \max(|A|,|B|,|C|)$ also 
leads to \eqref{eq:special case characterisation} with constant $\bC=2\delta^{-1}$,
simply because, for all $x\in\mathbb{T}$,
$$|A\symdif (A^\star+x)|\leq 2|A|\le 2\max(|A|,|B|,|C|).  $$
Likewise for $B$.
\end{proof}

\section{When one set is nearly rank one Bohr} \label{section:When}

The aim of this section is to establish for general groups $G$ that 
if $(A,B,C)$ is a strictly admissible triple 
with $\scriptd(A,B,C)$ small, if $(A,B,C)$ satisfies appropriate auxiliary hypotheses, 
and if one of the three sets $A,B,C$ is nearly a rank one Bohr set, 
then the other two are also nearly rank one Bohr sets (parallel to the first,
with the triple compatibly centered). 

\begin{proposition}  \label{prop:oneofthree}
Let $G$ be a compact connected Abelian topological group with normalized Haar measure $\mu$.
For any $\eta,\eta'>0$, there exist $c=c(\eta,\eta')>0$ and $\bC=\bC(\eta,\eta',c)<\infty$ such that the following holds.
Let $(A,B,C)$ be an $\eta$-strictly admissible triple of $\mu$-measurable subsets of $G$, 
with $\min(\mu(A),\mu(B),\mu(C))\geq \eta$ and $\mu(A)+\mu(B)+\mu(C)\leq 2-\eta'$. If there exists a rank one Bohr set $\mathcal{B}$ with
$$\mu(C\symdif \mathcal{B})\leq c(\eta,\eta')\max\big(\mu(A),\mu(B),\mu(C)\big),
$$
then there exists a compatibly centered parallel ordered triple $(\mathcal{B}_A,\mathcal{B}_B,\mathcal{B}_C)$ of rank one Bohr sets satisfying
\begin{equation}
\mu (A\symdif \mathcal{B}_A)
\leq \bC\mathcal{D}(A,B,C)^{1/2},
\end{equation}
and likewise for 
$\mu (B\symdif \mathcal{B}_B)$
and $\mu (C\symdif \mathcal{B}_C)$.
\end{proposition}

\begin{proof} Let $\eta,\eta'>0$ and $(A,B,C)$ be as in the statement of the proposition. We may assume that $(A,B,C)$ satisfy the supplementary hypothesis
\begin{equation}\label{eq:smallness of D}\mathcal{D}(A,B,C)<c(\eta,\eta')\max\big(\mu(A),\mu(B),\mu(C)\big)^2
\end{equation}
for a small constant $c(\eta,\eta')$. Indeed, otherwise
\begin{equation*}
\mu(A\symdif\mathcal{B}_A)
\leq \bC(\eta,\eta')\mathcal{D}(A,B,C)^{1/2}
\end{equation*}
holds trivially for any rank one Bohr set $\mathcal{B}_A$ with $\mu(\mathcal{B}_A)=\mu(A)$; likewise for $B$ and $C$.

First, consider the case in which $C$ is a rank one Bohr set. That is,  $C=\phi^{-1}(C^\star)+x$, for some continuous homomorphism $\phi:G\rightarrow \mathbb{T}$ and some $x\in G$. We assume
without loss of generality that $C=\phi^{-1}(C^\star)$.
Define $\phi_*: L^1(G)\to L^1(\torus)$
by 
\[ \int_E \phi_*(f)\,dm = \int_{\phi^{-1}(E)} f\,d\mu
\ \text{ for all measurable $E\subset\torus$.} \]
Then
$$\phi_*(\one_A*\one_B)=\phi_*(\one_A)*\phi_*(\one_B), $$
and consequently
$$\int_C\one_A*\one_B\,d\mu=\scriptt_G(\one_A,\one_B,\one_C) = \scriptt_{\mathbb{T}}(\phi_*(\one_A),\phi_*(\one_B), \one_{C^\star})
= \scriptt_\torus(f,g,\one_{\cstar}),
$$
where the functions
$$f:=\phi_*(\one_A)\text{ and }g:=\phi_*(\one_B)
$$
from $G$ to $[0,\infty)$ satisfy
$$0\leq f,g\leq 1, \; \textstyle\int_\torus f\,dm=\mu(A), \; \int_\torus g\,dm =\mu(B).  $$
Thus, by the Riesz-Sobolev inequality on $\mathbb{T}$,
$$\scriptt_G(\one_A,\one_B,\one_C)
 = \scriptt_\torus(f,g,\one_\cstar)
\leq\scriptt_{\mathbb{T}}(f^\star ,g^\star , \one_{C^\star}).
$$


Applying Lemma~\ref{hardy-littlewood type}
to the functions $f^\star, g^\star, \one_{\cstar}$ gives
\begin{eqnarray}
\begin{aligned}
\label{eq:functional gets larger}
\scriptt_G(\one_A,\one_B,\one_C) 
&\leq \scriptt_\mathbb{T}(f^\star,g^\star,\one_{\cstar})\\
&\leq \max\{\scriptt_\mathbb{T}(f^\star, \one_{B^\star},\one_{C^\star}), 
\scriptt_\mathbb{T}(\one_{A^\star}, g^\star,\one_{C^\star})\}\\
&\leq \scriptt_\mathbb{T}(\one_{A^\star},\one_{B^\star},\one_{C^\star}).
\end{aligned}
\end{eqnarray}
Moreover, since $f^\star$, $g^\star$ are non-increasing functions with $0\leq f^\star,g^\star\leq 1$, 
$\int_\torus f\,dm=m(A^\star)$ and $\int_\torus g\,dm=m(B^\star)$, the following holds.

\begin{claim} \label{eq:control on norms} 
There exists $\boldC<\infty$, depending only on $\eta$, such that
\begin{equation} \|f^\star-\one_{A^\star}\|_{L^1(\torus)}+\|g^\star-\one_{B^\star}\|_{L^1(\torus)}
\leq \boldC  \scriptd(A,B,C)^{1/2}. 
\end{equation}
\end{claim}


\begin{proof} By \eqref{eq:functional gets larger}
and because $\int_\torus (\one_\astar-f^\star)\,dm=0$,
\begin{align*}
\mathcal{D}(A,B,C)
\geq \int_\torus (\one_{A^\star}-f^\star)\cdot  (\one_{B^\star}*\one_{C^\star}) \,dm
= \int_\torus (\one_{A^\star}-f^\star)\cdot  (\one_{B^\star}*\one_{C^\star} -\gamma) \,dm
\end{align*}
for any constant $\gamma$, and in particular for $\gamma = \one_{\bstar}*\one_\cstar\left(\tfrac{\mu(A)}{2}\right)$.
The function
$K(x)= \one_{B^\star}*\one_{C^\star} -\gamma$ 
is nonnegative on $\astar$ and nonpositive on $\torus\setminus\astar$,
as is $\one_\astar-f^\star$, so
\begin{equation} \label{eq:Kbound} \scriptd(A,B,C) 
\ge \int_\torus |\one_\astar-f^\star|\cdot |K|\,dm.
\end{equation}


Let $a=\mu(A)/2$. 
Obtaining a lower bound for the right-hand side would be simpler
if $|K|$ enjoyed a strictly positive lower bound, but
$K(a)=0$. 
$K$ does satisfy $|K(x)|=|x-a|$
for $x\in[0,\tfrac12]$ with 
$|x-a|\le \tfrac12 \min(\mu(B)+\mu(C)-\mu(A),\, \mu(A)-|\mu(B)-\mu(C)|)$,
and the $\eta$--strict admissibility hypothesis ensures that
this holds whenever $|x-a| \le \tfrac12\eta\max(\mu(A),\mu(B),\mu(C))$.
Since $\one_\bstar*\one_\cstar$ is nonincreasing, we find that,
for $x\in[0,\tfrac12]$,
\[ |K(x)| \ge \left\{ 
\begin{aligned}
&|x-a|  \qquad \text{if $|x-a|\le \tfrac{\eta}2 \max(\mu(A),\mu(B),\mu(C))$} \\
&\tfrac{\eta}2 \max(\mu(A),\mu(B),\mu(C)) \qquad\text{ otherwise}.
\end{aligned} \right.  \]

It is elementary that if $0\le\psi\le 1$ then
$\int_\reals |x| \psi(x)\,dx \ge \tfrac{1}{4}\norm{\psi}_{L^1(\mathbb{R})}^2$.
Therefore from the lower bound for $K$
and the upper bound $\norm{\one_\astar-f^\star}_{C^0} \le 1$
it follows that
\[
\int_\torus |\one_\astar-f^\star|\cdot |K|\,dm
\ge 
c\min\Big(\norm{\one_\astar-f^\star}_{L^1(\torus)}, \ \eta\max(\mu(A),\mu(B),\mu(C)) \Big)
\cdot
\norm{\one_\astar-f^\star}_{L^1(\torus)}
\]
for a certain absolute constant $c>0$.
Now $\norm{\one_\astar-f^\star}_{L^1(\torus)}\le 2\mu(A)$, so, provided that $\eta\le 1$,
this implies that
\[ 
\int_\torus |\one_\astar-f^\star|\cdot |K|\,dm
\ge 
c\norm{\one_\astar-f^\star}_{L^1(\torus)}^2,\]
for a constant $c>0$ that only depends on $\eta$.
The indicated conclusion for $\one_\astar-f^\star$ follows directly
from this and  \eqref{eq:Kbound}. 
The same holds for $\one_\bstar-g^\star$ since the roles of $A,B$ can be interchanged.
\end{proof}


Since $f,f^\star$ have identical distribution functions and likewise for $g,g^\star$, there exist
$\tilde A,\tilde B\subset \torus$ satisfying
$\norm{f-\one_{\tilde A}}_{L^1(\torus)}
= \|f^\star-\one_{A^\star}\|_{L^1(\torus)}$
and $\norm{g-\one_{\tilde B}}_{L^1(\torus)}=
\|g^\star-\one_{B^\star}\|_{L^1(\torus)}$,
with $m(\tilde A) = \mu(A)=\int_\torus f\,dm$
and $m(\tilde B) = \mu(B)=\int_\torus g\,dm$.

Therefore, if $c(\eta,\eta')$ is sufficiently small, the triple $(\tilde{A},\tilde{B},C^*)$ 
is $\tfrac{\eta}{2}$-strictly admissible, $\min(\eta,\eta')$-bounded and $m(\tilde{A})+m(\tilde{B})+m(C^*)\leq 2-\tfrac{\eta'}{2}$.
Since $C^*$ is an interval, Proposition~\ref{prop:longname} states that there exists $\bar x\in\torus$ satisfying 
\begin{equation} \label{tildedsetsbound}
m(\tilde{A}\symdif (\astar +\bar x))
+  m(\tilde{B}\symdif (\bstar -\bar x))
\leq \bC\mathcal{D}(\tilde{A},\tilde{B},C^*)^{1/2},
\end{equation}
for a constant $\bC$ depending only on $\eta,\eta'$.
Now 
\begin{equation}
\label{eq:OK without perturbative}
\mathcal{D}(\tilde{A},\tilde{B},C^\star)
\le \boldC \scriptd(A,B,C)^{1/2}\max(m(A),m(B),m(C)).
\end{equation}
Indeed, since $m(\tilde A)=m(A)$ and $m(\tilde B)=m(B)$, it follows that $\tilde{A}^\star=A^\star$ and $\tilde{B}^\star=B^\star$, so
$$\scriptt_\mathbb{T}(\tilde{A}^\star,\tilde{B}^\star,C^\star)=\scriptt_\mathbb{T}(A^\star,B^\star,C^\star),
$$
while
\begin{eqnarray*}
\begin{aligned}
\scriptt_\mathbb{T}(\tilde{A},\tilde{B},C^\star)
&=\scriptt_\mathbb{T}\big(f+(\one_{\tilde{A}}-f),g+(\one_{\tilde{B}}-g),\one_{C^\star}\big) \\
&\geq \scriptt_\mathbb{T}(f,g,\one_{C^\star})\\
& -(\|\one_{\tilde{A}}-f\|_{L^1(\torus)}+\|\one_{\tilde{B}}-g\|_{L^1(\torus)})m(C^\star)
+\|\one_{\tilde{A}}-f\|_{L^1(\torus)}\|\one_{\tilde{B}}-g\|_{L^1(\torus)}\\
&\geq \scriptt_G(A,B,C)
-\boldC \scriptd(A,B,C)^{1/2}\max(m(A),m(B),m(C))
\end{aligned}
\end{eqnarray*}
by Claim \ref{eq:control on norms}. Thus, \eqref{eq:OK without perturbative} follows by \eqref{eq:smallness of D}.

The homomorphism $\phi$ preserves measure in the sense that
$\mu(\phi^{-1}(E))=m(E)$ for any measurable $E\subset\torus$.
Therefore, since $f=\phi_*(\one_A)$,
\begin{equation}
\mu\big(A\symdif \phi^{-1}(\tilde A)\big) = 
\norm{f-\one_{\tilde A}}_{L^1(\torus)}\le \boldC\scriptd(A,B,C)^{1/2}.
\end{equation}
Moreover, 
\eqref{tildedsetsbound} and \eqref{eq:OK without perturbative} together with this
property of $\phi$ yield
$$
 \mu\big(\phi^{-1}(\tilde A) \symdif \phi^{-1}(A^\star+ \bar x)\big)
\le \boldC \scriptd(A,B,C)^{1/4}\max(\mu(A),\mu(B),\mu(C))^{1/2}
$$
In all,
\begin{eqnarray*}
\begin{aligned}
\mu(A\symdif \scriptb_A) &\le \boldC \scriptd(A,B,C)^{1/4}\max(\mu(A),\mu(B),\mu(C))^{1/2}\\
&\leq \bC c(\eta,\eta')^{1/4}\max(\mu(A),\mu(B),\mu(C))
\end{aligned}
\end{eqnarray*}
with $\scriptb_A = \phi^{-1}(\astar)+x$ for some $x\in G$, and likewise for $B$, with $x$ replaced by $-x$. The last inequality above is due to \eqref{eq:smallness of D}, and it ensures that, as long as $c(\eta,\eta')$ is sufficiently small, the perturbative Lemma \ref{lemma:perturbative} can be applied, yielding the desired conclusion for $(A,B,C)$. The analysis of the case in which the set $C$ coincides with a rank one Bohr set is now complete.

Suppose next that
$$\mu(C\symdif \bar{C}\big)\leq c(\eta,\eta')\max\big(\mu(A),\mu(B),\mu(C)\big),
$$
where $\bar{C}=\phi^{-1}(C^*)$ for some continuous homomorphism $\phi:G\rightarrow \mathbb{T}$.

If $c(\eta,\eta')$ is sufficiently small, then the triple $(A,B,\bar{C})$ is $\tfrac{\eta}{2}$-strictly 
admissible and satisfies $\mu(A)+\mu(B)+\mu(\bar{C})\leq 2-\tfrac{\eta}{2}$, while, by \eqref{eq:smallness of D},
\begin{eqnarray*}
\begin{aligned}
\mathcal{D}(A,B,\bar{C})&\leq \boldC c(\eta,\eta')\max\big(\mu(A),\mu(B),\mu(C)\big)^2\\
&\leq \boldC c(\eta,\eta')
\big(1+c(\eta,\eta')\big)^2\max\big(\mu(A),\mu(B),\mu(\bar{C})\big)^2.
\end{aligned}
\end{eqnarray*}
Therefore, since $\bar{C}$ is a rank one Bohr subset of $G$, if $c(\eta,\eta')$ is sufficiently small
then the partial result proved above can be applied to $(A,B,\bar{C})$, ensuring that 
there  exists a compatibly centered parallel ordered triple $(\mathcal{B}_A,\mathcal{B}_B,\mathcal{B}_{\bar{C}})$ of rank one Bohr sets, such that 
\begin{align*}
\mu(A\symdif \mathcal{B}_A)
&\leq \boldC(\eta,\eta')\mathcal{D}(A,B,\bar{C})^{1/2}
\\&
\le \boldC(\eta,\eta') c(\eta,\eta') \max\big(\mu(A),\mu(B),\mu(C)\big),
\end{align*}
and likewise for $B$ and $\bar{C}$. Now, this further implies that
\begin{eqnarray*}
\begin{aligned}
\mu(C\symdif \mathcal{B}_{\bar{C}})&\leq \mu(C\symdif \bar{C})+\mu(\bar{C}\symdif\mathcal{B}_{\bar{C}})\\
&\leq \bC(\eta,\eta')c(\eta,\eta')\max\big(\mu(A),\mu(B),\mu(C)\big).
\end{aligned}
\end{eqnarray*}

Therefore, if $c(\eta,\eta')$ is sufficiently small then the triple $(A,B,C)$ satisfies the hypotheses of the perturbative Lemma~\ref{lemma:perturbative}, the conclusion of which implies the desired estimate for $(A,B,C)$.

\end{proof}


\section{Stability of the Riesz-Sobolev inequality} \label{section:finishRSstability}


In this section we complete
the proof of Theorem~\ref{thm:RSstability}. This proof consists of five main steps.
Firstly, given $\bE$ with small discrepancy $\scriptd(\bE)$
for the Riesz-Sobolev functional,
an associated triple $\bE'$ is constructed, also with small discrepancy
but with altered Haar measures $\mu(E'_j)$ satisfying a supplementary condition.
Secondly, under this supplementary condition,
small Riesz-Sobolev discrepancy for $\bE'$ implies that
$E'_3$ nearly saturates Kneser's sumset inequality.
Thirdly, the inverse theorems of Griesmer and/or Tao 
imply that any saturator $E'_3$ nearly coincides with a rank one Bohr set.
Fourthly, this conclusion for $E'_3$ implies that the given triple $\bE$
nearly coincides with a parallel compatibly centered triple of rank one Bohr sets,
with $o_{\scriptd(\bE)}(1)$ control.
In the fifth step, this crude bound is refined to $O(\scriptd(\bE)^{1/2})$.
All of the ingredients have been developed in preceding
sections. Here, we link them together.

\begin{proof}[Proof of Theorem~\ref{thm:RSstability}] 
Let $\eta>0$. 
Let $\delta_0>0$ be a sufficiently small positive constant,
which will depend only on $\eta$.
Let $(A,B,C)$ be an $\eta$--strictly admissible $\eta$--bounded ordered triple
of measurable subsets of $G$ 
satisfying
\begin{equation}\label{eq:delta small D}
\mathcal{D}(A,B,C)\leq \delta_0.
\end{equation}
In this discussion, $\Ceta$ will denote positive constants
that depend only on $\eta$, not on $(A,B,C)$.
$\Ceta$ is allowed to change in value from one occurrence to the next.

Assume without loss of generality that $\mu(C)\leq \mu(A)\leq \mu(B)$. 
The proof is organized into three cases, reflecting the analysis in \S\ref{section:towards}.



\medskip
\noindent{\textbf{Case 1:}} 
\textit{$\mu(A)\leq \tfrac{1}{2}$ and $\mu(C)\leq (1-\tfrac{\eta}{50})\mu(B)$.}

In this case, the lower bound assumption 
$\min(\mu(A),\mu(B),\mu(C))\ge\eta$ 
implies that, for $\delta_0$ sufficiently small, 
$(A,B,C)$ satisfies the hypotheses of Lemma~\ref{reducing to equal}. 
Define $\tau$ by $\mu(C)=\mu(A)+\mu(B)-2\tau$, 
and define $C' = S_{A,B'}(\tau)$. 
According to Lemma~\ref{reducing to equal},
there exists a measurable set $B'\subset G$ 
such that the triple $(A,B',C')$
also nearly saturates the Riesz-Sobolev inequality, in the sense that
\begin{equation}
\label{retainsmalldefect}
\mathcal{D}(A,B',C')\leq \eta^{-1} \mathcal{D}(A,B,C)\leq \delta_0 \eta^{-1},
\end{equation}
satisfies the key supplementary condition
\begin{equation}
\mu(A)=\mu(B'),
\end{equation}
and satisfies the technical conditions
\begin{gather*}
(A,B',C')\text{ is $\eta/2$--strictly admissible and $\eta^2/2$--bounded},
\\
\mu(C') \le \mu(A)-4\scriptd(A,B',C')^{1/2}. 
\end{gather*}

Therefore, if $\delta_0$ is sufficiently small then the triple $(A,B',C')$ satisfies the hypotheses of Lemma~\ref{lemma:Staulowerbound}, whose conclusion is that $C'$ nearly coincides with a superlevel set:
\begin{equation}\label{eq:almost equal level sets1}
\mu(C' \symdif S_{A,B'}(\beta))\leq 4\mathcal{D}(A,B',C')^{1/2}\leq 4(\delta_0/\eta)^{1/2}
\end{equation}
with $\beta =\tfrac{1}{2}\big(\mu(A)+\mu(B')-\mu(C')\big)$. 
Moreover, $(A,B',C')$ satisfies the hypotheses of the key
Lemma~\ref{lemma:equalitycase} (in particular, $\mu(A)=\mu(B')$), 
whose conclusion is that the superlevel set
$S_{A,B'}(\beta)$ has small difference set:
\begin{eqnarray*}
\begin{aligned}
\mu\big(S_{A,B'}(\beta)-S_{A,B'}(\beta)\big)&
\le 2\mu(S_{A,B'}(\beta)) + 12\scriptd(A,B',S_{A,B'}(\beta))^{1/2}\\
&\leq 2\mu(S_{A,B'}(\beta))+12 (\delta_0/\eta)^{1/2}.
\end{aligned}
\end{eqnarray*}
So long as $\delta_0$ is appropriately small, $S_{A,B'}(\beta)$ satisfies the hypotheses of 
Corollary~\ref{cor:equalitycase}, whose proof relied on the stability theorems of
Tao \cite{taokneser} and/or Griesmer \cite{griesmer} for Kneser's inequality.
Its conclusion is that
there exists a rank one Bohr set $\mathcal{B}_{\beta}$ satisfying
$\mu(\mathcal{B}_{\beta}\symdif S_{A,B'}(\beta)) \le \Ceta \delta_0$.
Combining this with \eqref{eq:almost equal level sets1} yields
$$\mu(\mathcal{B}_{\beta}\symdif C') \le \Ceta \delta_0.  $$

Therefore for sufficiently small $\delta_0$, 
the triple $(A,B',C')$ satisfies the hypotheses of 
Proposition~\ref{prop:oneofthree},
with parameters that depend only on $\eta$; $C'$ nearly coincides with
a rank one Bohr set, and $\scriptd(A,B',C')$ is small.
The proposition states that $A$ and $B'$ consequently
also nearly coincide with rank one Bohr sets; in particular, 
there exists a rank one Bohr set $\mathcal{B}'_A$ satisfying
$$
\mu(\mathcal{B}'_A\symdif A)\leq \Ceta \mathcal{D}(A,B',C')^{1/2}
\leq \Ceta (\delta_0/\eta)^{1/2}
$$
for some finite constant $\Ceta$. 
The last inequality is \eqref{retainsmalldefect}.

With this control of $A$ we return to the originally given triple $(A,B,C)$.
For sufficiently small $\delta_0$, the $\eta$-strictly admissible, $\eta$-bounded triple $(A,B,C)$ 
satisfies the hypotheses of Proposition~\ref{prop:oneofthree}, since $A$ is now known
to nearly coincide with a rank one Bohr set.
The proposition states that 
there exists a compatibly centered parallel ordered triple 
$(\mathcal{B}_A,\mathcal{B}_B,\mathcal{B}_C)$ of rank one Bohr sets satisfying
$$\mu(A\symdif \mathcal{B}_A)
+\mu(B\symdif \mathcal{B}_B)
+\mu(C\symdif \mathcal{B}_C)
\leq \Ceta\mathcal{D}(A,B,C)^{1/2}.
$$
This completes the proof in Case 1. \qed

\medskip
\noindent{\textbf{Case 2:}} \textit{$\mu(A)\leq \tfrac{1}{2}$ and $\mu(C)> (1-\tfrac{\eta}{50})\mu(B)$.}

In this case, $\eta$--strict admissibility and $\eta$--boundedness 
together with sufficient smallness of $\delta_0$ ensure that
$(A,B,C)$ satisfies the hypotheses of Lemma~\ref{for not strongly admissible}. 
Therefore, with $\tau$ defined by $\mu(C)=\mu(A)+\mu(B)-2\tau$, 
there exist measurable sets $C'\subset C$ and $A'\subset A$ that satisfy
\begin{equation*}
\left\{ \begin{aligned}
&(S_{C',A}(\tau),C',A)\text{ is $\eta/4$--strictly admissible
and $\eta/4$--bounded}
\\
&\mathcal{D}(S_{C',A}(\tau),C',A)\leq 16\mathcal{D}(C,B,A)
\\
&\mu(C')=\mu(A')=\mu(C)-\tfrac{1}{10} \eta\mu(B),
\end{aligned} \right. 
\end{equation*}
while 
\begin{equation*}
\left\{ \begin{aligned}
&(S_{C',A'}(\tau),C',A')\text{ is $\eta/2$--strictly admissible and $\eta/2$--bounded}
\\
&\mathcal{D}(S_{C',A'}(\tau),C',A')\leq 16\mathcal{D}(C,B,A)
\\
&\mu(S_{A',C'}(\tau)) \le (1-\tfrac{\eta/2}{50}) \mu(C').
\end{aligned} \right.
\end{equation*}

The triple $(S_{A',C'}(\tau),C',A')$ falls into Case~1 above, with parameters that 
depend only on $\eta$. Therefore, if $\delta_0$ is sufficiently small then
there exists a rank one Bohr set $\mathcal{B}_{C'}$ satisfying
$$\mu(C'\symdif \mathcal{B}_{C'})\leq \Ceta \mathcal{D}(S_{A',C'}(\tau),A',C')^{1/2}\leq 
\boldC_{\eta} \delta_0^{1/2}.
$$
Setting $F:=S_{C',A}(\tau)$, the $\eta/4$-strict admissibility and $\eta/4$-boundedness of the triple $(F,C',A)$ ensure that, for sufficiently small $\delta_0$, $(F,C',A)$ satisfies the hypotheses of 
Proposition~\ref{prop:oneofthree}.
Therefore there exists a rank one Bohr set $\mathcal{B}_A$ satisfying
$$\mu(\mathcal{B}_A\symdif A)\leq \Ceta \mathcal{D}(F,C',A)^{1/2}\leq \Ceta \delta_0^{1/2}.
$$
By $\eta$--admissibility and $\eta$--boundeness, 
$(A,B,C)$ satisfies the hypotheses of 
Proposition~\ref{prop:oneofthree} provided that $\delta_0$ is sufficiently small.
Therefore there exists a compatibly centered parallel ordered triple 
$(\mathcal{B}_A',\mathcal{B}_B,\mathcal{B}_C)$ of rank one Bohr sets satisfying
$$\mu(A\symdif \mathcal{B}_A')+\mu(B\symdif \mathcal{B}_B)+\mu(C\symdif \mathcal{B}_C)
\leq \Ceta\mathcal{D}(A,B,C)^{1/2}.
$$
\qed

\medskip
\noindent{\textbf{Case 3:}} $\mu(A)>\tfrac{1}{2}$.

As discussed in \S\ref{section:twokey},
the triple $(C,G\setminus A, G\setminus B)$ is $\tfrac{\eta}{4}$-strictly admissible 
and $\tfrac{\eta}{4}$-bounded. 
Moreover, since $\tfrac12 < \mu(A)\le\mu(B)$, 
$\mu(G\setminus A)<\tfrac12$ and $\mu(G\setminus B)<\tfrac12$.
Therefore, $(C,G\setminus A, G\setminus B)$ falls in the range of one of the two cases already analyzed above.
Thus there exists a compatibly centered parallel ordered triple $(\mathcal{B}_C,\mathcal{B}_{G\setminus A},\mathcal{B}_{G\setminus B})$ of rank one Bohr sets satisfying
\begin{eqnarray*}
\begin{aligned}
\mu\big((G\setminus A)\symdif \mathcal{B}_{G\setminus A}\big)
\leq \Ceta\mathcal{D}(C,G\setminus A, G\setminus B)^{1/2}
=\Ceta\mathcal{D}(A,B,C)^{1/2}
\leq \Ceta\delta^{1/2}
\end{aligned}
\end{eqnarray*}
and likewise for
$\mu\big((G\setminus B)\symdif \mathcal{B}_{G\setminus B}\big)$
and for
$\mu(C\symdif \mathcal{B}_C)$.
The equality of 
$\mathcal{D}(C,G\setminus A, G\setminus B)^{1/2}$ with $\mathcal{D}(A,B,C)^{1/2}$
was established in Lemma~\ref{lemma:complementD}.

For any measurable subsets $E_1,E_2$ of $G$, $\mu(E_1\symdif E_2)=\mu\big((G\setminus E_1)\symdif (G\setminus E_2)\big)$. 
Therefore the compatibly centered parallel ordered triple $(\mathcal{B}_A,\mathcal{B}_B,\mathcal{B}_C)$ of rank one Bohr sets with $\mathcal{B}_A:=G\setminus\mathcal{B}_{G\setminus A}$, $\mathcal{B}_B:=G\setminus\mathcal{B}_{G\setminus B}$ satisfies
$$\mu(A\symdif \mathcal{B}_A)+\mu(B\symdif \mathcal{B}_B)+\mu(C\symdif \mathcal{B}_C)
\le \Ceta \scriptd(A,B,C)^{1/2}.
$$
The proof of Theorem~\ref{thm:RSstability} is complete.
\end{proof}


\section{Cases of equality in the Riesz-Sobolev inequality}

Theorem~\ref{thm:RSuniqueness} states that if 
$\scriptt_G(\bE) = \scriptt_\torus(\bE^\star)$,
and if $\bE$ is admissible,
then there exists a $\scriptt_G$--compatibly centered ordered triple of parallel 
rank one Bohr sets satisfying $\mu(E_j\symdif \scriptb_j)=0$
for every $j\in\{1,2,3\}$.
There are two cases. If $\bE$ is strictly admissible,
then there exists $\eta>0$ such that $\bE$ is $\eta$--strictly
admissible and $\eta$--bounded. Therefore $\bE$ satisfies
the hypotheses of Theorem~\ref{thm:RSstability}, the quantitative
stability theorem, with $\delta=0$. That theorem, whose proof
has been completed above, gives the required conclusion.

If $\bE$ is admissible but not strictly admissible,
then after appropriate permutation of the three indices,
$\mu(E_1)+\mu(E_2) = \mu(E_3)<1$,
and \[\langle \one_{E_1}*\one_{E_2},\one_{-E_3}\rangle = \mu(E_1)\mu(E_2)
= \langle \one_{E_1}*\one_{E_2},\one_{G}\rangle.\]
Therefore $\one_{E_1}*\one_{E_2}=0$ $\mu$--almost everywhere
on the complement of $-E_3$, that is,
$E_1+_0 E_2$ is contained in the union of $-E_3$ with a nullset.
Thus $\mu(E_1+_0 E_2) \le \mu(E_3)$. The converse inequality holds
by Kneser's theorem, so $\mu(\symdif(E_1+_0 E_2,\,-E_3)=0$. 
It is a corollary of more quantitative results of 
Griesmer \cite{griesmer} and Tao \cite{taokneser}
that equality of $\mu(E_1+_0 E_2)$ with $\mu(E_1)+\mu(E_2)$
implies existence of a parallel pair of rank one Bohr sets satisfying
$\mu(E_j\symdif \scriptb_j)=0$ for $j=1,2$.
Set $\scriptb_3 = \scriptb_1 +\scriptb_2$.
Then $(\scriptb_1,\scriptb_2,\scriptb_3)$
is an ordered triple of rank one Bohr sets with all required properties.
\qed

\section{Stability in the relaxed framework} \label{section:relax2}

Theorem~\ref{thm:relaxedstability},
a stability theorem for the Riesz-Sobolev inequality
in the situation in which
indicator functions of sets are replaced by functions
taking values in $[0,1]$, follows from slight modification of the 
proof of Theorem~\ref{thm:RSstability}.

\begin{proof}[Proof of Theorem~\ref{thm:relaxedstability}] 
Let $f,g,h$ be as in the statement of the theorem. 
To simplify notation, set 
\[ M = \max\big(\textstyle\int f\,d\mu,\textstyle\int g\,d\mu,\textstyle\int h\,d\mu\big). \]
With the notation of \S\ref{section:When},
\[\langle f*g,h\rangle_G\le \langle f^\star*g^\star,h^\star\rangle_\torus
\le \langle \one_\astar*\one_\bstar,h^\star\rangle_\torus\leq \langle \one_{A^\star}*\one_{B^\star},\one_{C^\star}\rangle_\torus.\]
As shown in \S\ref{section:When}, this implies that
\[ \norm{h^\star-\one_\cstar}_{L^1(\torus)} \le \bC\scriptd^{1/2}.\]
Since $h$ has the same distribution function as $h^\star$,
there exists a set $C\subset G$ satisfying
\begin{equation} \label{smallL1} \norm{h-\one_C}_{L^1(G,\mu)} 
=\norm{h^\star-\one_\cstar}_{L^1(\torus)}
\le \bC\scriptd^{1/2}.\end{equation}
The same reasoning applies to $f$ and to $g$, yielding
corresponding sets $A,B\subset\torus$, respectively.
Now 
\[ \big|\,\langle f*g,h\rangle_G - \langle \one_A*\one_B,\one_C\rangle_G\,\big|
\le \bC M \scriptd^{1/2}, \]
so, for $\scriptd$ sufficiently small as a function of $\eta$ alone, 
\begin{align*} \langle \one_A*\one_B,\one_C\rangle_G
& \ge
\langle \one_\astar*\one_\bstar,\one_\cstar\rangle_\torus
- \bC M \scriptd^{1/2} 
\\
&= \langle \one_\astar*\one_\bstar,\one_\cstar\rangle_\torus
- \bC\max(\mu(A),\mu(B),\mu(C))\scriptd^{1/2} 
\\
& \ge \langle \one_\astar*\one_\bstar,\one_\cstar\rangle_\torus
- \bC M \scriptd^{1/2} 
\end{align*}
with the convention that the constant $\bC\in(0,\infty)$
may change from one occurrence to the next.
In the final line we have used
the fact that $\max(\mu(A),\mu(B),\mu(C))$ is comparable to $M$.

Therefore according to Theorem~\ref{thm:RSstability},
there exists a
compatibly centered parallel triple 
$(\tilde A, \tilde B, \tilde C)$ 
of rank one Bohr subsets of $G$
such that
\[ |A\symdif \tilde A| \le \bC M^{1/2} \scriptd^{1/4},\] 
with the same bound for $|B\symdif \tilde B|$ and $|C\symdif \tilde C|$.
In combination with \eqref{smallL1} and the corresponding
results for $f,g$, this gives
\[ \max\big(\, 
\norm{f-\one_{\tilde A}}_{L^1},\,
\norm{g-\one_{\tilde B}}_{L^1},\,
\norm{h-\one_{\tilde C}}_{L^1}\, 
\big)
\le \bC M^{1/2} \scriptd^{1/4}.
\] 

It is given in the hypotheses of Theorem~\ref{thm:relaxedstability} that
$\scriptd/M^2$ is less than some small absolute constant
that is at our disposal, but no lower bound is given. Therefore
this conclusion is weaker than the desired bound $\bC\scriptd^{1/2}$.
However, any bound of the form $o_{\scriptd/M^2}(1)\cdot M$ is sufficient
to place us in the perturbative context of
Lemma~\ref{lemma:perturbative_relaxed},
which gives the desired bound,
completing the proof of Theorem~\ref{thm:relaxedstability}.
\end{proof}

\section{A flow of subsets of $\torus$} \label{section:flow}

This section and the next 
develop an alternative approach which, 
as it now stands, applies directly only for $G=\torus$, but which
yields slightly superior results for that group;
the bounds remain appropriately uniform as the measures of $A,B,C$
tend to zero.
It is based on monotonicity of the functional $(A,B,C)\mapsto \scriptt_\torus(A,B,C)$
under a certain continuous one-parameter deformation.
Such a monotonicity phenomenon is well-known for $G=\reals$ \cite{christRSult}. 
The variant developed here, which applies to $\torus$, is less effective but nonetheless useful. In the present section we develop the deformation
and its basic properties for Kneser's inequality and the Riesz-Sobolev inequality.
In the following section we apply it to establish an improved stability theorem for $\torus$.

In the present section and in \S\ref{section:finishingtorus},
the Lebesgue measure of a subset $E\subset\torus$ is denoted by $|E|$.
All integrals over $\torus$ are formed with respect to Lebesgue measure.

\medskip
Let $\scriptl(\torus)$ be the class of all 
equivalence classes of Lebesgue measurable sets $E\subset \torus$ with $|E|>0$, and
$E$ equivalent to $E'$ if and only if $|E\symdif E'|=0$.
Assuming that $|E|>0$, define 
\begin{equation} T_E = -\ln(|E|)>0.\end{equation}
In the next theorem, $E$ and $E_j$ denote arbitrary
equivalence classes of Lebesgue measurable subsets of $\torus$.
For equivalence class $A,B$, the notation $A\subset B$
means of course that any two representatives of these classes
satisfy $|B\setminus A|=0$.

Recall that
\begin{equation} \label{zerosumset}
A+_0 B = \{x: \one_A*\one_B(x)>0\}.
\end{equation}
The inequality $|A+B|_* \ge \min(|A|+|B|,1)$ for all measurable $A,B\subset G$
implies that
\begin{equation} \label{ineq:zerosumset} |A+_0 B| \ge \min(|A|+|B|,1)
\ \text{ for all measurable $A,B\subset \mathbb{T}$.} \end{equation}
Indeed, given $A,B\subset \mathbb{T}$, 
denote by $A^\dagger\subset A$ and $B^\dagger\subset B$ 
the sets of Lebesgue points of $A,B$, respectively. From the fact that almost
every point is a Lebesgue point, it follows easily that
$$A^\dagger+_0B^\dagger=A^\dagger+B^\dagger. $$
Therefore, since $A^\dagger\subset A$ and $B^\dagger\subset B$, it follows that
$$|A+_0B|\geq |A^\dagger+_0B^\dagger|=|A^\dagger+B^\dagger|\geq \min(|A^\dagger|+|B^\dagger|,1)=\min(|A|+|B|,1),
$$
establishing \eqref{ineq:zerosumset} for $A,B$.

\begin{theorem}\label{thm:flow}
There exists a flow $(t,E)\mapsto E(t)$ of elements of $\scriptl(\torus)$,
defined for $t\in[0,T_E]$, having the following properties. 
\begin{enumerate}
\item
$E(0)=E$ and $E(T_E) = \torus$.
\item $E(s)\subset E(t)$ whenever $s\le t$.
\item 
$|E(t)| = e^t|E|$ for all $t\in[0,T_E]$.
\item
$|E(s)\symdif E(t)|\to 0$ as $s\to t$.
\item
If $E\subset\tilde E$ then $E(s)\subset \tilde E(s)$ for all $s\in[0,T_{\tilde E}]$.
\item
$e^{-t}|E_1(t)\symdif E_2(t)|\le e^{-s} |E_1(s)\symdif E_2(s)|$ for all $E_1,E_2$
and every $0\le s\le t\le \min(T_{E_1},T_{E_2})$.
\item
If $0\le s\le t\le T_E$ then
$E(t) = (E(s))(t-s)$

\item
If $E$ is the rank one Bohr set
$\{x: \norm{\phi(x)}\le r\}$ associated to a nonconstant homomorphism
$\phi:\torus\to\torus$
then $E(t) = \{x: \norm{\phi(x)}\le e^t r\}$.
\item
$(E+y)(t) = E(t)+y$ for every $E\in\scriptl(\torus)$, $y\in\torus$, and $t\le T_E$.
Likewise, $(-E)(t)=-E(t)$. 
\item
The function $t\mapsto e^{-t}|E_1(t)\,+_0\,E_2(t)|$
is nonincreasing on $[0,\min(T_{E_1},T_{E_2})]$. 
\item 
The function $t\mapsto e^{-2t}\scriptt(E_1(t),E_2(t),E_3(t))$ 
is nondecreasing on $[0,\tau]$
provided that $\tau\le \min_{j\in\{1,2,3\}}(T_{E_j})$
and $\sum_{j=1}^3 |E_j(\tau)|\le 2$.
\end{enumerate}
\end{theorem}

Each conclusion is to be interpreted in terms of equivalence classes of measurable sets. 
Thus, for instance, $A\subset B$ means $|B\setminus A|=0$.

As mentioned earlier, a flow with variants of these properties is known for $\reals$.
See for instance a discussion in \cite{christRSult}.
Such a flow acting on a dense class of sets, 
namely finite unions of intervals, is discussed in \cite{liebloss}.
That it extends to arbitrary sets has been known to experts \cite{burchardcommunication},
though it seems not to have been extensively discussed in the literature.

The flow for $\mathbb{R}$ \cite{christRSult} 
preserves Lebesgue measures, whereas that of Theorem~\ref{thm:flow}
does not. There exists no flow for $\torus$ that mimics all properties of the flow for $\reals$.
Indeed, a rank one Bohr set $E\subset\torus$
is a union of small intervals centered at the elements of a finite cyclic subgroup
$H$ of $\torus$, or at elements of a coset of $H$. 
$E$ satisfies $|E+E|=2|E|$ if $|E|<\tfrac12$, so $E$ realizes
equality in the sumset inequality. There is 
no way to continuously deform one such set $E$ to another, through
sets satisfying $|E(t)+E(t)|=2|E(t)|$ with $|E(t)|$ independent of $t$, 
if the two sets in question are associated
to subgroups $H$ having different numbers of elements.

The flow of Theorem~\ref{thm:flow} lacks another key property
of its analogue for $\reals$, a lack which
may appear to severely limit its utility, although we will show
in the next section that it is nonetheless a valuable tool.
The functionals $e^{-2t}\scriptt(E_1(t),E_2(t),E_3(t))$
and $e^{-t}|E_1(t) +_0 E_2(t)|$
are only defined with desired monotonicity properties for $t\le T$,
for a certain terminal time $T$. 
The defect is that the terminal time sets $E_j(T)$ need not possess any particular
structure (such as $E_j(T) = E_j(T)^\star$ up to translation, or $E_j(T)=\torus$). 
In contrast, the corresponding flow for $\reals$ deforms all three sets to their
symmetrizations $E_j^\star$.

\begin{proof}
The proof is nearly identical in many respects to that of a corresponding
result for $\reals$ proved in \cite{christRSult}, with the exception 
of the conclusion concerning $|E_1(t)+_0 E_2(t)|$. We will provide only a sketch
which deals with those points at which differences arise.

One begins by defining
$t\mapsto E(t)$ in the special case
in which $E$ is a finite union of closed intervals.
One verifies the stated
properties in that case, then uses these properties to show that the flow
extends to $\scriptl(\torus)$ via uniform continuity with respect to the metric
$\rho(E,E') = |E\symdif E'|$.\footnote{The flow of Theorem~\ref{thm:flow} 
acts on equivalence classes of sets. Its restriction
to finite unions of closed intervals agrees with the preliminary flow
defined on finite unions of intervals, up to Lebesgue null sets.}

Let $E=\cup_j I_j$ (a finite union), where $I_j\subset\torus$ is a closed arc 
of length $|I_j|$ with center $c_j$, 
and these closed arcs are pairwise disjoint.
Define $E(t)=\cup_j I_j(t)$,
where $I_j(t)$  is the arc with center $c_j$ and length $e^t|I_j|$,
for all $0\le t\le T_1$, where $T_1$ is the smallest $t$ for which
some pair of arcs $I_i(t),I_j(t)$ intersect. Any two arcs that do intersect
share only an endpoint (or two endpoints, in the case in which the union
has length $1$). Thus $E(T_1)$ may be expressed in a unique way
as a disjoint union of finitely many closed arcs, with certain centers.
The number of such arcs is strictly smaller than the number of arcs
comprising the initial set $E$.
Repeat the first step for this  new collection of  arcs,
stopping at the first time $T_2>T_1$ at which intersection occurs.
Again reorganize $E(T_2)$ as a union of finitely many pairwise
disjoint closed arcs, and repeat until a single arc remains.
This occurs, because the number of arcs is reduced with each
iteration, and it is not possible for the number of arcs to exceed
$1$ if the measure of their union equals $1$.
Continue until $|E(t)|=1$. 

We claim that if $E_j$ is a finite union of $N_j$ pairwise disjoint
closed arcs for each index $j\in\{1,2,3\}$,
and if $\tau>0$ is sufficiently small that $E_j(t)$
is defined for $t\in[0,\tau]$ and is a union of exactly $N_j$
pairwise disjoint closed arcs for every $t\in[0,\tau)$
for each index $j$, then $e^{-2t}\scriptt(\bE(t))$
is a nondecreasing function of $t\in[0,\tau]$.
It suffices to prove this for $t\in[0,T_1]$.

Write $\one_{E_j(t)} = \sum_{n=1}^{N_j}\one_{I_{j,n}(t)}$
with the natural notations. Then $|I_{j,n}(t)| = e^t|I_{j,n}(0)|$
for all indices $j,n$. By linearity of $\scriptt$,
it suffices to show that $t\mapsto e^{-2t}\scriptt(\bI(t))$ is a nondecreasing function
for any triple $\bI(t)=(I_j(t): j\in\{1,2,3\})$
of intervals,
with centers $c_j$ of $I_j(t)$ independent of $t$
and with lengths $|I_j(t)|=e^t|I_j(0)|$.
By translation-invariance, we may assume that $c_1=c_2=0$.
By reflecting about $0$ if necessary, we may assume that the center $\bar{c}_3:=-c_3$ of $-I_3$ satisfies $e^t \bar{c}_3\in[0,\tfrac12]$.

Set $l_j=|I_j(0)|/2$.
Now 
\begin{align*}
\scriptt(\bI(t)) = \iint_{\torus^2}
\one_{\norm{x}\le e^t l_1} \one_{\norm{y}\le e^t l_2} \one_{\norm{x+y-\bar{c}_3}\le e^t l_3}
\,dx\,dy.
\end{align*}

Define $K(x) = \one_{\tilde I_1}*\one_{\tilde I_2}(x)$ for $x\in\reals$,
where $\tilde I_j = [-\tfrac12 l_j,\tfrac12 l_j]\subset\reals$.
Then, since $|I_1(t)|+|I_2(t)|<1$ (as $t<T_1$),
$\scriptt(\bI(t))$ can be expressed as 
\[ \scriptt(\bI(t)) = \int_{\reals}
e^t (K(e^{-t}u) + K(e^{-t}(u-1))) \one_{|u-\bar{c}_3|\le e^tl_3}(u)\,du.  \]
Splitting this as a sum of two integrals and 
substituting $u=e^t x$ in one and $u = e^ty+1$ in the other gives
\begin{align*} 
e^{-2t} \scriptt(\bI(t)) 
= \int_\reals K(x)\one_{-I_3-\bar{c}_3}(x-e^{-t}\bar{c}_3)\,dx
+\int_\reals K(y)\one_{-I_3-\bar{c}_3}(y+e^{-t}(1-\bar{c}_3))\,dy.
\end{align*}

Because $K$ is nonnegative, even, and is nonincreasing on $[0,\infty)$, 
each of the two integrals above represents a nondecreasing function of $t$
for any interval $I_3$. This completes the proof of monotonicity.

The conclusions of Theorem~\ref{thm:flow} now follow in the same way as in \cite{christRSult}, 
with the exception of monotonicity of $e^{-t}|E_1(t) +_0 E_2(t)|$,
which was not discussed there.
Set $E_3 = -(E_1 +_0 E_2)$.
Then $\scriptt(E_1,E_2,E_3) = |E_1||E_2|$.
We have shown above that $t\mapsto e^{-2t} \scriptt(E_1(t),E_2(t),E_3(t))$
is a nondecreasing function of $t$. In particular,
\[ e^{-2t} \scriptt(E_1(t),E_2(t),E_3(t)) \ge
\scriptt(E_1(0),E_2(0),E_3(0)) =
\scriptt(E_1,E_2,E_3) = |E_1|\cdot |E_2|. \]
But 
\[ \int_\torus \one_{E_1(t)}* \one_{E_2(t)}
\le |E_1(t)|\cdot|E_2(t)| = e^{2t}|E_1|\cdot|E_2|.\]
Therefore
\[ 
\int_{-E_3(t)} \one_{E_1(t)}* \one_{E_2(t)}
=\int_\torus \one_{E_1(t)}* \one_{E_2(t)},\]
forcing $\{x: \one_{E_1(t)}*\one_{E_2(t)}(x)>0\}\subset -E_3(t)$
up to a Lebesgue null set. Therefore 
\[ e^{-t}|E_1(t) +_0 E_2(t)| \le e^{-t}|E_3(t)|
= |E_1+_0 E_2|.\]

If $0\le s\le t$ then $E_j(t) = (E_j(s))(t-s)$,
so the general relation
\[ e^{-s}|E_1(s) +_0 E_2(s)| \le e^{-t}|E_1(t) +_0 E_2(t)|\]
follows from the case $s=0$.
\end{proof}

\begin{remark} \label{cor:monotonicity of D}
An equivalent formulation of the monotonicity of $e^{-2t}\scriptt(E_1(t),E_2(t),E_3(t))$
is that
$t\mapsto e^{-2t}\scriptd(E_1(t),E_2(t),E_3(t))$ is nonincreasing on $[0,\tau]$, 
provided that $\tau\leq\min_{j\in\{1,2,3\}}T_{E_j}$ and $\sum_{j=1}^3|E_j(\tau)|\leq 2$.
The monotonicity will be invoked in this form.

Indeed, for $t\in[0,\tau]$, 
\begin{eqnarray*}
\begin{aligned}
e^{-2t}\scriptd(E_1(t),&E_2(t),E_3(t))\\
&=e^{-2t}\scriptt\big(E_1(t)^\star,E_2(t)^\star,(-E_3(t))^\star\big)-e^{-2t}\scriptt\big(E_1(t),E_2(t),-E_3(t)\big)\\
&=e^{-2t}\scriptt\big(E_1^\star(t),E_2^\star(t),(-E_3)^\star(t)\big)-e^{-2t}\scriptt\big(E_1(t),E_2(t),(-E_3)(t)\big)\\
&=\scriptt\big(E_1^\star,E_2^\star,(-E_3)^\star\big)-e^{-2t}\scriptt\big(E_1(t),E_2(t),(-E_3)(t)\big).
\end{aligned}
\end{eqnarray*}
Now $e^{-2t}\scriptt\big(E_1(t),E_2(t),(-E_3)(t)\big)$ is nondecreasing 
by the final conclusion of Theorem~\ref{thm:flow};
its hypotheses are satisfied since
$|(-E_3)(\tau)|=|E_3(\tau)|$ and $T_{-E_3}=T_{E_3}$. 
\end{remark}

The following remark, which will not be used in this paper
but which may nonetheless be of interest, 
also follows in the same way as in \cite{christRSult}.

\begin{proposition} \label{prop:flowtointervals}
Let $E\subset\reals^1$ be a Lebesgue measurable set with finite measure.
For each $t\in(0,T_E]$, $E(t)$ equals a union of intervals, up to a Lebesgue null set.
\end{proposition}
That is, there exists a countable family of pairwise disjoint intervals $I_n(t)$ 
such that $|E(t) \symdif \bigcup_n I_n(t)|=0$.


\medskip
The next lemma  makes it possible
to propagate control of a triple $\bE(t)$ backwards in time, 
with respect to the flow $t \mapsto \bE(t)$, in the analysis of inequality
\eqref{ineq:RS} for $\torus$.

\begin{lemma}[Time reversal] \label{lemma:timereversal}
For each $\eta,\eta'>0$ there exist $\delta_1>0$ and $\bC<\infty$ 
with the following property.
Let $\bE$ be an $\eta$--strictly admissible 
ordered triple of measurable subsets of $\torus$, satisfying $\sum_j |E_j|\le 2-\eta'$ .
Let $0<t\le \min_{1\le j\le 3}T_{E_j}$ with
$e^t-1\le \delta_1$.
Suppose that 
there exists $\by=(y_1,y_2,y_3)\in\torus^3$ satisfying $y_1+y_2=y_3$
such that
\begin{equation} |E_j(t)\symdif (E_j(t)^\star+y_j)| \le \delta_1 \max_j |E_j(t)|
\qquad \forall\,j\in\{1,2,3\}. \end{equation}

Then there exists $\bz = (z_1,z_2,z_3)\in\torus^3$
satisfying $z_1+z_2=z_3$ such that
\begin{equation} \label{timereversed} 
|E_j\symdif (E_j^\star+z_j)| \le \bC\scriptd(\bE)^{1/2} 
\qquad \forall\,j\in\{1,2,3\}.
\end{equation}
\end{lemma}

\begin{proof}
Requiring $\delta_1\le 1$, as we may, yields
\begin{align*}
|E_j\symdif (E_j^\star+y_j)|
&\le
|E_j\symdif E_j(t)|
+ |E_j(t)\symdif (E_j(t)^\star+y_j)|
+ |(E_j(t)^\star+y_j)\symdif (E_j^\star+y_j)|
\\&
\le (e^t-1)|E_j|
+ \delta_1e^t\max_k|E_k|
+ (e^t-1) |E_j|
\\&
\le (2(e^t-1)+\delta_1)\max_k|E_k| 
\\&
=O(\delta_1\max_k|E_k|).\end{align*}
Therefore, if $\delta_1$ is sufficiently small,
then $\bE$ satisfies the hypotheses of Lemma~\ref{lemma:perturbative}.
Its conclusion is the desired inequality \eqref{timereversed}.
\end{proof}


\section{Concluding steps for $\torus$} \label{section:finishingtorus}

In this section we prove the following slight improvement of Theorem~\ref{thm:RSstability}
in the case $G=\torus$.
The improvement lies in the absence of any lower bound for $\min(m(A),m(B),m(C))$.
That no lower bound is needed, is to be expected after the work of Bilu \cite{bilu}
on the sumset inequality.

\begin{theorem} \label{thm:RS_Tstability}
For each $\eta>0$ there exist $\delta_0>0$ and $\bC<\infty$ with the following property.
Let $(A,B,C)$ be an $\eta$--strictly admissible ordered triple
of Lebesgue measurable subsets of $\torus$ satisfying $m(A)+m(B)+m(C)\le 2-\eta$.
Let $\delta\le \delta_0$.
If 
\begin{equation}
\int_C \one_A*\one_B\,dm \ge
\int_\cstar \one_\astar*\one_\bstar\,dm-\delta\max(m(A),m(B),m(C))^2
\end{equation}
then there exists a compatibly centered parallel ordered triple
$(\scriptb_A,\scriptb_B,\scriptb_C)$ of rank one Bohr subsets of $\torus$
satisfying
\begin{equation}
m(A\symdif \scriptb_A) \le \boldC \delta^{1/2} \max(m(A),m(B),m(C))
\end{equation}
and likewise for $(B,\scriptb_B)$ and $(C,\scriptb_C)$.
\end{theorem}

\begin{proof} By Theorem \ref{thm:RSstability}, the desired conclusion holds for all triples $(A,B,C)$ that additionally satisfy $\min(m(A),m(B),m(C))\geq \tfrac{1}{3}\eta^2$.

Now, let $(A,B,C)$ be a triple satisfying the hypotheses of the theorem, but with 
$$\min(m(A),m(B),m(C))<\tfrac{1}{3}\eta^2.
$$
Set $\bE = (E_1,E_2,E_3) = (A,B,C)$ and consider the flowed triples $\bE(t)$
for $0\le t\le T$, with $T$ chosen so that
$$\min_{j=1,2,3} m(E_j(t))=\tfrac{1}{3}\eta^2.
$$
That is, $\tfrac{1}{3}\eta^2=e^Tm$, for $m:=\min_{j=1,2,3}m(E_j)$. 

For all $t\in [0,T]$, the triple $\bE(t)$ is $\eta$-strictly admissible. Setting $M := \max_{j=1,2,3}m(E_j)$, the $\eta$-strict admissibility of $\bE$ ensures that
$$\max_{j=1,2,3}m(E_j(T))=e^TM\leq e^Tm\eta^{-1}=\tfrac{1}{3}\eta,
$$
whence
$$\sum_{j=1}^3m(E_j(t))\leq \eta \leq 2-\eta\text{ for all }t\in [0,T].
$$
Moreover, the assumption $\scriptd(\bE)\leq \delta M^2$ together by the monotonicity of the Riesz-Sobolev functional under the flow (discussed in \S\ref{section:flow}) imply that
$$\scriptd(\bE(t))\leq e^{2t}\scriptd(\bE)\leq e^{2t}\delta M^2=\delta \max_{j=1,2,3}m(E_j(t))^2\text{ for all }t\in [0,T].
$$

The triple $\bE(T)$ enjoys the additional property that it is $\eta^2$-bounded, and therefore satisfies the hypotheses of by Theorem \ref{thm:RSstability} with parameters depending only on $\eta$. It follows that, provided that $\delta_0$ is sufficiently small as a function of $\eta$ alone, there exists a compatibly centered parallel ordered triple $\bB:=(\scriptb_1,\scriptb_2,\scriptb_3)$ of rank one Bohr sets with
$$m\big(\scriptb_j\symdif E_j(T)\big)\leq \bC \delta^{1/2}\max_{j=1,2,3}m(E_j(T)).
$$
Assuming again that $\delta_0$ is sufficiently small as a function of $\eta$, the time reversal Lemma \ref{lemma:timereversal}
can be applied in a straightforward series of reverse time steps
to conclude that there exists a compatibly centered triple
$(\scriptb_1',\scriptb_2',\scriptb_3')$
of rank one Bohr sets such that
\[ m(\scriptb_j'\symdif E_j)\le \bC\scriptd(\bE)^{1/2} \]
for each $j\in\{1,2,3\}$.
\end{proof}


\end{document}